\documentclass[reqno, twoside, 11pt]{article}
\usepackage[latin1]{inputenc}
\usepackage[T1]{fontenc}
\usepackage[english]{babel}
\usepackage{color}
\usepackage{tocloft}
\usepackage{verbatim}

\usepackage{longtable}

\usepackage[
  hmarginratio={1:1},     
  vmarginratio={1:1},     
  textwidth=17cm,        
  textheight=24cm,
  heightrounded          
]{geometry}

\usepackage[noindentafter]{titlesec}
\titlespacing{\paragraph}{%
  0pt}{
0.25\baselineskip}{
1em}

\titlespacing{\section}{%
0pt}{
0.2cm}{
0em}

\titlespacing{\subsection}{%
0pt}{
0cm}{
0em}

\titlespacing{\subsubsection}{%
0pt}{
0cm}{
0em}

\usepackage{amsmath,amsthm,amssymb,accents}
\usepackage{mathrsfs, mathtools,xpatch}
\usepackage{dsfont}

\DeclareMathAlphabet{\mathbbe}{U}{bbold}{m}{n}

\mathtoolsset{showonlyrefs}

\allowdisplaybreaks

\usepackage[all]{xy} 
\input xy 
\xyoption{all}
\xyoption{2cell} 
\xyoption{v2}
\SelectTips{cm}{11} 

\usepackage{tikz}
\usetikzlibrary{snakes}
\usetikzlibrary{arrows}
\usepackage{tikz-cd}

\tikzset{string/.style={decorate, decoration={snake, segment length=3pt, amplitude=1pt}}}

\usepackage[hidelinks]{hyperref}

\usepackage[numbers,sort&compress]{natbib}
\makeatletter
\renewcommand{\@biblabel}[1]{[#1]\hfill}
\makeatother

\usepackage[mathscr]{euscript}
\DeclareMathAlphabet{\pazocal}{OMS}{zplm}{m}{n}

\usepackage[shortlabels]{enumitem}

\newtheoremstyle{thm}                                                           
{0.15cm}                                         
{0.15cm}                                         
{\itshape}      
{}                                      
{\bfseries}                             
{}                                      
{0.2cm}                                         
{\thmname{#1}~\thmnumber{#2}\thmnote{ (#3)}}%

\makeatletter
\xpatchcmd{\proof}{\topsep6\p@\@plus6\p@\relax}{}{}{}
\makeatother

\newtheoremstyle{rmk}                                                           
{0.15cm}                                         
{0.15cm}                                         
{}      
{}                                      
{\bfseries}                             
{}                                      
{0.2cm}                                         
{}                                      

\theoremstyle{thm}
\newtheorem{theorem}[equation]{Theorem}

\newtheorem{corollary}[equation]{Corollary}

\newtheorem{lemma}[equation]{Lemma}
\newtheorem{proposition}[equation]{Proposition}
\newtheorem{definition}[equation]{Definition}

\theoremstyle{rmk}
\newtheorem{example}[equation]{Example}
\newtheorem{remark}[equation]{Remark}

\numberwithin{equation}{section}

\makeatletter
\newlength{\@thlabel@width}%
\newcommand{\thmenumhspace}{\settowidth{\@thlabel@width}{(1)}\sbox{\@labels}{\unhbox\@labels\hspace{\dimexpr-\leftmargin+\labelsep+\@thlabel@width-\itemindent}}}
\makeatother

\usepackage[todo]{optional}

\newcommand{\Cov}{{\mathrm{Cov}}}
\newcommand{\GCov}{{\mathrm{GCov}}}

\newcommand{\Pl}{\operatorname{Pl}}

\newcommand{\Sing}{\operatorname{Sing}}
\newcommand{\Ad}{{\mathrm{Ad}}}

\newcommand{\Cc}{{\mathrm{Cc}}}
\newcommand{\cc}{{\mathrm{cc}}}

\newcommand{\Sp}{{\mathrm{Sp}}}

\newcommand{\Ex}{{\operatorname{Ex}}}

\newcommand{\rmh}{\mathrm{h}}
\newcommand{\rmB}{{\mathrm{B}}}
\newcommand{\rmD}{\mathrm{D}}
\newcommand{\rmC}{\mathrm{C}}

\newcommand{\Map}{\mathrm{Map}}
\newcommand{\ev}{\operatorname{ev}}

\newcommand{\dd}{\mathrm{d}}

\newcommand{\Plot}{\mathrm{Plot}}

\newcommand{\opp}{\mathrm{op}}

\newcommand{\hoRan}{\operatorname{hoRan}}

\newcommand{\holim}{\mathrm{holim}}
\newcommand{\hocolim}{{\mathrm{hocolim}}}

\newcommand{\colim}{{\mathrm{colim}}}
\newcommand{\coeq}{{\mathrm{coeq}}}

\renewcommand{\lim}{{\mathrm{lim}}}

\newcommand{\scI}{\mathscr{I}}
\newcommand{\scV}{{\mathscr{V}}}
\newcommand{\scK}{{\mathscr{K}}}
\newcommand{\scH}{\mathscr{H}}

\newcommand{\scJ}{\mathscr{J}}

\newcommand{\scC}{\mathscr{C}}
\newcommand{\scD}{\mathscr{D}}
\newcommand{\scE}{\mathscr{E}}

\newcommand{\scU}{{\mathscr{U}}}
\newcommand{\scY}{{\mathscr{Y}}}

\newcommand{\scM}{{\mathscr{M}}}

\newcommand{\CI}{\mathcal{I}}
\newcommand{\CJ}{\mathcal{J}}

\newcommand{\CW}{\mathcal{W}}

\newcommand{\bH}{{\mathbf{H}}}

\newcommand{\sfc}{\mathsf{c}}

\newcommand{\NN}{\mathbb{N}}
\newcommand{\bbA}{{\mathbb{A}}}
\newcommand{\RN}{\mathbb{R}}

\newcommand{\bbS}{\mathbb{S}}
\newcommand{\bbL}{\mathbb{L}}

\newcommand{\frg}{{\mathfrak{g}}}

\newcommand{\ul}[1]{\underline{#1}}

\newcommand{\cC}{{\check{C}}}

\newcommand{\Set}{{\mathscr{S}\mathrm{et}}}

\newcommand{\sSet}{{\mathscr{S}\mathrm{et}_{\hspace{-.03cm}\Delta}}}
\newcommand{\Fun}{{\mathscr{F}\hspace{-.03cm}\mathrm{un}}}

\newcommand{\Cat}{{\mathscr{C}\mathrm{at}}}

\newcommand{\Mfd}{{\mathscr{M}\mathrm{fd}}}
\newcommand{\Dfg}{{\mathscr{D}\mathrm{fg}}}

\newcommand{\CSS}{\mathscr{CSS}}
\newcommand{\SSp}{{\mathscr{SS}\hspace{-0.02cm}\mathrm{p}}}
\newcommand{\Cart}{{\mathscr{C}\mathrm{art}}}

\newcommand{\Top}{{\Delta\mathscr{T}\mathrm{op}}}

\newcommand{\Bun}{{\mathscr{B}\mathrm{un}}}

\newcommand{\weq}{\overset{\sim}{\longrightarrow}}

\newenvironment{myitemize}{\begin{itemize}[itemsep=-0.1cm, leftmargin=*, topsep=0cm]}{\end{itemize}}
\newenvironment{myenumerate}{\begin{enumerate}[itemsep=-0.1cm, leftmargin=*, topsep=-0.1cm, label=(\arabic*)]}{\end{enumerate}}
\newenvironment{mysubenumerate}{\begin{enumerate}[itemsep=-0.1cm, leftmargin=*, topsep=-0.1cm, label=(\roman*)]}{\end{enumerate}}

\newcommand{\qen}{\hfill$\triangleleft$}

\newcommand{\qandq}{\quad \text{and} \quad}
\newcommand{\qqandqq}{\qquad \text{and} \qquad}

\newcommand{\bbDelta}{{\mathbbe{\Delta}}}

\newlength{\Displayskip}
\begin{document}

\setlength{\abovedisplayskip}{\Displayskip}
\setlength{\belowdisplayskip}{\Displayskip}

\begin{flushright}
\small
\textsf{Hamburger Beiträge zur Mathematik Nr.\,857}\\
{\sf ZMP--HH/20-13} 
\end{flushright}


\begin{center}
\LARGE{\textbf{The $\RN$-Local Homotopy Theory of Smooth Spaces}}
\end{center}
\begin{center}
\large Severin Bunk
\end{center}

\begin{abstract}
\noindent
Simplicial presheaves on cartesian spaces provide a general notion of smooth spaces.
There is a corresponding smooth version of the singular complex functor, which maps smooth spaces to simplicial sets.
We consider the localisation of the (projective or injective) model category of smooth spaces at the morphisms which become weak equivalences under the singular complex functor.
We prove that this localisation agrees with a motivic-style $\RN$-localisation of the model category of smooth spaces.
Further, we exhibit the singular complex functor for smooth spaces as one of several Quillen equivalences between model categories for spaces and the above $\RN$-local model category of smooth spaces.
In the process, we show that the singular complex functor agrees with the homotopy colimit functor up to a natural zig-zag of weak equivalences.
We provide a functorial fibrant replacement in the $\RN$-local model category of smooth spaces and use this to compute mapping spaces in terms of singular complexes.
Finally, we explain the relation of our fibrant replacement to the concordance sheaf construction introduced recently by Berwick-Evans, Boavida de Brito and Pavlov.
\end{abstract}

\tableofcontents

\section{Introduction and overview}
\label{sec:Intro}

Topological spaces and simplicial sets can be used to construct the same homotopy theory.
This is made rigorous by the fact that the singular complex and geometric realisation functors form a Quillen equivalence between the standard model structure on the category $\mathscr{T}\mathrm{op}$ of topological spaces and the Kan-Quillen model structure on the category $\sSet$ of simplicial sets.
Both of these model categories formalise what is often called the homotopy theory of spaces, or $\infty$-groupoids (which are the same according to Grothendieck's homotopy hypothesis).
The two models differ significantly in their features, though, in that topological spaces encode homotopies via the notion of continuity, while simplicial sets are inherently combinatorial.
Consequently, each of these two models for the homotopy theory of spaces has its own merits in different contexts and applications.

Apart from continuity or combinatorics, another important feature spaces can possess, and which is relevant in many problems in mathematics, is smoothness.
The prime example of a category of smooth spaces is the category $\Mfd$ of manifolds and smooth maps, which underlies much of geometry.
There exists a notion of smooth homotopy within the category $\Mfd$, and one can find smooth versions of many topological concepts, such as cohomology, which are invariant under these smooth homotopies.
However, the category $\Mfd$ is poorly behaved in many ways.
For instance, it is far from being complete or cocomplete, thus making it unable to admit a model structure in the sense of Quillen~\cite{Quillen:Ho_Algebra}.

The way to cure this is to weaken---and therefore to generalise---the concept of a manifold.
Here, we take the following approach to smooth spaces, with the main goal of simultaneously capturing the notions of manifold and (higher) stack.
We start from the category $\Cart$ of cartesian spaces: its objects are all smooth manifolds that are diffeomorphic to $\RN^n$ for any $n \in \NN_0$.
Its morphisms are all smooth maps between these manifolds.
We define a smooth space to be a simplicial presheaf on $\Cart$---informally, we understand the sections of a simplicial presheaf over $c \in \Cart$ as $c$-parameterised families of simplices in a space.
We denote the category of simplicial presheaves on $\Cart$ by $\scH$.
It contains many geometrically interesting objects that are not manifolds or even diffeological spaces~\cite{IZ:Diffeology} (for instance the presheaf of $k$-forms, or the simplicial presheaf of $G$-bundles with connection, for any Lie group $G$).
The category of manifolds, the category of diffeological spaces and the category of simplicial sets each include fully faithfully into $\scH$.

The category $\scH$ carries two natural model structures, namely the projective and the injective model structure on functors $\Cart^\opp \to \sSet$ (where the category $\sSet$ of simplicial sets carries the Kan-Quillen model structure).
We denote the projective and injective model categories by $\scH^p$ and $\scH^i$, respectively, and we write $\scH^{p/i}$ to refer to either of these model structures simultaneously.
The projective and injective model structures are canonically Quillen equivalent via the identity functors $\scH^p \rightleftarrows \scH^i$, but they are not Quillen equivalent to $\sSet$.
In that sense, the model structures $\scH^{p/i}$ do not yet define smooth versions of the homotopy theory of spaces.
To achieve that, one needs a weaker notion of equivalence in $\scH$.

There exist (at least) two candidates for such weakened versions of equivalences in $\scH$.
First, in~\cite{MW:Mumford,GMTW} a notion of weak equivalence has been introduced on sheaves on $\Mfd$ as follows:
let $\Delta_e^k \cong \RN^k$ denote the smooth extended (affine) $k$-simplex.
Extending the usual face and degeneracy maps of the topological standard simplices, this gives rise to a cosimplicial cartesian space $\Delta_e \colon \bbDelta \to \Cart \subset \Mfd$.
Via precomposition, this induces a functor from (pre)sheaves on manifolds to simplicial sets.
In~\cite{MW:Mumford,GMTW}, a morphism of sheaves is considered a weak equivalence of (pre)sheaves whenever it becomes a weak equivalence of simplicial sets under this functor.
We adapt this to our set-up as follows:
for technical reasons, we work with presheaves on cartesian spaces rather than manifolds, and we work with \textit{simplicial} (pre)sheaves instead of ordinary (pre)sheaves.
Let $\delta \colon \bbDelta \to \bbDelta \times \bbDelta$ denote the diagonal functor.
We define the \emph{smooth singular complex functor} as the composite
\begin{equation}
\label{eq:intro S_e}
	S_e \colon \scH \xrightarrow{\Delta_e^*} s\sSet \xrightarrow{\delta^*} \sSet\,,
\end{equation}
where the first functor evaluates $F \in \scH$ on the extended simplices to obtain a bisimplicial set, of which the second functor then takes the diagonal.
Let $S_e^{-1}(W_\sSet)$ denote the class of morphisms in $\scH$ that are mapped to a weak equivalence by $S_e$.
The generalisation to \textit{simplicial} presheaves of the homotopy theory from~\cite{MW:Mumford,GMTW} is then the left Bousfield localisation
\begin{equation}
	L_{S_e^{-1}(W_\sSet)} \scH^{p/i}\,.
\end{equation}

The second notion of weak equivalence in $\scH$ is motivated by motivic homotopy theory (see~\cite{Voevodsky:A1_HoThy,MV:A1-HoThy_of_Schemes,DLORV:Motivic_HoThy}, for instance).
Let $I$ denote the class of all morphisms in $\scH$ of the form $c \times \RN \to c$, where $c$ ranges over all objects in $\Cart$, and where the morphism is the identity on $c$ and collapses $\RN$ to the point.
The left Bousfield localisation
\begin{equation}
	\scH^{p/i\, I} \coloneqq L_I \scH^{p/i}
\end{equation}
is then a version in smooth geometry of motivic localisation.
We call $\scH^{p/i\, I}$ the \emph{$\RN$-local model category} of simplicial presheaves on $\Cart$, or equivalently of smooth spaces.
This localisation has appeared before in~\cite{Schreiber:DCCT,Dugger:Universal_HoThys} and other works of these authors.
For presheaves with values in stable $\infty$-categories, this type of localisation has been investigated in~\cite{BNV--Diff_Coho_Spectra}.
Our first main result is:

\begin{theorem}
\label{st:intro HoThys agree}
There is an equality of model categories:
\begin{equation}
	\scH^{p/i\, I} = L_{S_e^{-1}(W_\sSet)} \scH^{p/i}\,.
\end{equation}
\end{theorem}

A large part of this paper is concerned establishing several explicit Quillen equivalences between these model categories and various model categories describing the homotopy theory of spaces.
Concretely, these are:
the Kan-Quillen model structure on simplicial sets $\sSet$, the model category of $\Delta$-generated topological spaces $\Top$ and the diagonal model structure on bisimplicial sets $(s\sSet)_{diag}$.
Further, we show in Proposition~\ref{st:L_J SSp = L_J CSS} that the model category $(s\sSet)_{diag}$ coincides with the following two model categories:
(1) the localisation $L_{\Delta^\bullet \boxtimes \Delta^0}s\sSet$ of the injective model structure on bisimplicial sets $s\sSet$ at the collapse morphisms $\Delta^n \to \Delta^0$, for $n \in \NN_0$ (where both are seen as simplicial diagrams of discrete simplicial sets), and
(2) the localisation $L_{\Delta^1 \boxtimes \Delta^0} \CSS$ of the model category of complete Segal spaces at the collapse morphism $\Delta^1 \to \Delta^0$.

To state our main theorem, we need the following notation:
we let $\sfc_\bbDelta \colon \sSet \to s\sSet$ denote the functor with sends a simplicial set to a constant simplicial diagram in $\sSet$.
Similarly, let $\tilde{\sfc} \colon \sSet \to \scH$ denote the functor sending a simplicial set to a constant simplicial presheaf on $\Cart$.
Finally, we have a functor
\begin{equation}
	Re \colon \scH^p \to \Top\,,
	\qquad
	F \longmapsto \int^{c \in \Cart^\opp} |F(c)| \times \rmD c\,,
\end{equation}
where $\rmD c$ denotes the underlying topological space of $c \in \Cart$, and where $|{-}| \colon \sSet \to \Top$ is the geometric realisation functor.
If $L : \scC \to \scD : R$ is a pair of adjoint functors, we also express this by the notation $L \dashv R$.
We prove:

\begin{theorem}
Each arrow in the following diagram is a Quillen equivalence (left or right as indicated)
\begin{equation}
\label{eq:The QEq diagram}
\begin{tikzcd}[column sep={3.5cm,between origins}, row sep=0.85cm]
	& \scH^{p \ell I}
	\ar[r, shift left=0.15cm, "1" {description, pos=0.75}, "\perp"' yshift=0.05cm]
	\ar[dd, equal]
	& \scH^{i \ell I}
	\ar[l, shift left=0.15cm, "1" {description, pos=0.75}]
	\ar[dd, equal]
	&
	\\
	L_{S_e^{-1}W} \scH^p
	\ar[dr, equal]
	\ar[ur, equal]
	& &
	& L_{S_e^{-1}W} \scH^i
	\ar[dl, equal]
	\ar[ul, equal]
	\\
	& \scH^{pI}
	\ar[r, shift left=0.15cm, "1" {description, pos=0.75}, "\perp"' yshift=0.05cm]
	\ar[dl, shift left=-0.15cm, "Re" {description, pos=0.75}, "\perp" yshift=0.15cm]
	\ar[dd, shift left=0.15cm, "\ev_*" {description, pos=0.75}]
	& \scH^{iI}
	\ar[l, shift left=0.15cm, "1" {description, pos=0.75}]
	\ar[dd, shift left=0.15cm, "\Delta_e^*" {description, pos=0.75}]
	\ar[dr, "S_e"' {description, pos=0.75}]
	&
	\\
	\Top
	\ar[ur, shift left=-0.15cm, "S" {description, pos=0.75}]
	\ar[dr, shift left=-0.15cm, "\Sing"' {pos=0.5}]
	& &
	& \sSet
	\ar[dl, shift left=-0.15cm, "\delta_!" {description, pos=0.75}, "\perp" yshift=0.15cm]
	\ar[ul, bend right=15, "L_e" {description, pos=0.75}, "\perp" yshift=0.1cm]
	\ar[ul, bend left=15, "R_e" {description, pos=0.75}, "\perp"' yshift=-0.1cm]
	\\
	& \sSet
	\ar[ul, shift left=-0.15cm, "|-|"' {pos=0.5}, "\perp" yshift=0.15cm]
	\ar[uu, shift left=0.15cm, "\tilde{\sfc}"' {description, pos=0.75}, "\dashv"' xshift=0.0cm]
	\ar[r, shift left=0.15cm, "\sfc_\bbDelta", "\perp"' yshift=0.05cm]
	& L_{\Delta^\bullet \boxtimes \Delta^0} \sSet
	\ar[l, shift left=0.15cm, "\ev_{[0]}"]
	\ar[uu, shift left=0.15cm, "\Delta_{e!}"' {description, pos=0.75}, "\dashv"' xshift=0.0cm]
	\ar[ur, shift left=-0.15cm, "\delta^*" {description, pos=0.75}]
	\ar[r, equal]
	& (s\sSet)_{diag}
\end{tikzcd}
\end{equation}
Equalities in this diagram indicate identities of model categories.
Moreover, omitting the functor $R_e$ in the bottom-right triangle, the diagram is a commutative (up to canonical natural isomorphisms) diagram of Quillen equivalences (i.e.~the left adjoints commute and the right adjoints commute).
\end{theorem}

The model structures in the top row are obtained by first localising $\scH^{p/i}$ at good open coverings of cartesian spaces and subsequently localising at the morphisms $c \times \RN \to c$.
The fact that these model structures are equal to the localisation of $\scH^{p/i}$ only at the morphisms $c \times \RN \to c$ holds true because each $c \in \Cart$ is contractible (see Corollary~\ref{st:coincidence of model structures}); it would not hold true if we were working with simplicial presheaves on the category $\Mfd$ of manifolds instead of $\Cart$.
This insight has many pleasant technical consequences.
We provide a comparison with the theory of simplicial sheaves on manifolds in Appendix~\ref{app:H^(pl) simeq wtH^(pl)}.

We remark that there also exists a further model category for a theory of smooth spaces: the $\RN$-local model category of enriched simplicial presheaves on a simplicial category explored in~\cite{HS:Enr_Spli_PShs_and_motivic_HoCat}%
\footnote{We thank the referee for making us aware of this work.}.
This model category is Quillen equivalent to $\scH^{p/i\, I}$ by~\cite[Thm.~2.4]{HS:Enr_Spli_PShs_and_motivic_HoCat} (see also~\cite[Sec.~2]{BEBdBP:Class_sp_of_oo-sheaves} for applications to $\RN$-local simplicial presheaves on the category of manifolds in particular).

Next, we prove various comparison results between the different functors in diagram~\eqref{eq:The QEq diagram}.
The most important one of these, stated below, underlines the significance of the functor $S_e \colon \scH \to \sSet$:
it is a model for the homotopy colimit of diagrams $\Cart^\opp \to \sSet$.

\begin{theorem}
\label{st:intro comparison}
Let $Q^p \colon \scH^p \to \scH^p$ be a cofibrant replacement functor for the projective model structure.
There is a zig-zag of natural weak equivalences
\begin{equation}
	|{-}| \circ S_e\ \overset{\sim}{\longleftarrow} |{-}| \circ S_e \circ Q^p\ \weq\ Re \circ Q^p\,.
\end{equation}
In particular, there is a natural isomorphism in $\rmh(\Top)$,
\begin{equation}
	|S_e F| \cong \hocolim(|F| \colon \Cart^\opp \to \Top)
\end{equation}
and a natural isomorphism in $\rmh(\sSet)$,
\begin{equation}
	S_e F \cong \hocolim(F \colon \Cart^\opp \to \sSet)\,.
\end{equation}
\end{theorem}

On a very formal level, it allows us to identify $S_e$ as a presentation of the left adjoint in the cohesive structure on the $\infty$-topos of presheaves of spaces on $\Cart$.
This has also been observed recently in~\cite{BEBdBP:Class_sp_of_oo-sheaves}.

Finally, we construct a fibrant replacement functor for $\scH^{p/i\, I}$, motivated by the \emph{concordance sheaves} introduced in~\cite{BEBdBP:Class_sp_of_oo-sheaves}, generalising concepts from~\cite{MW:Mumford}.
We thereby obtain explicit access to the mapping spaces in $\scH^{p/i\, I}$:

\begin{theorem}
\label{st:intro Representability}
Let $F,G \in \scH$ be any simplicial presheaves on $\Cart$.
Let $M \in \Mfd$ be any manifold, and define $\ul{M} \in \scH$ by setting $\ul{M}(c) = \Mfd(c,M)$ for any cartesian space $c \in \Cart$.
There are canonical isomorphisms
\begin{align}
	\Map_\sSet(S_e F, S_e G) &\cong \Map_{\scH^{p/i\, I}}(F,G)\,,
	\\
	\Map_\Top(M, |S_e G|) &\cong \Map_{\scH^{p/i\, I}}(\ul{M},G)
\end{align}
 in $\rmh \sSet$, the homotopy category of spaces.
\end{theorem}

Finally, we remark that during the revision of this paper, building on Theorem~\ref{st:intro Representability} and~\cite{BEBdBP:Class_sp_of_oo-sheaves}, Schreiber and Sati have enhanced this result to compute the homotopy type of the mapping space $F^{Q \ul{M}} \in \scH$, where $M$ is a manifold, $\ul{M} = \Mfd(-,M)$ is its associated sheaf, and $F \in \scH$ is a homotopy sheaf~\cite[Thm.~3.3.53]{SS:Equivariant_Pr-oo-bundles} (this problem goes back to a question by C.~Rezk and originally was answered by D.~Pavlov%
\footnote{See the two discussion threads \url{https://nforum.ncatlab.org/discussion/6816/the-shape-of-function-objects} and \url{https://nforum.ncatlab.org/discussion/11361/shape-via-cohesive-path-groupoid}; we thank the anonymous referee for pointing these out to us.}%
).
Moreover, the $I$-localisation of $\infty$-sheaves on manifolds and their fibrant replacement have also been treated extensively in~\cite{ADH:Differential_Cohomology} since the first version of this paper appeared.

\paragraph*{Outline.}
In Section~\ref{sec:L_I H and first spaces from presheaves} we define the $\RN$-localisations $\scH^{p/i\, I}$ of $\scH^{p/i}$.
We show that $\scH^{p/i\, I}$ can also be obtained as further localisations of the \v{C}ech local model structures with respect to differentiably good open coverings.
We define the functor $Re \colon \scH^{p I} \to \Top$ and show that it is a left Quillen equivalence.

In Section~\ref{sec:S_e and its properties}, we study the smooth singular complex functor $S_e \colon \scH \to \sSet$.
We first show that $S_e \colon \scH^{p/i\, I} \to \sSet$ is a left Quillen equivalence.
Subsequently, we establish $S_e \colon \scH^{i I} \to \sSet$ as a right Quillen equivalence.
As an intermediate step, we relate the model categories $\scH^{i I}$ and $\sSet$ to localisations of the model category of complete Segal spaces.

Section~\ref{sec:comparison of spaces from presheaves} is concerned with the comparison of different ways of extracting spaces from simplicial presheaves on $\Cart$.
The key concept is to extend the homotopy equivalence that embeds the topological standard simplices into the smooth extended simplices to obtain natural weak equivalences between functors from $\scH$ to $\sSet$ and $\Top$.
We show that, for each manifold $M$, the space $S_e \ul{M}$ assigned to its associated sheaf recovers the homotopy type of $M$ and prove Theorem~\ref{st:intro comparison}.

In Section~\ref{sec:Concordance} we construct a fibrant replacement functor for $\scH^{p/i\, I}$ and use it to prove Theorem~\ref{st:intro Representability}.
We spell out the relation of this theorem to~\cite{BEBdBP:Class_sp_of_oo-sheaves} and works of Dugger.
Then we apply Theorem~\ref{st:map spaces via concordance} to prove the coincidence of model structures from Theorem~\ref{st:intro HoThys agree}.

We include two appendices; Appendix~\ref{app:Injective fib rep} contains the explicit construction of a fibrant replacement functor for the injective model structure on $\scH$, which features in the proof of Theorem~\ref{st:intro Representability}.
Building on this, we provide a Quillen equivalence between model categories for homotopy sheaves on $\Cart$ and homotopy sheaves on $\Mfd$ in Appendix~\ref{app:H^(pl) simeq wtH^(pl)}.

\paragraph*{Notation and conventions}

\noindent
\begin{longtable}{p{2cm}@{\hspace{-0.5cm}}p{14cm}}
$\bbDelta$ & The simplex category
\\
$\sSet$ & The category of simplicial sets
\\
$s\sSet$ & The category of bisimplicial sets; when used as a model category, this will always be endowed with the Kan-Quillen model structure
\\
$\Top$ & The model category of $\Delta$-generated topological spaces
\\
$\ul{\scC}(-,-)$ & Simplicially enriched hom-functor in a simplicial category $\scC$
\\
$\Cart$ & The category of cartesian spaces
\\
$\scY_c$ & The image of a cartesian space $c \in \Cart$ under the Yoneda embedding
\\
$\Mfd$ & The category of smooth manifolds and smooth maps
\\
$\ul{M}$ & The (simplicially constant) simplicial presheaf on $\Cart$ associated to a manifold $M$
\\
$\Dfg$ & The category of diffeological spaces
\\
$\scH$ & The category of simplicial presheaves on $\Cart$
\\
$\scH^p$ & Projective model structure on $\scH$
\\
$\scH^i$ & Injective model structure on $\scH$
\\
$\scH^{p/i}$ & $\scH$ with either one of the model structures $\scH^p$ or $\scH^i$; statements where $\scH^{p/i}$ is used will always mean that the statement applies to \textit{both} model structures.
\\
$\scH^{p/i\,I}$ & $\RN$-localisation of $\scH^{p/i}$
\\
$\scH^{p/i\, \ell}$ & Left Bousfield localisation of $\scH^{p/i}$ at the (\v{C}ech nerves of) differentiably good open coverings in $\Cart$
\\
$\Delta_e^k$ & Extended $k$-simplex as an element of $\Cart$
\\
$\Delta_e^*$ & Evaluation of elements in $\scH$ on all extended simplices, produces objects in $s\sSet$
\\
$\delta$ & Diagonal functor $\bbDelta \to \bbDelta^2$
\\
$\delta^*$ & Takes the diagonal of a bisimplicial set
\\
$S_e$ & $= \delta^* \circ \Delta_e^*$ the smooth singular complex functor
\\
$L_e, R_e$ & Left and right adjoints to $S_e$, respectively
\\
$Re$ & Realisation functor $\scH \to \Top$
\\
$S$ & Right adjoint to $Re$
\\
$\Cc^{p/i}$ & Fibrant replacement functor in $\scH^{p/i\, I}$
\\
$Q'$ & Modified cofibrant replacement functor in $\scH^p$
\\
$R'$ & Right adjoint to $Q'$
\\
$\tilde{\sfc}$ & Takes a simplicial set to its constant simplicial presheaf on $\Cart$
\\
$\ev_*$ & Evaluates elements of $\scH$ on the final object $* = \RN^0 \in \Cart$
\\
$\rmD$ & Sends a diffeological space to the $\Delta$-generated topological space induced by the plots
\\
$\rmC$ & right adjoint to $\rmD$
\\
$\rmh (\scM)$ & Homotopy category of a model category $\scM$
\end{longtable}

\vspace{-0.4cm}
\begin{myitemize}
\item Given two categories $\scC, \scI$, we write $\Cat(\scI, \scC)$ or $\scC^\scI$ for the category of functors $\scI \to \scC$.

\item We will also be working with the category $s\sSet = \Cat(\bbDelta^\opp, \sSet)$ of bisimplicial sets.
Our convention is always to write a bisimplicial set as a functor
\begin{equation}
	X \colon \bbDelta^\opp \to \sSet\,,
	\quad
	[n] \mapsto X_n\,,
	\quad \text{with} \quad
	X_{n,k} \coloneqq \big( X[n] \big) [k]\,.
\end{equation}
\end{myitemize}

\vspace{-0.2cm}
\subsection*{Acknowledgements}

The author is grateful to Birgit Richter for her encouragement to carry out this project.
Further, the author would like to thank Walker Stern and Lukas Müller for numerous discussions, as well as Dan Christensen, Enxin Wu, Daniel Berwick-Evans, Pedro Boavida de Brito and Dmitri Pavlov for insightful comments on an earlier version of this paper.
Finally, the author is grateful to the anonymous referee, whose very careful reading and detailed comments improved this paper notably.
The author was partially supported by the Deutsche Forschungsgemeinschaft (DFG, German Research Foundation) under Germany's Excellence Strategy---EXC 2121 ``Quantum Universe''---390833306.

\section{$\RN$-local model structures and smooth spaces}
\label{sec:L_I H and first spaces from presheaves}

We start by setting up the model-categorical background used in this paper.
Partially following~\cite{Schreiber:DCCT} and~\cite{Dugger:HoSheaves}, we consider the category of simplicial presheaves on cartesian spaces with its canonical projective and injective model structures.
In analogy with $\bbA^1$-local homotopy theory, we localise this category at all the morphisms $c \times \RN \to c$, where $c$ is any cartesian space and where the morphism is the projection onto the factor $c$.
Extending ideas from~\cite{Dugger:HoSheaves}, we establish several Quillen equivalences of this localised model category with the categories of simplicial sets and topological spaces.

\subsection{$\RN$-local model structures on simplicial presheaves}
\label{sec:R-local MoStrs}

In this section we start by setting up the various model structures on simplicial presheaves that will play a role in this article.
We start by introducing the central mathematical objects.

\begin{definition}
We let $\Cart$ denote the (small) category of submanifolds of $\RN^\infty$ that are diffeomorphic to $\RN^n$, for any $n \in \NN_0$.
These manifolds are called \emph{cartesian spaces}.
The morphisms $c \to d$ in $\Cart$ are the smooth maps $c \to d$ between these manifolds.
\end{definition}

In other words, $\Cart$ is the full subcategory of the category $\Mfd$ of smooth manifolds and smooth maps on the cartesian spaces.

\begin{definition}
Let $\scH \coloneqq \Cat(\Cart^\opp, \sSet)$ denote the category of simplicial presheaves on $\Cart$, and let $c \mapsto \scY_c \in \Cat(\Cart^\opp, \Set) \subset \scH$ denote its Yoneda embedding.
There is a fully faithful inclusion
\begin{equation}
	\ul{({-})} \colon \Mfd \to \scH\,,
	\qquad
	\ul{M}(c) = \Mfd(c,M)\,.
\end{equation}
\end{definition}

We can view each $c \in \Cart$ as a manifold via the inclusion $\Cart \subset \Mfd$, and we have $\ul{c} = \scY_c$.
We view $\sSet$ as endowed with the Kan-Quillen model structure.
The category $\Cart$ carries a Grothendieck coverage $\tau$ of differentiably good open coverings---see~\cite[App.~A]{FSS:Cech_diff_char_classes_via_L_infty} for details.
A covering of $c \in \Cart$ in this coverage is a collection of morphisms $\{\iota_i \colon c_i \to c\}_{i \in \Lambda}$ in $\Cart$ such that each $\iota_i$ is an embedding of an open subset, the images of the maps $\iota_i$ cover $c$ (i.e.~each $x \in c$ lies in the image of some $\iota_i$), and every finite intersection
\begin{equation}
\label{eq:def C_(i_0...i_n)}
	C_{i_0 \ldots i_n} \coloneqq \bigcap_{a = 0}^n \iota_{i_a}(c_{i_a}) \subset c
\end{equation}
with $i_0, \ldots, i_n \in \Lambda$ is either empty or a cartesian space.
We let $\ell$ denote the class of \v{C}ech nerves in $\scH$ of differentiably good open coverings in $\Cart$.
Given a simplicial model category $\scM$ and a class $S$ of morphisms in $\scM$, we denote by $L_S \scM$ the left Bousfield localisation of $\scM$ at the morphisms in $S$ (see~\cite{Hirschhorn:MoCats} for more background).

\begin{remark}
The left Bousfield localisation of a simplicial, left proper, cellular model category is again a simplicial model category~\cite[Thm.~4.1.1(4)]{Hirschhorn:MoCats}.
The model category $\sSet$ is both left proper and cellular (see~\cite[Prop.~12.1.4]{Hirschhorn:MoCats} for the cellular part), and these properties carry over to the projective model structure on $\Fun(\scC, \sSet)$, where $\scC$ is any small category~\cite[Prop.~12.1.5]{Hirschhorn:MoCats}.
The same holds true for the injective model structure on simplicial presheaves~\cite[p.~180]{Jardine:Local_HoThy}.
In particular, left Bousfield localisations of $\scH^{p/i}$ are again simplicial.
Alternatively, one can obtain these localisations as simplicial model categories using Barwick's theory of enriched Bousfield localisations~\cite{Barwick:Enriched_B-Loc}.
\qen
\end{remark}

\begin{definition}
\label{def:H^pi and its locs}
We define the following model categories:
\begin{myenumerate}
\item $\scH^{p/i}$ is the projective (resp.~injective) model structure on $\scH$.
We also refer to $\scH^{p/i}$ as the model categories of \emph{smooth spaces}.

\item We define the left Bousfield localisations
\begin{equation}
	\scH^{p/i\, \ell} \coloneqq L_\ell \scH^{p/i}\,.
\end{equation}
This is the projective (resp.~injective) model structure for sheaves of $\infty$-groupoids on $\Cart$.

\item Let $I \coloneqq \{ \scY_c \times \scY_\RN \longrightarrow \scY_c\}_{c \in \Cart}$ be the set of morphisms obtained by taking the product of the collapse map $\RN \to *$ with all identities $\{1_c\}_{c \in \Cart}$.
We define the projective (resp.~injective) \emph{$\RN$-local model category of smooth spaces} as the left Bousfield localisation
\begin{equation}
		\scH^{p/i\, I} \coloneqq L_I \scH^{p/i}\,.
\end{equation}

\end{myenumerate}
\end{definition}

We can further form the localisations
\begin{equation}
		\scH^{p/i\, \ell I} \coloneqq L_I \scH^{p/i\, \ell} = L_\ell \scH^{p/i\, I}\,,
\end{equation}
which agree for purely abstract reasons (localisations commute).

\begin{remark}
We interpret the localisation $\scH^{p/i\, I}$ as an $\RN$-localisation of $\scH^{p/i}$ akin to motivic localisation (see, for example, \cite{Voevodsky:A1_HoThy,MV:A1-HoThy_of_Schemes,DLORV:Motivic_HoThy, Jardine:Local_HoThy}).
Though thinking of objects in $\scH$ as smooth spaces, we will mostly refer objects in $\scH$ by the more technically precise term of simplicial presheaves, and we will sometimes refer to fibrant objects in $\scH^{p/i\, I}$ as \emph{essentially constant} or \emph{$\RN$-local} simplicial presheaves (see Proposition~\ref{st:tensor and ev weqs in H^(p/i I)} for justification of this terminology).
\qen
\end{remark}

Note that we use the term \textit{symmetric monoidal} model category in the sense of~\cite[Def.~4.1.13, Def.~4.2.6]{Hovey:MoCats}.
In particular, a \textit{cartesian} model category is a symmetric monoidal model category whose monoidal product is the category-theoretic product.

\begin{proposition}
\label{st:basic props of H^pi and its locs}
All model structures in Definition~\ref{def:H^pi and its locs} are simplicial, left proper, tractable and symmetric monoidal.
\end{proposition}

\begin{proof}
Except for the claim that the model structures are symmetric monoidal, all assertions follow from~\cite[Thm.~4.7]{Barwick:Enriched_B-Loc}, building on results in~\cite{Beke:Sheafifiable_HoThys}.
The model structures for sheaves of $\infty$-groupoids are symmetric monoidal by~\cite[Thm.~4.58]{Barwick:Enriched_B-Loc}.
To see that $\scH^{p/i\, I}$ is symmetric monoidal, observe that the objects $\scY_c$, $c \in \Cart$, form a set of homotopy generators for $\scH^{p/i}$.
Let $F \in \scH^{p/i\, I}$ be a local object, and consider the internal hom object $F^{\scY_d}$ for any $d \in \Cart$.
For any of the morphisms $\scY_c \times \scY_\RN \to \scY_c$ in $I$, the internal hom adjunction yields a commutative diagram of simplicially enriched hom spaces
\begin{equation}
\begin{tikzcd}[row sep=0.75cm]
	\ul{\scH}(\scY_c, F^{\scY_d}) \ar[r] \ar[d, "\cong"']
	& \ul{\scH}(\scY_c \times \scY_\RN, F^{\scY_d}) \ar[d, "\cong"]
	\\
	\ul{\scH}(\scY_{c \times d}, F) \ar[r]
	& \ul{\scH}(\scY_{c \times d} \times \scY_\RN, F)
\end{tikzcd}
\end{equation}
Here we have used that $\Cart$ has finite products.
The bottom horizontal morphism is induced by the morphism $c \times d \times \RN \to c \times d$, which is an element of $I$.
Hence, the bottom morphism is a weak equivalence in $\sSet$.
Therefore, it follows from~\cite[Prop.~4.47]{Barwick:Enriched_B-Loc} that $\scH^{p/i\, I}$ is symmetric monoidal.
The exact same proof shows that $\scH^{p/i\, \ell I}$ is symmetric monoidal as well.
The fact that $\scH^{p/i\, I \ell}$ is symmetric monoidal will follow from Corollary~\ref{st:coincidence of model structures}.
\end{proof}

The injective case can also be found in~\cite{Rezk:Cartesian_presentation}.
We record the following two direct observations:

\begin{proposition}
\label{st:p/i Quillen equivalences}
There are commutative diagrams of simplicial Quillen adjunctions:
\begin{equation}
\begin{tikzcd}[row sep=0.75cm]
	\scH^p \ar[r, shift left=0.1cm] \ar[d, shift left=-0.1cm]
	& \scH^i \ar[l, shift left=0.1cm] \ar[d, shift left=-0.1cm]
	\\
	\scH^{p \ell} \ar[r, shift left=0.1cm] \ar[u, shift left=-0.1cm] \ar[d, shift left=-0.1cm]
	& \scH^{i \ell} \ar[l, shift left=0.1cm] \ar[u, shift left=-0.1cm] \ar[d, shift left=-0.1cm]
	\\
	\scH^{p \ell I} \ar[r, shift left=0.1cm] \ar[u, shift left=-0.1cm]
	& \scH^{i \ell I} \ar[l, shift left=0.1cm] \ar[u, shift left=-0.1cm]
\end{tikzcd}
\qandq
\begin{tikzcd}
	\scH^p \ar[r, shift left=0.1cm] \ar[d, shift left=-0.1cm]
	& \scH^i \ar[l, shift left=0.1cm] \ar[d, shift left=-0.1cm]
	\\
	\scH^{p I} \ar[r, shift left=0.1cm] \ar[u, shift left=-0.1cm] \ar[d, shift left=-0.1cm]
	& \scH^{i I} \ar[l, shift left=0.1cm] \ar[u, shift left=-0.1cm] \ar[d, shift left=-0.1cm]
	\\
	\scH^{p \ell I} \ar[r, shift left=0.1cm] \ar[u, shift left=-0.1cm]
	& \scH^{i \ell I} \ar[l, shift left=0.1cm] \ar[u, shift left=-0.1cm]
\end{tikzcd}
\end{equation}
where the rightwards and downwards arrows are the left adjoints.
All arrows are identity functors, and all horizontal arrows are Quillen equivalences.
\end{proposition}

\begin{proposition}
\label{st:same weqs}
Each pair of model categories defined in Definition~\ref{def:H^pi and its locs}(1)--(4) (based on either the projective or the injective model structure), respectively, has the same weak equivalences.
That is, their respective underlying relative categories agree.
\end{proposition}

The reason why we also refer to fibrant objects in $\scH^{p/i\, I}$ as \emph{essentially constant} simplicial presheaves is the following fact (the second statement is a generalisation of~\cite[Lemma~3.4.2]{Dugger:HoSheaves}):

\begin{proposition}
\label{st:tensor and ev weqs in H^(p/i I)}
Let $F \in \scH$.
The following statements hold true:
\begin{myenumerate}
\item The canonical morphism $F \otimes \scY_c \to F$ is a weak equivalence in $\scH^{p/i\, I}$, for every $c \in \Cart$.

\item Let $F \in \scH^{p/i}$ be fibrant.
The following are equivalent:
\begin{mysubenumerate}
\item $F$ is fibrant in $\scH^{p/i\, I}$.

\item The canonical map $F(*) \to F(c)$ is a weak equivalence in $\sSet$, for every $c \in \Cart$.

\item For any $f \colon c \to d$ in $\Cart$, the morphism $F(f) \colon F(d) \to F(c)$ is a weak equivalence in $\sSet$.
\end{mysubenumerate}
\end{myenumerate}
\end{proposition}

\begin{proof}
Ad~(1):
By Proposition~\ref{st:same weqs}, it suffices to show this for $\scH^{i I}$.
There, every object is cofibrant (since every object in $\sSet$ is cofibrant), so that the functor $F \otimes (-) \colon \scH^{i I} \to \scH^{i I}$ is left Quillen.
Thus, it suffices to show that the morphism $\scY_c \to *$ is a weak equivalence.
Since $c \cong \RN^n$ for some $n \in \NN_0$, we can reduce to the case where $c = \RN^n$.

We can write the collapse morphism $\RN^n \to *$ as a composition
\begin{equation}
	\RN^n \cong \RN^{n-1} \times \RN
	\longrightarrow \RN^{n-1} \cong \RN^{n-2} \times \RN
	\longrightarrow \cdots \longrightarrow *\,,
\end{equation}
where each arrow is an element of $I$.
Thus, the claim follows.

Ad~(2):
Condition~(ii) implies that $F$ is fibrant in $\scH^{p/i\, I}$, i.e.~condition (i):
since $\Cart$ has finite products, we have a commutative triangle
\begin{equation}
\begin{tikzcd}[column sep={1.5cm,between origins}, row sep=0.5cm]
	& F(*) \ar[dl, "\sim"'] \ar[dr, "\sim"] &
	\\
	F(c) \ar[rr] & & F(c \times \RN)
\end{tikzcd}
\end{equation}
for any $c \in \Cart$.
The fact that $F$ is $I$-local thus follows from the two-out-of-three property of weak equivalences in $\sSet$.

We now show that~(i) implies~(ii):
By part~(1) we know that for each $c \in \Cart$, the morphism $\scY_c \to *$ is a weak equivalence in $\scH^{p/i\, I}$.
The claim then follows from the enriched Yoneda Lemma:
the top arrow in the commutative diagram
\begin{equation}
\begin{tikzcd}
	\ul{\scH}(*, F) \ar[r] \ar[d, "\cong"'] & \ul{\scH}(\scY_c, F) \ar[d, "\cong"]
	\\
	F(*) \ar[r] & F(c)
\end{tikzcd}
\end{equation}
is a weak equivalence since $\scH^{p/i\, I}$ is simplicial (Proposition~\ref{st:basic props of H^pi and its locs}), so that $\ul{\scH}(-,F)$ is a right Quillen functor.
As representables are cofibrant in $\scH^{p/i\, I}$, it follows that the functor $\ul{\scH}(-,F)$ preserves the weak equivalence $\scY_c \weq *$.

Finally, it is clear that~(iii) implies~(ii), and the converse implication follows readily by combining the fact that any morphism $f$ in $\Cart$ fits into a commutative triangle
\begin{equation}
\begin{tikzcd}
	c \ar[rr, "f"] \ar[dr]
	& & d \ar[dl]
	\\
	& * &
\end{tikzcd}
\end{equation}
with the two-out-of-three property of weak equivalences.
\end{proof}

Our definitions of model structures on $\scH$ are redundant, by the following useful theorem:

\begin{theorem}
\label{st:criterion for coincidence of MoStrs}
\emph{\cite[Prop.~E.1.10]{Joyal:Thy_of_QCats_and_appls}}
Let $\scM$ and $\scM'$ be two model categories with the same underlying category.
Then, $\scM$ and $\scM'$ coincide as model categories if and only if they have the same cofibrations and the same fibrant objects.
\end{theorem}

\begin{corollary}
\label{st:coincidence of model structures}
We have the following identities of model categories:
\begin{equation}
	\scH^{p/i\, \ell I}
	= \scH^{p/i\, I}\,.
\end{equation}
In particular, every \v{C}ech-local weak equivalence is an $I$-local weak equivalence.
\end{corollary}

\begin{proof}
By their construction as left Bousfield localisations, all of the above three model categories have the same cofibrations.
Thus, it suffices to check that their fibrant objects coincide.

An object $F \in \scH^{p/i\, \ell I}$ is fibrant if and only if it is fibrant in $\scH^{p/i\, \ell}$ and satisfies that the canonical map
\begin{equation}
	F(c) \cong \ul{\scH} (\scY_c, F)
	\longrightarrow \ul{\scH} (\scY_c \times \scY_\RN, F) \cong F(c \times \RN)
\end{equation}
is a weak equivalence in $\sSet$, for every $c \in \Cart$.
That is, an object in $\scH^{p/i\, \ell I}$ is fibrant precisely if it is fibrant in both $\scH^{p/i\, \ell}$ and in $\scH^{p/i\, I}$.
In particular, this implies that $F$ is fibrant also in $\scH^{p/i\, I}$.

To prove the converse, we first introduce some notation.
Given a set $S$, let $S^{[\cdot]} \in \sSet$ denote the simplicial set whose $n$-simplices are $n{+}1$-tuples $(i_0, \ldots, i_n) \in S^{n+1}$ of elements in $S$; its $i$-th face maps forget the respective $i$-th entry, and its $j$-th degeneracy maps duplicate the respective $j$-th entry of a tuple.
Consider an object $c \in \Cart$, and let $\scU = \{c_i \to c\}_{i \in \Lambda}$ be a differentiably good open covering of $c$.
Let $\Lambda_{ne}^n \subset \Lambda^{n+1}$ be the subset on those $n{+}1$-tuples $(i_0, \ldots, i_n) \in \Lambda^{n+1}$ such that $C_{i_0 \cdots i_n} \neq \emptyset$ (see~\eqref{eq:def C_(i_0...i_n)} for the notation).
One checks that this defines a simplicial subset \smash{$\Lambda_{ne} \subset \Lambda^{[\cdot]}$}.

Given a projectively fibrant simplicial presheaf $F \in \scH$, consider the maps of simplicial sets
\begin{align}
\label{eq:ascent morphism}
	F(c) \longrightarrow &\underset{\bbDelta}{\holim} \Big( \cdots \prod_{i_0, \ldots, i_n \in \Lambda} \ul{\scH}(Q \scY_{C_{i_0 \ldots i_n}}, F) \cdots \Big)
	\\*
	\weq &\underset{\bbDelta}{\holim} \Big( \cdots \prod_{(i_0, \ldots, i_n) \in \Lambda_{ne}^n} F(C_{i_0 \ldots i_n}) \cdots \Big)\,.
\end{align}
The second map is a weak equivalence because representables are already cofibrant.

Now, let $F \in \scH^{p/i\, I}$ be fibrant.
We need to check that $F$ satisfies descent with respect to the Grothendieck coverage $\tau$ of differentiably good open coverings on $\Cart$.
Since $F$ is essentially constant (i.e.~fibrant in $F \in \scH^{p/i\, I}$), by Proposition~\ref{st:tensor and ev weqs in H^(p/i I)} the collapse maps $c \to *$ induce weak equivalences $F(*) \weq F(c)$.
We thus have a commutative diagram
\begin{equation}
\begin{tikzcd}
	F(*) \ar[r] \ar[d, "\sim"']
	& \underset{\bbDelta}{\holim} \Big( \cdots \displaystyle{\prod_{i_0, \ldots, i_n \in \Lambda_{ne}^n}} F(*) \cdots \Big) \ar[d, "\sim"]
	\\
	F(c) \ar[r]
	& \underset{\bbDelta}{\holim} \Big( \cdots \displaystyle{\prod_{i_0, \ldots, i_n \in \Lambda_{ne}^n}} F(C_{i_0 \ldots i_n}) \cdots \Big)
\end{tikzcd}
\end{equation}
in $\sSet$.
We claim that the top morphism in this diagram is an equivalence:
to see this, we first note that, by assumption on the open covering $\scU$, the collapse morphism $\Sing(C_{i_0 \ldots i_n}) \to *$ is a weak equivalence in $\sSet$ for any $i_0, \ldots, i_n \in \Lambda$ such that $C_{i_0 \ldots i_n}$ is non-empty.
For any fibrant $K \in \sSet$, we thus obtain a weak equivalence
\begin{equation}
	K \cong \ul{\sSet}(*,K) \weq \ul{\sSet} \big( \Sing(C_{i_0 \ldots i_n}), K \big)\,.
\end{equation}
Since $F(*) \in \sSet$ is fibrant, the product of these morphisms indexed by $i_0, \ldots, i_n \in \Lambda$ is still a weak equivalence, so we obtain a commutative diagram
\begin{equation}
\label{eq:diag I-loc implies descent}
\begin{tikzcd}
	F(*) \ar[r] \ar[d, "\sim"']
	& \underset{\bbDelta}{\holim} \Big( \cdots \displaystyle{\prod_{i_0, \ldots, i_n \in \Lambda_{ne}^n}} F(*) \cdots \Big) \ar[d, "\sim"]
	\\
	\ul{\sSet} \big( \Sing(c),F(*) \big) \ar[r]
	&  \ul{\sSet} \Big( \underset{\bbDelta^\opp}{\hocolim} \Big( \cdots \displaystyle{\coprod_{i_0, \ldots, i_n \in \Lambda_{ne}^n}} \Sing(C_{i_0 \ldots i_n}) \cdots \Big), F(*) \Big)\,.
\end{tikzcd}
\end{equation}
It follows from an application of \cite[Thm.~A.3.1]{Lurie:HTT}, or by~\cite[Thm.~1.1]{DI:Top_hypercovers_and_A1} that the morphism
\begin{equation}
	\underset{\bbDelta^\opp}{\hocolim} \Big( \cdots \displaystyle{\coprod_{i_0, \ldots, i_n \in \Lambda}} \Sing(C_{i_0 \ldots i_n}) \cdots \Big)
	\longrightarrow\Sing(c)
\end{equation}
is a weak equivalence in $\sSet$.
This is preserved by the right Quillen functor $\ul{\sSet}(-,F(*))$, and thus the claim follows.
\end{proof}

\begin{remark}
Corollary~\ref{st:coincidence of model structures} fails if one considers simplicial (pre)sheaves on \textit{manifolds} rather than on cartesian spaces:
the proof we give above relies on the fact that  the left-hand vertical morphism in Diagram~\eqref{eq:diag I-loc implies descent} is a weak equivalence.
This is true because every $c \in \Cart$ has an underlying topological space which is contractible.
In contrast, consider a simplicial presheaf on manifolds, $G \colon \Mfd^\opp \to \sSet$, which is projectively fibrant and $I$-local, i.e.~it satisfies that the canonical morphism $G(M) \to G(M \times \RN)$ is a weak equivalence for every $M \in \Mfd$.
Then, $G$ does not necessarily satisfy descent with respect to open coverings of manifolds.
For instance, consider the presheaf $[-,\bbS^1]$, sending $M \in \Mfd$ to the set of homotopy classes of continuous (or smooth) maps from $M$ to $\bbS^1$.
This is projectively fibrant and $I$-local.
However, let $\scU = \{U_i\}_{i \in \Lambda}$ be an open covering of $M = \bbS^1$ such that each finite intersection $U_{i_0 \ldots i_n}$ is empty or a cartesian space.
Then,
\begin{equation}
	\holim_{\bbDelta} \prod_{i_0, \ldots, i_n \in \Lambda} [U_{i_0 \ldots i_n}, \bbS^1] \simeq *\,,
\end{equation}
but $[\bbS^1, \bbS^1] \not\simeq *$.
For more details on the relation between sheaves on $\Cart$ and sheaves on $\Mfd$, see Appendix~\ref{app:H^(pl) simeq wtH^(pl)}.
\qen
\end{remark}

\begin{proposition}
\label{st:equiv localisations of H}
Let $L_{\RN^\bullet} \scH^{p/i}$ denote the left Bousfield localisation of $\scH^{p/i}$ at the collapse morphisms $\{c \to *\}_{c \in \Cart}$.
Further, let $L_\Cart \scH^{p/i}$ denote the left Bousfield localisation of $\scH^{p/i}$ at all morphisms in $\Cart$.
We have the following identities of model categories:
\begin{equation}
	\scH^{p/i\, I}
	= L_{\RN^\bullet} \scH^{p/i}
	= L_\Cart \scH^{p/i}\,.
\end{equation}
\end{proposition}

\begin{proof}
This follows from Proposition~\ref{st:tensor and ev weqs in H^(p/i I)} and Theorem~\ref{st:criterion for coincidence of MoStrs}.
\end{proof}

\subsection{Evaluation on the point}
\label{sec:c -| ev_*}

Here we present the first of several ways of extracting a space from an object $F \in \scH$ and show that it provides a Quillen equivalence between $\scH^{p/i\, I}$ and the Kan-Quillen model category $\sSet$.

Consider the Quillen adjunction
\begin{equation}
\begin{tikzcd}
	\tilde{\sfc} : \sSet \ar[r, shift left=0.14cm, "\perp"' yshift=0.05cm]
	& \scH^{p/i} : \ev_*\,, \ar[l, shift left=0.14cm]
\end{tikzcd}
\end{equation}
whose left adjoint $\tilde{\sfc}$ sends a simplicial set $K$ to the constant simplicial presheaf with value $K$, and whose right adjoint evaluates a simplicial presheaf at the final object $* \in \Cart$.
(Indeed, the adjunction is Quillen for both targets $\scH^p$ and $\scH^i$; in the projective case, we readily see that $\ev_*$ is right Quillen, and in the injective case we see that $\tilde{\sfc}$ is left Quillen.)
Composing this with the localisation adjunction $\scH^{p/i} \rightleftarrows \scH^{p/i\, I}$, we obtain Quillen adjunctions
\begin{equation}
\label{eq:tc -| ev_*}
\begin{tikzcd}
	\tilde{\sfc} : \sSet \ar[r, shift left=0.14cm, "\perp"' yshift=0.05cm]
	& \scH^{p/i\, I} : \ev_*\,. \ar[l, shift left=0.14cm]
\end{tikzcd}
\end{equation}

\begin{lemma}
\label{st:ev of tc -| ev_* on fibrants}
Let $e \colon \tilde{\sfc} \circ \ev_* \to 1_{\scH}$ denote the counit of the adjunction~\eqref{eq:tc -| ev_*}.
The morphism $e_{|F} \colon \tilde{\sfc} \circ \ev_*(F) \longrightarrow F$ is an objectwise weak equivalence whenever $F \in \scH^{p/i\, I}$ is fibrant.
\end{lemma}

\begin{proof}
For any $F \in \scH$, the morphism $e_{|F}$ of simplicial presheaves is the morphisms $F(c) \to F(*)$ in $\sSet$ induced by the collapse maps $c \to *$.
It readily follows from Proposition~\ref{st:tensor and ev weqs in H^(p/i I)} that $e_{|F}$ is an objectwise weak equivalence whenever $F$ is fibrant.
\end{proof}

\begin{lemma}
\label{st:R^i c K is iI-fibrant}
Let $K \in \sSet$ be any simplicial set.
Let $R^i \colon \scH^i \to \scH^i$ denote a fibrant replacement functor in $\scH^i$ (see Appendix~\ref{app:Injective fib rep} for an explicit construction).
Then, the simplicial presheaf $R^i \tilde{\sfc}(K)$ is fibrant in $\scH^{i I}$.
\end{lemma}

\begin{proof}
The claim follows from Proposition~\ref{st:equiv localisations of H} since the morphism $\tilde{\sfc}K \to R^i \tilde{\sfc}K$ is an objectwise weak equivalence.
\end{proof}

We can now prove a version of~\cite[Thm.~3.4.3]{Dugger:HoSheaves} in the context of simplicial presheaves on cartesian spaces rather than on manifolds.
(There, the proof is outlined for simplicial presheaves on manifolds, where several additional steps are necessary.
Since we work over cartesian spaces, we can employ a slightly different strategy in our proof that allows us to avoid these additional steps.)
Let $\cc^{p/i} \colon 1_\scH \weq \Cc^{p/i}$ be a functorial fibrant replacement in $\scH^{p/i\, I}$ (we provide an explicit construction in Section~\ref{sec:Concordance}).
Note that once a fibrant replacement $\Cc^p$ is given, $\cc^i \colon 1_\scH \weq \Cc^i$ can be defined as the composition
\begin{equation}
\begin{tikzcd}
	1_\scH \ar[r, "\cc^p"]
	& \Cc^p \ar[r, "r^i"]
	& R^i \circ \Cc^p \eqqcolon \Cc^i
\end{tikzcd}
\end{equation}
where $r^i \colon 1_\scH \to R^i$ is a fibrant replacement in the injective model structure.
We always take $\Cc^i$ to be of this form.

\begin{lemma}
\label{st:cK -> Cc^p/i cK is objwise weq}
For each $K \in \sSet$, the morphism
\begin{equation}
	\cc^{p/i}_{\tilde{\sfc} K} \colon \tilde{\sfc} K \longrightarrow \Cc^{p/i} \tilde{\sfc} K
\end{equation}
is an objectwise weak equivalence.
\end{lemma}

\begin{proof}
Let $R$ denote a fibrant replacement functor in $\sSet$ (such as the $\Ex^\infty$-functor).
Consider the commutative diagram
\begin{equation}
\begin{tikzcd}[column sep=1.75cm, row sep=1cm]
	\tilde{\sfc} K \ar[r, "\cc^p_{\tilde{\sfc} K}"] \ar[d]
	& \Cc^p \tilde{\sfc} K \ar[d]
	\\
	\tilde{\sfc} R K \ar[r, "\cc^p_{\tilde{\sfc} R K}"]
	& \Cc^p \tilde{\sfc} R L
\end{tikzcd}
\end{equation}
The left vertical morphism is an objectwise weak equivalence.
The right vertical morphism is a weak equivalence in $\scH^{p I}$ between fibrant objects (in $\scH^{p I}$).
Thus, it is also a weak equivalence in $\scH^p$.
The bottom left object is fibrant in $\scH^{p I}$ by construction.
Hence, the bottom horizontal morphism, which is only a weak equivalence in $\scH^{p I}$ a priori, is even an objectwise weak equivalence.
Therefore, by the two-out-of-three property of weak equivalences in $\scH^p$, the top horizontal morphism is an objectwise weak equivalence as well.
The injective case follows since $\cc^i$ is the composition of this morphism by the objectwise weak equivalence $r^i_{\tilde{\sfc} K}$.
\end{proof}

\begin{theorem}
\label{st:tc -| ev_* is QEq}
The Quillen adjunction $\tilde{\sfc} \dashv \ev_*$ from~\eqref{eq:tc -| ev_*} is a Quillen equivalence.
\end{theorem}

\begin{proof}
We will show that both the derived unit and derived counit of the Quillen adjunction $\tilde{\sfc} \dashv \ev_*$ are weak equivalences.
This implies the claim by~\cite[Prop.~1.3.13]{Hovey:MoCats}.
For the derived counit, let $F \in \scH^{p/i\, I}$ be fibrant.
Since all objects in $\sSet$ are cofibrant, it suffices to check that
\begin{equation}
\begin{tikzcd}
	\tilde{\sfc} \circ \ev_* (F) \ar[r, "e_{|F}"]
	& F\,,
\end{tikzcd}
\end{equation}
is a weak equivalence in $\sSet$, where $e_{|F}$ is the component at $F$ of the counit of $\tilde{\sfc} \dashv \ev_*$.
This holds true by Lemma~\ref{st:ev of tc -| ev_* on fibrants}.
For the derived unit, let $K \in \sSet$ and consider the composition
\begin{equation}
\begin{tikzcd}
	K \ar[r, "\eta_K"]
	& \ev_* \tilde{\sfc} K \ar[r, "\cc^{p/i}_{\tilde{\sfc} K}"]
	& \ev_* \Cc^{p/i} \tilde{\sfc} K\,,
\end{tikzcd}
\end{equation}
The morphism $\eta_K$ is an isomorphism, and the morphism $\cc^{p/i}_{\tilde{\sfc} K}$ is an objectwise weak equivalence by Lemma~\ref{st:cK -> Cc^p/i cK is objwise weq}.
\end{proof}

\begin{corollary}
Let $W_\sSet$ denote the weak equivalences in $\sSet$ and $W_I$ those in $\scH^{p/i\, I}$.
The functor $\tilde{\sfc}$ preserves and reflects weak equivalences as a functor of relative categories $(\sSet, W_\sSet) \longrightarrow (\scH, W_I)$.
\end{corollary}

\begin{example}
Let $G$ be a Lie group with Lie algebra $\frg$.
Consider the object $\Bun_{G,0}^\nabla \in \scH^p$, whose value on $c \in \Cart$ is the nerve of the following groupoid (in particular, $\Bun_{G,0}^\nabla$ is fibrant in $\scH^p$ by construction):
its objects are $\frg$-valued 1-forms $A \in \Omega^1(c,\frg)$ such that $\dd A + \frac{1}{2} [A,A] = 0$, and its morphisms $A \to A'$ are smooth maps $g \colon c \to G$ such that $A' = \Ad(g^{-1}) \circ A + g^* \mu_G$, where $\mu_G$ is the Maurer-Cartan form on $G$.
In other words, $A$ is a flat $G$-connection on a trivial principal $G$-bundle on $c$, and $g$ is equivalently a morphism of flat principal $G$-bundles on $c$.
In particular, any such morphism $g$ is actually a \emph{constant} map $g \colon c \to G$.
Observe that $\Bun_{G,0}^\nabla(*)$ is the nerve of the groupoid with one object and the group underlying $G$ as its morphisms.
It hence follows that the functor $\Bun_{G,0}^\nabla(*) \longrightarrow \Bun_{G,0}^\nabla(c)$ is fully faithful (on the underlying groupoids), for any $c \in \Cart$.
Since any flat $G$-bundle on $c$ is isomorphic to the trivial flat $G$-bundle (because $c \cong \RN^n$ for some $n \in \NN_0$), the functor $\Bun_{G,0}^\nabla(*) \longrightarrow \Bun_{G,0}^\nabla(c)$ is also essentially surjective.
Since the nerve of an equivalence of groupoids is an equivalence of Kan complexes, it follows that $\Bun_{G,0}^\nabla$ is a fibrant object in $\scH^{p I}$.
\qen
\end{example}

\subsection{Topological realisation}
\label{sec:Re -| S}

In this subsection we further build on and extend ideas from~\cite{Dugger:HoSheaves} to investigate a second way of obtaining a space from a simplicial presheaf on $\Cart$.
This time, we send a simplicial presheaf to a certain coend valued in topological spaces.

More precisely, we let $\Top$ denote the category of $\Delta$-generated topological spaces (see~\cite{Dugger:Delta-Top, Vogt:Conv_Spaces} for background).
We will be working with $\Top$ as our choice of category of topological spaces throughout; however, most of the theory in this paper also works with the category of Kelley spaces (also known as $k$-spaces; see, for instance,~\cite{Hovey:MoCats}), except for where we work explicitly with diffeological spaces (Lemma~\ref{st:Re = D on Dfg}, Remark~\ref{rmk:S_e neq Re for Dfg spaces}).

We provide some very compact background on $\Delta$-generated topological spaces.
A topological space $X$ is $\Delta$-generated precisely if its topology coincides with the final topology induced by all continuous maps $|\Delta^n| \to X$, for all $n \in \NN_0$.
(Here, $|\Delta^n|$ is the standard topological $n$-simplex.)
In particular, the category of $\Delta$-generated topological spaces and continuous maps, denoted $\Top$, is cartesian closed~\cite{Vogt:Conv_Spaces}.
The product in $\Top$, however, is not the usual product of topological spaces---one has to pass to the $\Delta$-generated topology after taking the usual product of topological spaces.

The category $\Top$ carries a cofibrantly generated model structure, having the same generating cofibrations and generating trivial cofibrations as the standard model structure on topological spaces~\cite{Dugger:Delta-Top, Haraguchi:Coreflective_subcats_of_MoCat, FR:Directed_HoThy, SYH:Ho_and_Coho_via_enr_BiFctrs}.
The geometric realisation functor $|{-}| \colon \sSet \to \Top$ takes values in $\Delta$-generated spaces, since $\Top$ is closed under colimits of topological spaces and contains $|\Delta^n|$ for each $n \in \NN_0$.
By construction of the model structure on $\Top$, the induced adjunction $|{-}| : \sSet \rightleftarrows \Top : \Sing$ is a Quillen adjunction ($|{-}|$ sends generating (trivial) cofibrations to (trivial) cofibrations).
Further, by the same proof as in~\cite[Lemma~3.18]{Hovey:MoCats}, it follows that $|{-}|$ preserves finite products.
Then, the proof of~\cite[Prop.~4.2.11]{Hovey:MoCats} applies as well, showing that $\Top$ is a symmetric monoidal (even cartesian) model category.
It also follows that $\Top$ is a simplicial model category and that $|{-}|$ is a monoidal left Quillen functor.
Finally, the Quillen adjunction $|{-}| \dashv \Sing$ is even a Quillen equivalence, since the inclusion of $\Top$ into Kelley spaces (or all topological spaces) is a Quillen equivalence~\cite{Dugger:Delta-Top}; the claim then follows from the two-out-of-three property of Quillen equivalences.

In this section, we provide a left Quillen equivalence $\scH^{p I} \to \Top$.
The main ideas for this section stem from~\cite{Dugger:HoSheaves}; there the full proof is technically rather involved.
Again, we circumvent these problems here by working over cartesian spaces rather than over the category of manifolds.

Let $\Dfg$ denote the category of diffeological spaces (see~\cite{Souriau:Groupes_Differentiels, IZ:Diffeology}); we use the conventions of~\cite{Bunk:Higher_Sheaves}, so that $\Dfg$ is the full subcategory of $\Cat(\Cart^\opp, \Set)$ on the concrete sheaves with respect to the Grothendieck coverage $\tau$ (see the beginning of Section~\ref{sec:R-local MoStrs}).
Concretely, a diffeological space can be defined as a pair $(X, \Plot_X)$, where $X \in \Set$, and where $\Plot_X$ assigns to every cartesian space $c \in \Cart$ a subset $\Plot_X(c) \subset \Set(c,X)$ of the maps from the underlying set of $c$ to $X$.
These maps are called \emph{plots of $X$} and have to satisfy that
\begin{myenumerate}
\item $\Plot_X(*) = X$ (every constant map is a plot),

\item for every $f \in \Cart(c,d)$ and every $g \in \Plot_X(d)$, we have that $g \circ f \in \Plot_X(c)$ (i.e.~$\Plot_X$ is a presheaf on $\Cart$), and

\item the presheaf $\Plot_X$ is a sheaf with respect to $\tau$.
\end{myenumerate}
We will often identify a diffeological space $(X,\Plot_X)$ with the sheaf it defines (see~\cite{Bunk:Higher_Sheaves} for more background), and we will denote this simply by $X$.

\begin{example}
For any manifold $M \in \Mfd$, the presheaf $\ul{M}$, given by $c \mapsto \Mfd(c,M)$, is a diffeological space.
In particular, this applies to every cartesian space $d \in \Cart$; for these, we have $\ul{d} = \scY_d$ as (pre)sheaves on $\Cart$.
\qen
\end{example}

\begin{definition}
\label{def:D}
Let $\rmD \colon \Dfg \to \Top$ be the functor defined as follows:
for $X \in \Dfg$, we let $\rmD X$ be the underlying set of the diffeological space $X \in \Dfg$, endowed with the final topology defined by its plots $c \to X$, where $c$ ranges over all cartesian spaces.
A morphism $f \in \Dfg(X,Y)$ is sent to the map it defines on the sets underlying $X$ and $Y$.
We call $\rmD$ the \emph{diffeological topology} functor and $\rmD(X)$ the \emph{underlying topological space of $X$}.
\end{definition}

The $\Delta$-generated topological spaces are in fact precisely those topological spaces that arise as the underlying topological spaces of diffeological spaces~\cite{SYH:Ho_and_Coho_via_enr_BiFctrs,CSW:D_Topology}.

\begin{proposition}
\emph{\cite{CW:HoThy_of_Dfg_Spaces,SYH:Ho_and_Coho_via_enr_BiFctrs}}
There exists an adjunction
\begin{equation}
\begin{tikzcd}
	\rmD : \Dfg \ar[r, shift left=0.14cm, "\perp"' yshift=0.05cm]
	& \Top : \rmC\,, \ar[l, shift left=0.14cm]
\end{tikzcd}
\end{equation}
where (under the embedding of $\Dfg$ into presheaves on $\Cart$) we have $\rmC(T)(c) = \Top(c, T)$ for any topological space $T$ and any cartesian space $c$.
\end{proposition}

The following proposition consists of results that can already be found in~\cite{CSW:D_Topology}; we only include the proofs here for completeness.

\begin{proposition}
\label{st:D on mfds and products}
Let $(-) \times (-)$ denote the product in $\Top$, and let $(-) \times^t (-)$ denote the usual product of topological spaces.
The functor $\rmD$ has the following properties:
\begin{myenumerate}
\item For any manifold $M$, the space $\rmD(\ul{M})$ coincides with the underlying topological space of $M$.

\item For any manifolds $M, N \in \Mfd$, the canonical maps $\rmD(\ul{M} \times \ul{N}) \to \rmD \ul{M} \times \rmD \ul{N} \to \rmD \ul{M} \times^t \rmD \ul{N}$ are homeomorphisms.

\item $\rmD \colon \Dfg \to \Top$ preserves finite products.
\end{myenumerate}
\end{proposition}

\begin{proof}
Part~(1) is~\cite[Example~3.7]{CW:HoThy_of_Dfg_Spaces}:
it is clear that any subset $U \subset M$ which is open in the manifold topology is also open in the diffeological topology.
Conversely, if $U$ is open in $\rmD(\ul{M})$, then its intersection with all images of charts of $M$ must be open.
As these images form a basis for the manifold topology, $U$ is open in the manifold topology.

Part~(2) follows readily from Part~(1) together with the fact that $M \times N$ is again a manifold.

Part~(3) is merely~\cite[Lemma~4.1]{CSW:D_Topology} and the remarks following that lemma.
For completeness, we fill in the details omitted there.
In~\cite{CSW:D_Topology} it is proven that the natural map $\rmD(X \times Y) \to \rmD X \times \rmD Y$ is a homeomorphism whenever $\rmD X$ is locally compact Hausdorff.
Since $\rmD \ul{c}$ is locally compact Hausdorff for any $c \in \Cart$, and since $\rmD$ preserves colimits, we have the following canonical isomorphisms in $\Top$:
let $X, Y \in \Dfg$ be arbitrary.
Using that $\Dfg$ and $\Top$ are cartesian closed, we compute
\begin{align}
	\rmD(X \times Y)
	&\cong \rmD \big( (\underset{\Cart/X}{\colim}^\Dfg \ul{c}) \times Y \big)
	\\
	&\cong \underset{\Cart/X}{\colim}^\Top \rmD (\ul{c} \times Y)
	\\
	&\cong (\underset{\Cart/X}{\colim}^\Top \rmD \ul{c}) \times \rmD Y
	\\
	&\cong \rmD X \times \rmD Y\,.
\end{align}
In the third isomorphism we have used the above-mentioned result~\cite[Lemma~4.1]{CSW:D_Topology}.
\end{proof}

\begin{remark}
We point out that we only use manifolds without boundary or corners here.
For manifolds with boundary, part~(1) of Proposition~\ref{st:D on mfds and products} fails---see, for instance, \cite[Cor.~4.47]{CW:HoThy_of_Dfg_Spaces}.
\qen
\end{remark}

Since each cartesian space $c \in \Cart$ is diffeomorphic to $\RN^n$ for some $n \in \NN_0$, and since $\RN^n$ is (isomorphic to) a CW complex for any $n \in \NN_0$, it follows that $\rmD \ul{c}$ is cofibrant in $\Top$ for every $c \in \Cart$.
We have the following version of~\cite[Prop.~2.3]{Dugger:Universal_HoThys}:

\begin{theorem}
There exists a Quillen adjunction $Re \dashv S$, sitting inside a weakly commutative diagram
\begin{equation}
\label{eq:Re -| S defining diagram}
\begin{tikzcd}[column sep=1.5cm, row sep=1.25cm]
	\Cart \ar[d, "\scY"'] \ar[r, "\rmD"]
	& \Top \ar[dl, shift left=0.1cm, "S"]
	\\
	\scH^p \ar[ur, shift left=0.1cm, "Re"]
\end{tikzcd}
\end{equation}
Further, there is a canonical natural isomorphism $Re \circ \scY \cong \rmD$.
\end{theorem}

\begin{proof}
The functor $Re$ is defined as the (enriched) left Kan extension of $\rmD$ along $\scY$ in digram~\eqref{eq:Re -| S defining diagram}.
Explicitly, we can write
\begin{align}
\label{eq:Re S explicit}
	Re(F) &= \int^{c \in \Cart} F(c) \otimes \rmD \ul{c}
	\\*
	&= \int^{c \in \Cart} |F(c)| \times \rmD \ul{c}\,,
	\\[0.2cm]
	S(T)(c) &= \Top(|\Delta^\bullet| \times \rmD \ul{c}, T)
	\\*
	&\cong \Sing(T^{\rmD \ul{c}})\,.
\end{align}
Since $\rmD c$ is cofibrant in $\Top$ and $\Sing \colon \Top \to \sSet$ is right Quillen, it follows that $S$ maps fibrations (resp.~trivial fibrations) in $\Top$ to objectwise fibrations (resp.~trivial fibrations) in $\scH$.
Thus, $S$ is right Quillen.

For the second claim, we observe the canonical isomorphisms
\begin{align}
\label{eq:Re on representables}
	Re (\scY_d) = \int^{c \in \Cart} \scY_d(c) \otimes \rmD \ul{c}
	\cong \int^{c \in \Cart} \Cart(c,d) \times \rmD \ul{c}
	\cong \rmD \ul{d}\,,
\end{align}
for all $d \in \Cart$.
The statement now follows from Proposition~\ref{st:D on mfds and products}(1).
\end{proof}

\begin{lemma}
The adjunction $Re \dashv S$ has the following properties:
\begin{myenumerate}
\item It is a simplicial adjunction.

\item $S$ is monoidal.
\end{myenumerate}
\end{lemma}

\begin{proof}
Part~(1) holds true since geometric realisation preserves finite products of simplicial sets and since the functor $K \otimes (-) \colon \Top \to \Top$ is a left adjoint, for any $K \in \sSet$.
Part~(2) holds true since $S$ is right adjoint and $\Top$ is cartesian monoidal.
\end{proof}

\begin{lemma}
\label{st:Re = D on Dfg}
Consider the fully faithful inclusion $\iota \colon \Dfg \hookrightarrow \Cat(\Cart^\opp, \Set) \hookrightarrow \scH$.
The diagram
\begin{equation}
\begin{tikzcd}
	\Dfg \ar[r, hookrightarrow, "\iota"] \ar[dr, "\rmD"'] & \scH \ar[d, "Re"]
	\\
	& \Top
\end{tikzcd}
\end{equation}
commutes up to natural isomorphism.
In particular, for any manifold $M \in \Mfd$, $Re \ul{M}$ is homeomorphic to the underlying topological space of $M$.
\end{lemma}

\begin{proof}
For $X \in \Dfg$, we have canonical natural isomorphisms
\begin{equation}
	Re \circ \iota(X) \cong Re \int^{c \in \Cart} \iota(X)(c) \otimes \scY_c
	\cong \int^{c \in \Cart} \iota(X)(c) \otimes \rmD \ul{c}
\end{equation}
and
\begin{equation}
	\rmD X \cong \rmD \int^{c \in \Cart} \iota(X)(c) \otimes \ul{c}
	\cong \int^{c \in \Cart} \iota(X)(c) \otimes \rmD\ul{c}\,.
\end{equation}
Combining this with Proposition~\ref{st:D on mfds and products} completes the proof.
\end{proof}

\begin{proposition}
\label{st:Re -| S on H^pI}
The pair $Re \dashv S$ induces a Quillen adjunction
\begin{equation}
\begin{tikzcd}
	Re : \scH^{p I} \ar[r, shift left=0.14cm, "\perp"' yshift=0.05cm]
	& \Top : S\,. \ar[l, shift left=0.14cm]
\end{tikzcd}
\end{equation}
\end{proposition}

\begin{proof}
By~\cite[Props.~3.1.6,~3.3.18]{Hirschhorn:MoCats}, it suffices to check that $Re$ sends the morphisms $\scY_c \times \scY_\RN \to \scY_c$ to weak equivalences.
Since $Re$ preserves products of representables (by Proposition~\ref{st:D on mfds and products} and~\eqref{eq:Re -| S defining diagram}), it even suffices to check that $Re$ sends the morphism $\scY_\RN \to *$ to a weak equivalence.
That it does so is evident from~\eqref{eq:Re -| S defining diagram}.
\end{proof}

\begin{proposition}
The functor $Re$ from diagram~\eqref{eq:Re -| S defining diagram} has the following properties:
\begin{myenumerate}
\item If $F \in \scH^p$ is cofibrant, $Re$ sends the morphism $F \times \scY_\RN \to F$ to a weak equivalence in $\Top$.

\item For each differentiably good open covering $\scU = \{c_a \to c\}_{a \in A}$ in $\Cart$, the functor $Re$ sends the \v{C}ech nerve $\cC \scU \to \scY_c$ to a weak equivalence in $\Top$.
\end{myenumerate}
\end{proposition}

\begin{proof}
Ad~(1):
The morphism is a weak equivalence in $\scH^{p I}$ between cofibrant objects by Proposition~\ref{st:tensor and ev weqs in H^(p/i I)}.
Therefore, the claim follows from Proposition~\ref{st:Re -| S on H^pI}.

Ad~(2):
Let $\cC\scU \to \scY_c$ denote the \v{C}ech nerve of the open covering $\scU$.
We view this as a morphism from a simplicial presheaf $\cC\scU$ to a simplicially constant presheaf $\scY_c$.
Since $\scU$ is a differentiably good open covering, $\cC\scU$ is levelwise a coproduct of representable presheaves on $\Cart$; hence, $\cC\scU$ is cofibrant in $\scH^p$.
By construction of the \v{C}ech model structure $\scH^{p\ell}$, the morphism $\cC\scU \to \scY_c$ is a weak equivalence in $\scH^{p \ell}$.
By Corollary~\ref{st:coincidence of model structures}, this is also a weak equivalence in $\scH^{p I}$.
The result now follows from Proposition~\ref{st:Re -| S on H^pI} and since both $\cC\scU$ and $\scY_c$ are cofibrant.
\end{proof}

We now prove an important property of the model categories $\scH^{p/i\, I}$ which allows us to detect $I$-local weak equivalences.
Dugger calls this property \emph{rigidity} in~\cite{Dugger:HoSheaves}.

\begin{proposition}
\label{st:rigidity}
\emph{\cite[Lemma~3.4.4]{Dugger:HoSheaves}}
Let $F,G \in \scH^{p/i\, I}$ be fibrant.
Then, a morphism $\psi \colon F \to G$ is an $I$-local weak equivalence if and only if $\psi_{|*} \colon F(*) \to G(*)$ is a weak equivalence in $\sSet$.
\end{proposition}

\begin{proof}
Since $\psi$ is a morphism between local objects in a left Bousfield localisation of $\scH^{p/i}$, it is an equivalence in $\scH^{p/i\, I}$ if and only if it is an objectwise weak equivalence in $\scH^{p/i}$.
For each $c \in \Cart$, the morphism $c \to * = \RN^0$ induces a commutative square
\begin{equation}
\begin{tikzcd}[column sep=1.cm, row sep=0.75cm]
	F(*) \ar[r, "e_{|F}"] \ar[d, "\psi_{|*}"']
	& F(c) \ar[d, "\psi_{|c}"]
	\\
	G(*) \ar[r, "e_{|G}"']
	& G(c)
\end{tikzcd}
\end{equation}
Since $F$ and $G$ are $\RN$-local, the claim now follows from Proposition~\ref{st:tensor and ev weqs in H^(p/i I)}.
\end{proof}

\begin{theorem}
\label{st:Re -| S is QEq}
There is a commutative diagram of simplicial Quillen equivalences
\begin{equation}
\label{eq:tc Re |-| diagram}
\begin{tikzcd}[column sep={1.75cm,between origins}, row sep={2.cm,between origins}]
	& \scH^{p I} \ar[rd, shift left=0.1cm, "Re"] \ar[ld, shift left=0.1cm, "\ev_*"] &
	\\
	\sSet \ar[rr, shift left=0.1cm, "|{-}|"] \ar[ur, shift left=0.1cm, "\tilde{\sfc}"]
	& & \Top \ar[ll, shift left=0.1cm, "\Sing"] \ar[ul, shift left=0.1cm, "S"]
\end{tikzcd}
\end{equation}
where $\tilde{\sfc}$, $Re$, and $|{-}|$ are the left adjoints.
\end{theorem}

\begin{proof}
It is well-established that the pair $|{-}| \dashv \Sing$ is a simplicial Quillen equivalence (see e.g.~\cite{Hovey:MoCats}).
We have also seen in Theorem~\ref{st:tc -| ev_* is QEq} that the adjunction $\tilde{\sfc} \dashv \ev_*$ is a simplicial Quillen equivalence.
The commutativity of~\eqref{eq:tc Re |-| diagram} follows from the definitions~\eqref{eq:Re S explicit} of the functors $Re$ and $S$, which use $|{-}|$ and $\Sing$, respectively.
The fact that $Re \dashv S$ is a Quillen equivalence then follows from the two-out-of-three property of Quillen equivalences.
\end{proof}

\begin{remark}
A slightly different version of Theorem~\ref{st:Re -| S is QEq} has been found previously in~\cite{Dugger:Universal_HoThys,Dugger:HoSheaves}, working over $\Mfd$ instead of $\Cart$.
We found that $\Cart$ has several technical advantages (in particular due to Corollary~\ref{st:coincidence of model structures}) and provides a sufficiently large category of parameter spaces to describe geometric and topological structures, as Theorem~\ref{st:Re -| S is QEq} shows (see also~\cite{Schreiber:DCCT} for various applications of this formalism).
\qen
\end{remark}

\section{The singular complex of a simplicial presheaf}
\label{sec:S_e and its properties}

In this section we introduce the \emph{smooth singular complex}, sometimes also called the \emph{concordance space}, of a simplicial presheaf on $\Cart$.
We investigate its homotopical properties---for instance, it sends smooth homotopies to simplicial homotopies---and we establish it both as a left Quillen equivalence $\scH^{p/i\, I} \to \sSet$ and as a right Quillen equivalence $\scH^{i I} \to \sSet$.

\subsection{Extended simplices and the smooth singular complex}
\label{sec:L_e -| S_e -| R_e construction}

In a fashion similar to motivic homotopy theory (see e.g.~\cite{MV:A1-HoThy_of_Schemes,Voevodsky:A1_HoThy,DLORV:Motivic_HoThy}), we consider the extended affine simplices in order to build our smooth singular complex functor.
However, we purely rely on the smooth manifold structure of the affine cartesian simplices rather than on their function algebras.

\begin{definition}
The \emph{extended $n$-simplex} is the cartesian space
\begin{equation}
	\Delta_e^n \coloneqq \Big\{ (t^0, \ldots, t^n) \in \RN^{n+1}\, \Big| \, \sum_{i = 0}^n t^i = 1 \Big\} \subset \RN^{n+1}\,.
\end{equation}
Face and degeneracy maps are defined as the affine linear extensions of the face and degeneracy maps of the standard simplices $|\Delta^n|$.
The extended simplices thus define a functor $\Delta_e \colon \bbDelta \to \Cart$.
\end{definition}

By construction, the topological standard simplex
\begin{equation}
	|\Delta^n| = \Big\{t \in \RN^{n+1}\, \Big| \, \sum_{i = 0}^n t^i = 1,\, 0 \leq t^i \leq 1\ \forall i = 0, \ldots, n \Big\}
\end{equation}
is a subset of the extended simplex $\Delta_e^n$, for any $n \in \NN_0$.
This inclusion $|\Delta^n| \hookrightarrow \Delta_e^n$ is compatible with the face and degeneracy maps.
Recalling the functor $\rmD \colon \Dfg \to \Top$ from Definition~\ref{def:D}, we see that there is a morphism
\begin{equation}
	\iota \colon |\Delta| \to \rmD \Delta_e
\end{equation}
of functors $\bbDelta \to \Top$.
In particular, the diagram
\begin{equation}
\label{eq:iota is natural}
\begin{tikzcd}[column sep=1.5cm, row sep=1.25cm]
	{|\Delta^n|} \ar[r, "\iota^n"] \ar[d, "|\Delta|(\sigma)"']
	& \rmD \Delta_e^n \ar[d, "\rmD \Delta_e(\sigma)"]
	\\
	{|\Delta^k|} \ar[r, "\iota^k"']
	& \rmD \Delta_e^k
\end{tikzcd}
\end{equation}
in $\Top$ commutes for every morphism $\sigma \in \bbDelta([n],[k])$.

The extended simplices functor $\Delta_e$ induces a Quillen adjunction
\begin{equation}
\begin{tikzcd}[column sep=1.5cm]
	\scH^p \ar[r, shift left=0.14cm, "1", "\perp"' yshift=0.05cm]
	& \scH^i \ar[r, shift left=0.14cm, "\Delta_e^*", "\perp"' yshift=0.05cm] \ar[l, shift left=0.14cm, "1"]
	& (\sSet^{\bbDelta^\opp})^i = (\sSet^{\bbDelta^\opp})^{Reedy}\,. \ar[l, shift left=0.14cm, "(\Delta_e)_*"]
\end{tikzcd}
\end{equation}
Here we have made use of~\cite[Thm.~15.8.7]{Hirschhorn:MoCats}, which implies that the injective model structure on bisimplicial sets agrees with the Reedy model structure.
We recall

\begin{theorem}
\emph{\cite[Thm.~5.2.3]{Riehl:Cat_HoThy}}
Let $\scM$ be a simplicial model category.
Then, the realisation functor
\begin{equation}
	|{-}|_\scM \colon \scM^{\bbDelta^\opp} \to \scM\,,
	\qquad
	X_\bullet \longmapsto \int^{[n] \in \bbDelta^\opp} \Delta^n \otimes X_n
\end{equation}
is a left Quillen functor with respect to the Reedy model category structure on $\scM^{\bbDelta^\opp}$.
\end{theorem}

\begin{proposition}
Let $\delta \colon \bbDelta \to \bbDelta \times \bbDelta$ be the diagonal functor.
There exists a canonical isomorphism
\begin{equation}
	|{-}|_\sSet \cong \delta^*\,,
\end{equation}
of functors $s\sSet \to \sSet$, where $\delta^*(X)_n = X_{n,n}$ is the pullback along the diagonal functor.
\end{proposition}

\begin{proof}
This is a standard application of the Yoneda Lemma in the (co)end calculus.
\end{proof}

\begin{corollary}
The diagonal functor is a left Quillen functor
\begin{equation}
	\delta^* \colon \big( \sSet^{(\bbDelta^\opp)} \big)^i \longrightarrow \sSet\,.
\end{equation}
In particular, it is homotopical, i.e.~it preserves all weak equivalences.
\end{corollary}

Consequently, we can define a left Quillen functor as the composition
\begin{equation}
\label{eq:def ev_(D_e)}
	S_e \coloneqq \delta^* \circ \Delta_e^* \colon \scH^{p/i} \longrightarrow \sSet\,.
\end{equation}
Consider a complete and cocomplete category $\scE$, two categories $\scC, \scD$, and a functor $F \colon \scC \to \scD$.
Recall that, in this situation, the functor $F^* \colon \Cat(\scD,\scE) \longrightarrow \Cat(\scC, \scE)$ has a left adjoint $F_!$ and a right adjoint $F_*$, which are given by the left and the right Kan extension along $F$.
By the construction of $S_e$ as a composition of pullback functors which act on categories of simplicial presheaves, we infer:

\begin{proposition}
\label{st:L_e -| S_e -| R_e}
The functor $S_e = \delta^* \circ \Delta_e^*$ has both adjoints.
We thus obtain a triple of adjunctions $L_e \dashv S_e \dashv R_e$, where $L_e$ and $R_e$ are given by the compositions
\begin{equation}
	L_e = \Delta_{e!} \circ \delta_!\,,
	\qqandqq
	R_e = \Delta_{e*} \circ \delta_*\,.
\end{equation}
The adjunction $S_e \dashv R_e$ is a simplicial Quillen adjunction.
\end{proposition}

\begin{definition}
\label{def:sm sing cplx}
We call the functor $S_e \colon \scH^{p/i} \to \sSet$ the \emph{smooth singular complex functor}.
For $F \in \scH$, the simplicial set $S_e F$ is called the \emph{smooth singular complex of $F$}.
\end{definition}

\subsection{$S_e$ as a left Quillen equivalence}
\label{sec:S_e is QEq}

We further investigate the homotopical properties of the smooth singular complex functor $S_e$.
So far, we know that the adjunction $S_e : \scH^{p/i} \rightleftarrows \sSet : R_e$ is Quillen.
Our goal here is to show that this Quillen adjunction descends to the localisation $\scH^{p/i\, I}$ and that there it even forms a Quillen equivalence.

\begin{definition}
Let $F,G \in \scH$ be two simplicial presheaves on $\Cart$, and let $f_0, f_1 \colon F \to G$ be a pair of morphisms.
A \emph{smooth homotopy from $f_0$ to $f_1$} is a commutative diagram
\begin{equation}
\label{eq:smooth homotopy}
\begin{tikzcd}[column sep=1.25cm, row sep=1cm]
	F \times \Delta^{\{0\}} \ar[d, hookrightarrow] \ar[dr, "f_0"] &
	\\
	F \times \ul{\RN} \ar[r, "h" description] & G
	\\
	F \times \Delta^{\{1\}} \ar[u, hookrightarrow] \ar[ur, "f_1"'] &
\end{tikzcd}
\end{equation}
in $\scH$, where the vertical inclusions are induced by the maps $* \to \RN$, given by $* \mapsto 0$ and $* \mapsto 1$.
\end{definition}

\begin{lemma}
\label{st:S_e: smooth and spl homotopies}
The functor $S_e \colon \scH \to \sSet$ maps smoothly homotopic morphisms to simplicially homotopic morphisms.
\end{lemma}

\begin{proof}
The projection $(t^0,t^1) \mapsto t^0$ yields a diffeomorphism $\psi \colon \Delta_e^1 \to \RN$ of cartesian spaces.
Observe that there is a morphism of simplicial sets
\begin{equation}
	\nu \colon \Delta^1 \longrightarrow S_e \ul{\RN} = \Cart(\Delta_e^\bullet, \RN)\,,
\end{equation}
defined by sending the generating non-degenerate 1-simplex of $\Delta^1$ to the 1-simplex $\psi$.
Hence, using the fact that $S_e$ preserves products, we apply $S_e$ to diagram~\eqref{eq:smooth homotopy} and augment it using $\nu$ to obtain a commutative diagram
\begin{equation}
\begin{tikzcd}[column sep=3cm, row sep=1cm]
	& S_e F \times \Delta^{\{0\}} \ar[d, hookrightarrow] \ar[dr, "S_e f_0"] \ar[dl, hookrightarrow] &
	\\
	S_e F \times \Delta^1 \ar[r, "1 \times \nu"]
	& S_e F \times S_e\RN \ar[r, "S_e h" description]
	& S_e G
	\\
	& S_e F \times \Delta^{\{1\}} \ar[u, hookrightarrow] \ar[ur, "S_e f_1"'] \ar[ul, hookrightarrow] &
\end{tikzcd}
\end{equation}
This establishes a simplicial homotopy $S_e h \circ (1_X \times \nu)$ from $S_e f_0$ to $S_e f_1$.
\end{proof}

Lemma~\ref{st:S_e: smooth and spl homotopies} can be seen as a generalisation of~\cite[Lemma~4.10]{CW:HoThy_of_Dfg_Spaces} from diffeological spaces to simplicial presheaves.
Indeed, the composition
\begin{equation}
\begin{tikzcd}
	\Dfg \ar[r, hookrightarrow, "\iota"] & \scH \ar[r, "S_e"] & \sSet
\end{tikzcd}
\end{equation}
is precisely the \emph{smooth singular functor} from~\cite{CW:HoThy_of_Dfg_Spaces}.

\begin{proposition}
\label{st:S_e(c to *) is weq}
For any $c \in \Cart$, the functor $S_e$ sends the collapse morphism $\sfc \colon \scY_c \to *$ to a weak equivalence in $\sSet$.
\end{proposition}

\begin{proof}
Let $c \in \Cart$, and let $x \in c$ be any point.
The inclusion $x \colon * \to c$ induces a smooth homotopy equivalence $* \rightleftarrows c$.
The functor $S_e$ maps this to a simplicial homotopy equivalence according to Lemma~\ref{st:S_e: smooth and spl homotopies}.
\end{proof}

\begin{corollary}
\label{st:ev_Delta adjunction descends}
The functor $S_e$ induces Quillen adjunctions
\begin{equation}
\begin{tikzcd}
	S_e : \scH^{p/i\, I} \ar[r, shift left=0.14cm, "\perp"' yshift=0.05cm]
	& \sSet : R_e\,. \ar[l, shift left=0.14cm]
\end{tikzcd}
\end{equation}
\end{corollary}

\begin{proof}
Each morphism $\scY_c \times \scY_\RN \to \scY_c$ in $I$ is a morphism between cofibrant objects in $\scH^{p/i}$.
Therefore, by~\cite[Prop.~3.3.18]{Hirschhorn:MoCats} it suffices to show that $S_e$ sends each morphism in $I$ to a weak equivalence in $\sSet$.
Since $S_e$ preserves products, it suffices to check that $S_e(\scY_\RN \to *)$ is a weak equivalence in $\sSet$.
The claim then follows from Proposition~\ref{st:S_e(c to *) is weq}.
\end{proof}

We call a functor $(\scC, \CW_\scC) \to (\scD, \CW_\scD)$ between relative categories \textit{homotopical} if it preserves weak equivalences (see also~\cite{DHKS:Holims_and_homotopical_Cats, Shulman:Ho(co)lims_and_enriched_HoThy, Riehl:Cat_HoThy} for more background on homotopical categories and functors).

\begin{proposition}
The functors $S_e \colon \scH^{p/i\, I} \to \sSet$ are homotopical.
\end{proposition}

\begin{proof}
$S_e \colon \scH^{i I} \to \sSet$ is homotopical because it is left Quillen and every object in $\scH^{i I}$ is cofibrant.
The corresponding statement for the projective model structure now follows from Proposition~\ref{st:same weqs}.
\end{proof}

Note, in particular, that by Proposition~\ref{st:coincidence of model structures} the functor $S_e$ also sends weak equivalences in the \v{C}ech local model structures $\scH^{p/i\, \ell}$ to weak equivalences in $\sSet$.

\begin{theorem}
\label{st:S_e -| R_e is QEq}
The Quillen adjunctions
\begin{equation}
\begin{tikzcd}
	S_e : \scH^{p/i\, I} \ar[r, shift left=0.14cm, "\perp"' yshift=0.05cm]
	& \sSet : R_e\, \ar[l, shift left=0.14cm]
\end{tikzcd}
\end{equation}
are Quillen equivalences.
\end{theorem}

\begin{proof}
We have Quillen adjunctions
\begin{equation}
\begin{tikzcd}
	\sSet \ar[r, shift left=0.14cm, "\perp"' yshift=0.05cm, "\tilde{\sfc}"]
	& \scH^{p/i\, I} \ar[r, shift left=0.14cm, "\perp"' yshift=0.05cm, "S_e"] \ar[l, shift left=0.14cm, "\ev_*"]
	& \sSet \ar[l, shift left=0.14cm, "R_e"]
\end{tikzcd}
\end{equation}
and we know from Theorem~\ref{st:tc -| ev_* is QEq} that the Quillen adjunction $\tilde{\sfc} \dashv \ev_*$ is even a Quillen equivalence.
We readily see that $S_e \circ \tilde{\sfc}$ is the identity functor on $\sSet$.
It follows that, $S_e \dashv R_e$ is a Quillen equivalence.
\end{proof}

\begin{corollary}
\label{st:S_e pres and reflects weqs}
The functor $S_e$ both preserves and reflects weak equivalences as a functor of relative categories $(\scH, W_I) \longrightarrow (\sSet, W_\sSet)$.
\end{corollary}

\begin{proof}
The left Quillen equivalence $S_e \colon \scH^{iI} \to \sSet$ preserves and reflects weak equivalences since every object in $\scH^{iI}$ is cofibrant (see e.g.~\cite[Prop.~1.3.16]{Hovey:MoCats}).
The statement on the level of relative categories then follows from Proposition~\ref{st:same weqs}.
\end{proof}

\begin{corollary}
\label{st:smooth hoeqs are weqs}
Any smooth homotopy equivalence in $\scH$ is a weak equivalence in $\scH^{p/i\, I}$.
\end{corollary}

\begin{proof}
By Proposition~\ref{st:S_e: smooth and spl homotopies}, $S_e$ sends smooth homotopy equivalences to simplicial homotopy equivalences, which are, in particular, weak equivalences in $\sSet$.
Thus, the claim follows from Corollary~\ref{st:S_e pres and reflects weqs}.
\end{proof}

\begin{remark}
\label{rmk:HoThy induced by S_e}
Let $W_\sSet$ denote the class of weak equivalences in $\sSet$, and let $S_e^{-1}(W_\sSet)$ denote the class of morphisms in $\scH$ whose image under $S_e$ is in $W_\sSet$.
Corollary~\ref{st:S_e pres and reflects weqs} lets us suspect that there is an equivalence of model categories
\begin{equation}
	\scH^{p/i\, I} \simeq L_{S_e^{-1}(W_\sSet)} \scH^{p/i}\,.
\end{equation}
Using properties of local weak equivalences in Bousfield localisations should allow us to prove that conjecture here already, but instead we give a very direct proof later in Theorem~\ref{st:HoThy induced by S_e}.
\qen
\end{remark}

\subsection{$S_e$ as a right Quillen equivalence}

The goal of this subsection is to establish the smooth singular functor as a right Quillen functor $S_e \colon \scH^{i I} \to \sSet$.
Apart from having convenient technical implications on the functor $S_e \colon \scH^{i I} \to \sSet$, the appearance of several intermediate model structures of bisimplicial sets sheds additional light on the functor $S_e$.
We already know from Proposition~\ref{st:L_e -| S_e -| R_e} that $S_e = \delta^* \circ \Delta_e^*$ has a left adjoint $L_e = 
\Delta_{e!} \circ \delta_!$.
We will show that both its constituting functors $\Delta_{e!}$ and $\delta_!$ are left Quillen functors.

\subsubsection{Model structures for $\infty$-groupoids on the category of bisimplicial sets}

We start by analysing the functor $\delta_!$ in more detail.
Let
\begin{equation}
\begin{tikzcd}
	\iota_n \colon \Sp_n \coloneqq \Delta^1 \underset{\Delta^0}{\sqcup} \cdots \underset{\Delta^0}{\sqcup} \Delta^1
	\ar[r, hookrightarrow]
	& \Delta^n
\end{tikzcd}
\end{equation}
denote the \emph{spine-inclusion} of the $n$-simplex $\Delta^n$, for $n \geq 1$.
(Note that for $n = 1$ the morphism $\iota_1$ is an isomorphism.)

We write $s\sSet = \Cat(\bbDelta^\opp, \sSet)$ for the category of bisimplicial sets.
There exists a bifunctor
\begin{equation}
	\boxtimes \colon \sSet \times \sSet \longrightarrow s\sSet\,,
	\qquad
	(K \boxtimes L)_{m,n} \coloneqq K_m \times L_n.
\end{equation}
We view a bisimplicial set $X$ as a simplicial diagram $m \mapsto X_{m,\bullet}$ in $\sSet$.
Let
\begin{equation}
\Sp \coloneqq \{\iota_n \boxtimes \Delta^0 \colon \Sp_n \boxtimes \Delta^0 \hookrightarrow \Delta^n \boxtimes \Delta^0\, | \, n \geq 1\}
\end{equation}
denote the set of all spine inclusions, viewed as maps of vertically constant bisimplicial sets.
Let $J \in \sSet$ denote the nerve of the groupoid with two objects and a unique isomorphism between them.
The following definitions are taken from~\cite{Rezk:HoTheory_of_HoTheories,Barwick:infty-n-Cat_as_closed_MoCat,Horel:MoStrs_on_internal_Cats}.

\begin{definition}
We define the following model structures on the category $s\sSet$ of bisimplicial sets:
\begin{myenumerate}
\item We view $s\sSet = \Cat(\bbDelta^\opp, \sSet)$ as endowed with the injective model structure.
Recall that this coincides with the Reedy model structure~\cite{Hirschhorn:MoCats}.

\item We let $\SSp \coloneqq L_\Sp s\sSet$ be the left Bousfield localisation of $s\sSet$ at the spine inclusions.
This is the \emph{model category for Segal spaces}.

\item
The \emph{model category for complete Segal spaces} is the localisation $\CSS \coloneqq L_{J \boxtimes \Delta^0} \SSp$.
\end{myenumerate}
\end{definition}

Let $L_{\Delta^\bullet \boxtimes \Delta^0} s\sSet$ denote the left Bousfield localisation of the injective model category of bisimplicial sets at all collapse morphisms $\{\Delta^n \boxtimes \Delta^0 \to \Delta^0 \boxtimes \Delta^0\, | \, n \in \NN_0\}$.
Let $L_{\bbDelta \boxtimes \Delta^0} s\sSet$ denote the left Bousfield localisation of $s\sSet$ at all morphisms $\{\varphi \boxtimes 1_{\Delta^0} \colon \Delta^n \boxtimes \Delta^0 \to \Delta^m \boxtimes \Delta^0\, | \, n,m \in \NN_0\}$.
(Compare these localisations to those in Proposition~\ref{st:equiv localisations of H}.)
Finally, let $L_{\Delta^1 \boxtimes \Delta^0} \SSp$ and $L_{\Delta^1 \boxtimes \Delta^0} \CSS$ denote the left Bousfield localisations of $\SSp$ and $\CSS$, respectively, at the morphism $\Delta^1 \boxtimes \Delta^0 \to \Delta^0 \boxtimes \Delta^0$.
We will mostly be using the model category $L_{\Delta^\bullet \boxtimes \Delta^0} s\sSet$, but for conceptual clarity and for an interpretation as model categories for $\infty$-groupoids, we include the following proposition.

\begin{proposition}
\label{st:L_J SSp = L_J CSS}
The following left Bousfield localisations yield identical model categories:
\begin{equation}
	L_{\bbDelta \boxtimes \Delta^0} s\sSet
	\overset{(1)}{=} L_{\Delta^\bullet \boxtimes \Delta^0} s\sSet
	\overset{(2)}{=} L_{\Delta^1 \boxtimes \Delta^0} \SSp
	\overset{(3)}{=} L_{\Delta^1 \boxtimes \Delta^0} \CSS\,.
\end{equation}
\end{proposition}

\begin{proof}
By Theorem~\ref{st:criterion for coincidence of MoStrs} it suffices to check that all four model categories have the same cofibrations and fibrant objects.
For cofibrations, this is trivial since each of the model categories is a left Bousfield localisation of $s\sSet$.
It thus remains to check that the fibrant objects of the three model categories coincide.

Identity (1) is a direct consequence of the two-out-of-three property of weak equivalences.

For identity (2), let $X \in L_{\Delta^\bullet \boxtimes \Delta^0} s\sSet$ be fibrant.
That is, $X$ is injective fibrant in $s\sSet$ and the canonical map $X_0 \to X_n$ is a weak equivalence in $\sSet$ for any $n \in \NN_0$.
We have to show that $X$ satisfies the Segal condition, i.e.~that for every $n \geq 2$ the spine inclusion $\Sp_n \hookrightarrow \Delta^n$ induces a weak equivalence
\begin{equation}
	X_n \longrightarrow X_1 \underset{X_0}{\times} \cdots \underset{X_0}{\times} X_1\,.
\end{equation}
(As pointed out in~\cite{Rezk:HoTheory_of_HoTheories}, the strict pullback is a homotopy pullback here because $X$ is Reedy fibrant.)
Consider the commutative diagram
\begin{equation}
\label{eq:Segal cond. in Delta localisation}
\begin{tikzcd}[column sep=2cm, row sep=1.25cm]
	X_n \ar[r]
	& X_1 \underset{X_0}{\times} \cdots \underset{X_0}{\times} X_1
	\\
	X_0 \ar[u, "s"] \ar[r, "{(1_{X_0}, \ldots, 1_{X_0})}"', "\cong"]
	& X_0 \underset{X_0}{\times} \cdots \underset{X_0}{\times} X_0 \ar[u, "s_0 \times \cdots \times s_0"']
\end{tikzcd}
\end{equation}
Since $X$ is Reedy fibrant, the pullbacks on the right-hand side are homotopy pullbacks.
Therefore, both vertical maps in~\eqref{eq:Segal cond. in Delta localisation} are weak equivalences.
It follows by the commutativity of the diagram that $X$ satisfies the Segal condition.
Then, $X$ is fibrant in $L_{\Delta^1 \boxtimes \Delta^0} \SSp$ since, by assumption, the morphism $X_0 \to X_1$ is a weak equivalence.

Conversely, if $X$ is fibrant in $L_{\Delta^1 \boxtimes \Delta^0} \SSp$, then the top horizontal morphism in diagram~\eqref{eq:Segal cond. in Delta localisation} is a weak equivalence because $X$ satisfies the Segal condition, and the right-hand vertical morphism is a weak equivalence because $X$ is injective fibrant and $X$ is local with respect to $\Delta^1 \boxtimes \Delta^0 \to \Delta^0 \boxtimes \Delta^0$.
It thus follows by the commutativity of the diagram that also the left vertical morphism is an equivalence, for any $n \geq 2$, so that $X$ is fibrant in $L_{\Delta^\bullet \boxtimes \Delta^0} s\sSet$.

For identity (3), recall that in any Segal space $X$ there is a notion of when a morphism $f \in X_1$ is invertible (or a `homotopy equivalence' in the language of~\cite{Rezk:HoTheory_of_HoTheories}).
One defines the \emph{space $X_{weq}$ of homotopy equivalences} in $X$ to be the union of those connected components of $X_1$ that contain invertible morphisms (by~\cite[Lemma~5.8]{Rezk:HoTheory_of_HoTheories}, if $X \in \SSp$ is fibrant, then a connected component of $X_1$ contains a homotopy equivalence if and only if it consists purely of homotopy equivalences).
For any Segal space, the degeneracy morphism $s_0 \colon X_0 \to X_1$ factors as
\begin{equation}
\label{eq:completeness cond diag}
\begin{tikzcd}
	& X_{weq} \ar[d, hookrightarrow, "\iota_X"]
	\\
	X_0 \ar[ur, "\widehat{s_0}"] \ar[r, "s_0"'] & X_1
\end{tikzcd}
\end{equation}

Since every fibrant object in $\CSS$ is also fibrant in $\SSp$, this implies that every fibrant object in $L_{\Delta^1 \boxtimes \Delta^0}\CSS$ is fibrant in $L_{\Delta^1 \boxtimes \Delta^0}\SSp$.

Conversely, let $Y \in \SSp$ be fibrant.
Then, $Y$ is local in $L_{\Delta^1 \boxtimes \Delta^0}\SSp$ precisely if the morphism $s_0 \colon Y_0 \to Y_1$ is a weak equivalence.
We need to show that, in that case, the morphism $\widehat{s_0} \colon Y_0 \to Y_{weq}$ is a weak equivalence in $\sSet$.
However, since $\iota_Y \colon Y_{weq} \to Y$ is the inclusion of a union of connected components of $Y_1$, diagram~\eqref{eq:completeness cond diag} and the fact that $s_0$ is a weak equivalence imply that $\iota_Y$ hits every connected component of $Y_1$.
Therefore, $\iota_Y$ is a weak equivalence; it follows from the two-out-of-three property that $\widehat{s_0} \colon Y_0 \to Y_{weq}$ is a weak equivalence as well.
\end{proof}

\begin{remark}
Let $X \in \SSp$ be fibrant.
Since $\iota_X \colon X_{weq} \to X_1$ is the inclusion of a union of connected components of $X_1$, it follows that $\iota_X$ is a weak equivalence precisely if it is an isomorphism, i.e.~precisely if $X_{weq} = X_1$.
In other words, a fibrant object in both $L_{\Delta^1 \boxtimes \Delta^0} \SSp$ and $L_{\Delta^1 \boxtimes \Delta^0} \CSS$ is a complete Segal space with all 1-morphisms invertible.
In that sense, the fibrant objects in these model categories are $\infty$-groupoids.
The model category $L_{\bbDelta \boxtimes \Delta^0} s\sSet$ can be seen as the model category of essentially constant simplicial diagrams of spaces, in analogy to how $L_\Cart \scH^{p/i}$ describes essentially constant simplicial presheaves.
\qen
\end{remark}

\begin{remark}
\label{rmk:Diagonal MoStr and localisations}
The model structures in Proposition~\ref{st:L_J SSp = L_J CSS} further agree with the \emph{diagonal model structure} on bisimplicial sets (also known as the Moerdijk model structure): these model categories have the same underlying categories, the same cofibrations, and they also have the same fibrant objects (see, for instance,~\cite[Sec.~IV.3.3]{GJ:Spl_HoThy}).
Thus, the claimed equality follows from Theorem~\ref{st:criterion for coincidence of MoStrs}.
\qen
\end{remark}

We recall the following classical fact (see, for example, ~\cite[Sec.~IV.3.3]{GJ:Spl_HoThy}); 

\begin{proposition}
\label{st:delta_! -| delta^* is QAd}
The diagonal $\delta^* \colon s\sSet \to \sSet$ induces a Quillen adjunction
\begin{equation}
\begin{tikzcd}
	\delta_! : \sSet \ar[r, shift left=0.14cm, "\perp"' yshift=0.05cm]
	& L_{\Delta^\bullet \boxtimes \Delta^0} s\sSet : \delta^*\,. \ar[l, shift left=0.14cm]
\end{tikzcd}
\end{equation}
\end{proposition}

From now on, we will understand the adjunction $\delta_! \dashv \delta^*$ as the above Quillen adjunction.
There is another Quillen adjunction that relates $L_{\Delta^\bullet \boxtimes \Delta^0} s\sSet$ to the model category of simplicial sets, in analogy with Theorem~\ref{st:tc -| ev_* is QEq}.

\begin{proposition}
\label{st:c_D -| ev_[0] is QEq}
Consider the adjoint pair $\sfc_\bbDelta : \sSet \rightleftarrows s\sSet : \ev_{[0]}$, where $\sfc_\bbDelta = \Delta^0 \boxtimes (-)$, and where $\ev_{[0]}(X) = X_{0,\bullet}$.
This satisfies:
\begin{myenumerate}
\item $\sfc_\bbDelta \dashv \ev_{[0]}$ is a Quillen adjunction $\sSet \rightleftarrows s\sSet$.

\item Composing the Quillen adjunction from~(1) with the localisation adjunction $s\sSet \rightleftarrows L_{\Delta^\bullet \boxtimes \Delta^0} s\sSet$ yields a Quillen equivalence
\begin{equation}
\begin{tikzcd}
	\sfc_\bbDelta : \sSet \ar[r, shift left=0.14cm, "\perp"' yshift=0.05cm]
	& L_{\Delta^\bullet \boxtimes \Delta^0} s\sSet : \ev_{[0]}\,. \ar[l, shift left=0.14cm]
\end{tikzcd}
\end{equation}
\end{myenumerate}
\end{proposition}

\begin{proof}
It is straightforward to see that $\sfc_\bbDelta \colon \sSet \to s\sSet$ preserves cofibrations and further preserves as well as reflects weak equivalences.
This proves claim~(1).
Part~(2) proceeds entirely in parallel to the proofs of Lemma~\ref{st:cK -> Cc^p/i cK is objwise weq} and Theorem~\ref{st:tc -| ev_* is QEq}.
\end{proof}

\begin{corollary}
The functor $\sfc_\bbDelta \colon \sSet \longrightarrow L_{\Delta^\bullet \boxtimes \Delta^0} s\sSet$ induces a functor on the underlying relative categories and detects weak equivalences.
\end{corollary}

\subsubsection{The functors $\Delta_{e!}$ and $L_e$}

Next, we show that $\Delta_{e!} \colon s\sSet \to \scH^{i I}$ is left Quillen, and that the Quillen adjunction $\Delta_{e!} \dashv \Delta_{e*}$ descends to the localisation $L_{\bbDelta \boxtimes \Delta^0} s\sSet = L_{\Delta^\bullet \boxtimes \Delta^0} s\sSet$.
The functor $\Delta_{e!}$ acts as
\begin{equation}
	\Delta_{e!}(X) = \int^n X_{n,\bullet} \otimes \Delta_e^n\,.
\end{equation}
In particular, for bisimplicial sets in the image of $(-) \boxtimes (-)$ we find that
\begin{align}
\label{eq:D_e!(K bt L)}
	\Delta_{e!}(K \boxtimes L) &= \int^n K_n \otimes L \otimes \Delta_e^n
	\\
	&\cong \Big( \int^n K_n \otimes \Delta_e^n \Big) \otimes L
	\\
	&\cong \big( \Delta_{e!}(K \boxtimes \Delta^0) \big) \otimes L\,.
\end{align}
It follows that
\begin{equation}
\label{eq:D_e!(D^0 bt L)}
	\Delta_{e!}(\Delta^0 \boxtimes L) \cong \tilde{\sfc}L
\end{equation}
for any $L \in \sSet$, where $\tilde{\sfc} \colon \sSet \to \scH$ is the constant-presheaf functor.

\begin{lemma}
\label{st:pD_e^n to D_e^n is inj cof}
For any $n \in \NN_0$, the morphism $\Delta_{e!}(\partial \Delta^n \boxtimes \Delta^0 \longrightarrow \Delta^n \boxtimes \Delta^0)$ is a cofibration in $\scH^i$.
\end{lemma}

\begin{proof}
For $n=0,1$ this is straightforward.
Consider the presentation of $\partial \Delta^n$ as a coequaliser,
\begin{equation}
	\partial \Delta^n \cong \coeq \bigg( \coprod_{0 \leq i < j \leq n} \Delta^{n-2} \rightrightarrows \coprod_{0 \leq k \leq n} \Delta^{n-1} \bigg)\,.
\end{equation}
Since $\Delta_{e!}$ preserves colimits, we obtain
\begin{equation}
\label{eq:def partial Delta_e^n}
	\partial \Delta_e^n \coloneqq \Delta_{e!}(\partial \Delta^n \boxtimes \Delta^0)
	\cong \coeq \bigg( \coprod_{0 \leq i < j \leq n} \Delta_e^{n-2} \rightrightarrows \coprod_{0 \leq k \leq n} \Delta_e^{n-1} \bigg)\,.
\end{equation}
The colimit is taken in $\scH$ (not in $\Dfg$, even though all $\Delta_e^k$ are diffeological spaces), and so we have
\begin{equation}
	\partial \Delta_e^n(c)
	\cong \coeq \bigg( \coprod_{0 \leq i < j \leq n} \Delta_e^{n-2}(c) \rightrightarrows \coprod_{0 \leq k \leq n} \Delta_e^{n-1}(c) \bigg)
\end{equation}
for any $c \in \Cart$.
In particular, any section of $\partial \Delta_e^n$ over $c \in \Cart$ comes from some section $f \in \Delta_e^{n-1}(c)$ of a face of $\Delta_e^n$.
Two such sections $f,g \in \Delta_e^{n-1}(c)$ are identified precisely if they factor through the copy of $\Delta_e^{n-2}$ that joins the respective faces of $\Delta_e^n$ and if, further, $f$ and $g$ agree as maps $c \to \Delta_e^{n-2}$.

Let $f,g \in \partial \Delta_e^n(c)$ be any two elements, and assume that $\iota_e^n \circ f = \iota_e^n \circ g$, where $\iota_e^n \colon \partial \Delta_e^n \longrightarrow \Delta_e^n$ is the canonical morphism.
Observe that $\iota_e^n$ is injective as a map on the underlying sets $\partial \Delta_e^n(*) \hookrightarrow \Delta_e^n(*)$.
Since every section $f \colon \scY_c \to \partial \Delta_e^n$ is, in particular, a map $\scY_c(*) \to \partial \Delta_e^n(*)$ of the underlying sets, and analogously a section $\scY_c \to \Delta_e^n$ is, in particular, a map $\scY_c(*) \to \Delta_e^n(*)$ of underlying sets, it follows that $\iota_e^n \colon \partial \Delta_e^n \to \Delta_e^n$ is an objectwise monomorphism.
\end{proof}

\begin{remark}
We point out that $\partial \Delta_e^n$, as defined in~\eqref{eq:def partial Delta_e^n}, is not a diffeological space for $n \geq 2$.
For instance, consider a differentiably good open covering $c = c_0 \cup c_1$ of a cartesian space $c$.
We denote the intersection $c_0 \cap c_1$ by $c_{01} \in \Cart$.
Let $f_i \colon c_i \to \Delta_e^{n-1}$ be smooth maps, for $i = 0,1$, to adjacent faces of $\Delta_e^n$, such that $f_{i|c_{01}} \colon c_{01} \to \Delta_e^{n-2}$ factors through the $n{-}2$-simplex which joins the two faces.
These data to not lift to a section $f \in \partial \Delta_e^n(c)$, since such an $f$ must factor through only \textit{one} of the faces $\partial \Delta_e^{n-1}$.
That is, $\partial \Delta_e^n$ does not satisfy the sheaf condition.
\qen
\end{remark}

Let $\scC, \scD, \scE$ be categories.
Recall the notion of an \emph{adjunction of two variables} $\scC \times \scD \to \scE$ (see, for example, \cite[Def.~4.1.12]{Hovey:MoCats}).
We will denote an adjunction of two variables only by its tensor functor $\otimes \colon \scC \times \scD \to \scE$.
If $\scE$ has pushouts, then there is an induced \emph{pushout product}, or \emph{box product}, on morphisms:
given morphisms $f \colon A \to B$ in $\scC$ and $g \colon X \to Y$ in $\scD$, their pushout product (relative to $\otimes$) is the induced morphism in $\scE$ given by 
\begin{equation}
\begin{tikzcd}
	A \otimes Y \underset{A \otimes X}{\sqcup} B \otimes X
	\ar[r, "f \square g"]
	& B \otimes Y\,.
\end{tikzcd}
\end{equation}
We recall following definitions:

\begin{definition}
\emph{\cite[Def.~4.2.1]{Hovey:MoCats}}
Let $\scC, \scD, \scE$ be model categories, and let $\otimes \colon \scC \times \scD \to \scE$ be an adjunction of two variables.
Then, $\otimes$ is a \emph{Quillen adjunction of two variables} if the induced pushout product $f,g \mapsto f \square g$ satisfies the \emph{pushout-product axiom}:
\begin{myenumerate}
\item if both $f$ and $g$ are cofibrations, then so is $f \square g$, and

\item if, \emph{in addition}, $f$ or $g$ is a weak equivalence, then so is $f \square g$.
\end{myenumerate}
\end{definition}

\begin{definition}
\label{def:monoidal MoCat}
\emph{\cite[Def.~4.2.6]{Hovey:MoCats}}
A \emph{(symmetric) monoidal model category} is a closed (symmetric) monoidal category $(\scC, \otimes)$ together with a model structure on the underlying category $\scC$ such that:
\begin{myenumerate}
\item the closed monoidal structure is a Quillen adjunction of two variables $\otimes \colon \scC \times \scC \to \scC$.

\item Let $u \in \scC$ be the monoidal unit, and let $q_u \colon Q^\scC u \weq u$ be a cofibrant replacement.
Then, tensoring with any cofibrant object from the left or the right sends $q_u$ to a weak equivalence.
\end{myenumerate}
\end{definition}

\begin{example}
The model category $s\sSet$ with the injective model structure is cartesian (i.e.~symmetric monoidal with $\otimes = \times$).
Similarly, each of the model categories $\scH^{p/i}$, $\scH^{p/i\, \ell}$, and $\scH^{p/i\, I}$ is cartesian by Proposition~\ref{st:basic props of H^pi and its locs}.
In each of the monoidal model structures we encounter here, the monoidal unit is already cofibrant, so that the second axiom of Definition~\ref{def:monoidal MoCat} is trivially satisfied.
\qen
\end{example}

The injective model structure on $s\sSet$ is cofibrantly generated (see e.g.~\cite{Rezk:HoTheory_of_HoTheories}), with generating cofibrations
\begin{equation}
	\CI = \big\{ (\partial \Delta^n \boxtimes \Delta^0 \hookrightarrow \Delta^n \boxtimes \Delta^0)
	\square (\Delta^0 \boxtimes \partial \Delta^m \hookrightarrow \Delta^0 \boxtimes \Delta^m)
	\, \big| \, n,m \in \NN_0 \big\}
\end{equation}
and generating trivial cofibrations
\begin{equation}
	\CJ = \big\{ (\partial \Delta^n \boxtimes \Delta^0 \hookrightarrow \Delta^n \boxtimes \Delta^0)
	\square (\Delta^0 \boxtimes \Lambda^m_k \hookrightarrow \Delta^0 \boxtimes \Delta^m)
	\, \big| \, m,n \in \NN_0,\, 0 \leq k \leq m \big\}\,.
\end{equation}

\begin{proposition}
There is a Quillen adjunction
\begin{equation}
\begin{tikzcd}
	\Delta_{e!} : s\sSet \ar[r, shift left=0.14cm, "\perp"' yshift=0.05cm]
	& \scH^i : \Delta_e^*\,. \ar[l, shift left=0.14cm]
\end{tikzcd}
\end{equation}
\end{proposition}

\begin{proof}
We have already seen in Lemma~\ref{st:pD_e^n to D_e^n is inj cof} that $\Delta_{e!}$ sends the morphism $\partial \Delta^n \boxtimes \Delta^0 \hookrightarrow \Delta^n \boxtimes \Delta^0$ to the injective cofibration $\partial \Delta_e^n \hookrightarrow \Delta_e^n$ in $\scH^i$.
Further, it follows from Equations~\eqref{eq:D_e!(K bt L)} and~\eqref{eq:D_e!(D^0 bt L)} that
\begin{equation}
	\Delta_{e!}(\Delta^0 \boxtimes \partial \Delta^m \hookrightarrow \Delta^0 \boxtimes \Delta^m)
	= \big( \tilde{\sfc}\partial \Delta^m \longrightarrow \tilde{\sfc} \Delta^m \big)\,,
\end{equation}
which is an injective cofibration, and that
\begin{equation}
	\Delta_{e!}(\Delta^0 \boxtimes \Lambda^m_k \hookrightarrow \Delta^0 \boxtimes \Delta^m)
	= \big( \tilde{\sfc}\Lambda^m_k \longrightarrow \tilde{\sfc} \Delta^m \big)\,,
\end{equation}
which is an injective trivial cofibration.

It suffices the check that $\Delta_{e!}$ sends the generating (trivial) cofibrations of $s\sSet$ to (trivial) cofibrations in $\scH^i$ (see, for instance,~\cite[Lemma~2.1.20]{Hovey:MoCats}).
For $n, m \in \NN_0$, the pushout-product
\begin{equation}
	(\partial \Delta^n \boxtimes \Delta^0 \hookrightarrow \Delta^n \boxtimes \Delta^0)
	\square (\Delta^0 \boxtimes \partial \Delta^m \hookrightarrow \Delta^0 \boxtimes \Delta^m)
\end{equation}
is the canonical morphism
\begin{equation}
\begin{tikzcd}
	\big( (\partial \Delta^n \boxtimes \Delta^0) \otimes (\Delta^0 \boxtimes \Delta^m ) \big)
	\hspace{-1cm} \underset{\big( (\partial \Delta^n \boxtimes \Delta^0) \otimes (\Delta^0 \boxtimes \partial \Delta^m) \big) }{\sqcup}
	\hspace{-1cm} \big( (\Delta^n \boxtimes \Delta^0) \otimes (\Delta^0 \boxtimes \partial \Delta^m) \big)
	\ar[r]
	& (\Delta^n \boxtimes \Delta^0) \otimes (\Delta^0 \boxtimes \Delta^m)
\end{tikzcd}
\end{equation}
Note that, for $A, B, C, D \in \sSet$, there is a canonical natural isomorphism
\begin{equation}
	(A \boxtimes B) \otimes (C \boxtimes D)
	\cong (A \otimes C) \boxtimes (B \otimes D)
\end{equation}
in $s\sSet$.
Thus, the above pushout-product is canonically isomorphic to to morphism
\begin{equation}
\begin{tikzcd}
	(\partial \Delta^n \boxtimes \Delta^m)
	\underset{(\partial \Delta^n \boxtimes \partial \Delta^m)}{\sqcup}
	(\Delta^n \boxtimes \partial \Delta^m)
	\ar[r]
	& (\Delta^n \boxtimes \Delta^m)
\end{tikzcd}
\end{equation}
Since the functor $\Delta_{e!}$ preserves pushouts, and using~\eqref{eq:D_e!(K bt L)}, this is sent to the morphism
\begin{equation}
\begin{tikzcd}
	(\partial \Delta^n_e \otimes \tilde{\sfc}\Delta^m)
	\underset{(\partial \Delta^n_e \otimes \tilde{\sfc} \partial \Delta^m)}{\sqcup}
	(\Delta^n_e \otimes \partial \tilde{\sfc} \Delta^m)
	\ar[r]
	& (\Delta^n_e \otimes \tilde{\sfc} \Delta^m)\,,
\end{tikzcd}
\end{equation}
which coincides with the pushout product $(\partial \Delta^n_e \hookrightarrow \Delta^n_e) \square (\tilde{\sfc} \partial \Delta^m \hookrightarrow \tilde{\sfc} \Delta^m)$ in $\scH^i$.
By Lemma~\ref{st:pD_e^n to D_e^n is inj cof} and the fact that the model category $\scH^i$ is symmetric monoidal (even cartesian), this pushout-product is a cofibration in $\scH^i$.
Thus, $\Delta_{e!}$ sends the generating cofibrations of $\sSet$ to cofibrations.
An analogous argument shows that it also sends the generating trivial cofibrations to trivial cofibrations.
\end{proof}

Composing with the localisation $\scH^i \leftrightarrows \scH^{i I}$, we obtain a Quillen adjunction
\begin{equation}
\begin{tikzcd}
	\Delta_{e!} : s\sSet \ar[r, shift left=0.14cm, "\perp"' yshift=0.05cm]
	& \scH^{i I} : \Delta_e^*\,. \ar[l, shift left=0.14cm]
\end{tikzcd}
\end{equation}

\begin{proposition}
\label{st:D_e! -| D_e^* is QAd}
The Quillen adjunction $\Delta_{e!} : s\sSet \leftrightarrows \scH^{i I} : \Delta_e^*$ descends to a Quillen adjunction
\begin{equation}
\begin{tikzcd}
	\Delta_{e!} : L_{\Delta^\bullet \boxtimes \Delta^0} s\sSet \ar[r, shift left=0.14cm, "\perp"' yshift=0.05cm]
	& \scH^{i I} : \Delta_e^*\,. \ar[l, shift left=0.14cm]
\end{tikzcd}
\end{equation}
\end{proposition}

\begin{proof}
This is a direct consequence of the fact that $\Delta_{e!}(\Delta^n \boxtimes \Delta^0) \cong \Delta_e^n$.
The collapse morphism $\Delta_e^n \to *$ is a weak equivalence in $\scH^{i I}$ by Proposition~\ref{st:tensor and ev weqs in H^(p/i I)}.
\end{proof}

\begin{corollary}
The adjunction $L_e : \sSet \rightleftarrows \scH^{i I} : S_e$ is a Quillen adjunction.
\end{corollary}

\begin{proof}
This is a direct consequence of Proposition~\ref{st:delta_! -| delta^* is QAd} and Proposition~\ref{st:D_e! -| D_e^* is QAd}.
\end{proof}

\begin{theorem}
\label{st:D_e! -| D_e^* is QEq}
There is a commutative triangle of Quillen equivalences
\begin{equation}
\label{eq:}
\begin{tikzcd}[column sep={2cm,between origins}, row sep={2.5cm,between origins}]
	& L_{\Delta^\bullet \boxtimes \Delta^0} s\sSet \ar[rd, shift left=0.1cm, "\Delta_{e!}"] \ar[ld, shift left=0.1cm, "\ev_{[0]}"] &
	\\
	\sSet \ar[rr, shift left=0.1cm, "\tilde{\sfc}"] \ar[ur, shift left=0.1cm, "\sfc_\bbDelta"]
	& & \scH^{i I} \ar[ll, shift left=0.1cm, "\ev_*"] \ar[ul, shift left=0.1cm, "\Delta_e^*"]
\end{tikzcd}
\end{equation}
where $\sfc_\bbDelta$, $\Delta_{e!}$, and $\tilde{\sfc}$ are the left adjoints.
\end{theorem}

\begin{proof}
We know from Theorem~\ref{st:tc -| ev_* is QEq} that the bottom adjunction is a Quillen equivalence, and we know from Proposition~\ref{st:c_D -| ev_[0] is QEq} that the left diagonal adjunction is a Quillen equivalence.
Further, it is evident that the diagram of right adjoints commutes strictly.
Thus, the claim follows from the two-out-of-three property of Quillen equivalences.
\end{proof}

\begin{lemma}
\label{st:nat trafo ev_* -> S_e}
Consider the functors $\ev_*, S_e \colon \scH \to \sSet$.
\begin{myenumerate}
\item There is a canonical natural transformation $\gamma \colon \ev_* \to S_e$.

\item The right derived natural transformation of $\gamma$ is a natural weak equivalence.
\end{myenumerate}
\end{lemma}

\begin{proof}
Consider first the functors $\Delta_e^*, \sfc_\bbDelta \circ \ev_* \colon \scH^{p/i\, I} \longrightarrow s\sSet$.
For any $F \in \scH$, the collapse map $\Delta_e^n \to *$ induces a morphism $\widehat{\gamma}_{|F,n} \colon F(*) \to F(\Delta_e^n)$ of simplicial sets.
Since $* \in \Cart$ is final, this induces a natural transformation $\widehat{\gamma} \colon \sfc_\bbDelta \circ \ev_* \longrightarrow \Delta_e^*$.
Applying the diagonal functor $\delta^*$ to this natural transformation, we obtain a natural transformation
\begin{equation}
	\gamma \colon \ev_* = \delta^* \circ \sfc_\bbDelta \circ \ev_* \longrightarrow \delta^* \circ \Delta_e^* = S_e\,.
\end{equation}
This shows part~(1).
Part~(2) then follows from the fact that, whenever $F \in \scH^{p/i\, I}$ is fibrant, the morphism $F(*) \to F(\Delta_e^n)$ is a weak equivalence for every $n \in \NN_0$.
Therefore, if $F$ is fibrant, then $\widehat{\gamma}_{|F} \colon \sfc_\bbDelta(F(*)) \longrightarrow \Delta_e^*F$ is an objectwise weak equivalence in $s\sSet$.
The claim now follows from the fact that the diagonal functor $\delta^*$ is homotopical.
\end{proof}

\begin{theorem}
\label{st:L_e -| S_e is QEq}
The Quillen adjunction
\begin{equation}
\begin{tikzcd}
	L_e : \sSet \ar[r, shift left=0.14cm, "\perp"' yshift=0.05cm]
	& \scH^{i I} : S_e\,. \ar[l, shift left=0.14cm]
\end{tikzcd}
\end{equation}
is a Quillen equivalence.
\end{theorem}

\begin{proof}
Lemma~\ref{st:nat trafo ev_* -> S_e} provides a natural weak equivalence between the right derived functors of $\ev_*$ and $S_e$.
By Theorem~\ref{st:tc -| ev_* is QEq}, the functor $\ev_*$ is a right Quillen equivalence, and it follows that so is $S_e$.
\end{proof}

\begin{corollary}
\label{st:delta_! -| delta^* is QEq}
The Quillen adjunction 
\begin{equation}
\begin{tikzcd}
	\delta_! : \sSet \ar[r, shift left=0.14cm, "\perp"' yshift=0.05cm]
	& L_{\Delta^1 \boxtimes \Delta^0} s\sSet : \delta^*\,. \ar[l, shift left=0.14cm]
\end{tikzcd}
\end{equation}
is a Quillen equivalence.
\end{corollary}

\begin{proof}
This follows from the fact that $L_e = \Delta_{e!} \circ \delta_!$, together with Theorem~\ref{st:D_e! -| D_e^* is QEq} and the two-out-of-three property for Quillen equivalences.
\end{proof}

Corollary~\ref{st:delta_! -| delta^* is QEq} becomes particularly interesting in light of Proposition~\ref{st:L_J SSp = L_J CSS}:
it establishes a Quillen equivalence between the Kan-Quillen model structure on simplicial sets and each of the model structures in Proposition~\ref{st:L_J SSp = L_J CSS}.
In other words, Corollary~\ref{st:delta_! -| delta^* is QEq} shows that each of the model categories from Proposition~\ref{st:L_J SSp = L_J CSS} is a model category for $\infty$-groupoids.

\section{Comparison of spaces constructed from simplicial presheaves}
\label{sec:comparison of spaces from presheaves}

In Sections~\ref{sec:L_I H and first spaces from presheaves} and~\ref{sec:S_e and its properties} we have seen several ways of extracting a space from a simplicial presheaf on $\Cart$.
The main goal of Section~\ref{sec:comparison results} is to establish comparisons between the resulting spaces.
We use these to show that $S_e$ is a model for the homotopy colimit of $\Cart^\opp$-shaped diagrams of simplicial sets in Section~\ref{sec:S_e as hocolim}.
In Section~\ref{sec:Approx Thm for mfds} we show that applying $S_e$ to the presheaf associated to a manifold reproduces the homotopy type of the topological space underlying the manifold.

\subsection{Comparison results}
\label{sec:comparison results}

The right adjoint of $S_e = \delta^* \circ \Delta_e^*$ is given as $R_e = \Delta_{e*} \circ \delta_*$.
We start by making this functor more explicit:
consider a simplicial set $K \in \sSet$ and a cartesian space $c \in \Cart$.
Since the adjunction $S_e \dashv R_e$ is simplicial, there are natural isomorphisms
\begin{align}
\label{eq:R_e explicit}
	(R_e K)(c) &\cong \ul{\scH}(\scY_c, R_e K)
	\\*
	&\cong \ul{\sSet}(S_e \scY_c, K)
	\\*
	&\cong K^{S_e (\scY_c)}\,.
\end{align}
The following lemma is then immediate:

\begin{lemma}
\label{st:ev_* S = 1}
There exist canonical natural isomorphisms
\begin{equation}
	S_e \circ \tilde{\sfc} \cong 1_\sSet\,,
	\qqandqq
	\ev_* \circ R_e \cong 1_\sSet\,.
\end{equation}
\end{lemma}

It follows that there exists a natural isomorphism $\ev_* \circ R_e \circ \Sing \cong \Sing$.

\begin{lemma}
\label{st:Sing and exponential}
There is an isomorphism
\begin{equation}
	\Sing(T)^K \cong \Sing \big( T^{|K|} \big)\,,
\end{equation}
natural in both $T \in \Top$ and in $K \in \sSet$.
\end{lemma}

\begin{proof}
Since the adjunction $|{-}| \dashv \Sing$ is simplicial (because $|{-}|$ preserves finite products), we have binatural isomorphisms
\begin{align}
	\Sing(T)^K &= \ul{\sSet} \big( K, \Sing(T) \big)
	\\*
	&\cong \ul{\Top}(|K|, T)
	\\*
	&= \Sing(T^{|K|})\,.
\end{align}
Here we have used that $\Top$ is cartesian closed and simplicially enriched.
\end{proof}

Recall the diffeological topology functor from Definition~\ref{def:D}.

\begin{lemma}
\label{st:mps S_e -> Sing D, |-| S_e -> D}
For every manifold $M \in \Mfd$ there exist morphisms
\begin{equation}
	\varphi_M \colon S_e \ul{M} \longrightarrow \Sing (\rmD \ul{M})
	\qqandqq
	\psi_M \colon |{-}| \circ S_e (\ul{M}) \longrightarrow \rmD \ul{M}
\end{equation}
of simplicial sets and of topological spaces, respectively.
These assemble into natural transformations
\begin{equation}
	\varphi \colon S_e \longrightarrow \Sing \circ \rmD
	\qqandqq
	\psi \colon |{-}| \circ S_e \longrightarrow \rmD
\end{equation}
of functors $\Mfd \to \sSet$ and $\Mfd \to \Top$, respectively.
\end{lemma}

\begin{proof}
Let $n \in \NN_0$ and consider the set $\ul{M}(\Delta_e^n) = \Mfd(\Delta_e^n, M)$; it is the set of all smooth maps $\Delta_e^n \to M$ of manifolds.
Recall the morphism $\iota^\bullet \colon |\Delta^\bullet| \longrightarrow \rmD \Delta_e^\bullet$ of cosimplicial topological spaces from~\eqref{eq:iota is natural}.
If $f \colon \Delta_e^n \to M$ is any smooth map, then the composition $f \circ \iota^n \colon |\Delta^n| \to \rmD \ul{M}$ is continuous.
Here we have used that $\rmD \ul{M}$ coincides with the underlying topological space of the manifold $M$ (see Proposition~\ref{st:D on mfds and products}).
This provides a map
\begin{equation}
	\varphi_{M|n} \colon \ul{M}(\Delta_e^n) = \Mfd(\Delta_e^n, M)
	\overset{\rmD}{\longrightarrow} \Top(\rmD \ul{\Delta}_e^n, \rmD \ul{M})
	\overset{(\iota_n)^*}{\longrightarrow} \Top( |\Delta^n|, \rmD \ul{M})\,.
\end{equation}
Since $\iota^\bullet$ is a morphism of cosimplicial topological spaces, and since the maps $\varphi_{M|n}$ are defined by precomposition by $\iota^n$, it readily follows that $\varphi_{M|n}$ is natural in both $M \in \Mfd$ and $n \in \bbDelta$.
Thus, we obtain the desired morphism of simplicial sets
\begin{equation}
	\varphi_M \colon S_e \ul{M} \longrightarrow \Sing (\rmD \ul{M})\,.
\end{equation}
The composition
\begin{equation}
\label{eq:psi_M}
	\psi_M \colon |S_e \ul{M}| \xrightarrow{|\varphi_M|} |\Sing (\rmD \ul{M})|
	\xrightarrow[\sim]{e_M} \rmD \ul{M}
\end{equation}
then defines the morphism $\psi_M$, where $e \colon |{-}| \circ \Sing \weq 1_\Top$ is the counit of the Quillen equivalence $|{-}| \dashv \Sing$.
\end{proof}

\begin{lemma}
\label{st:nat weqs |-| S_e sim D on Cart}
The restrictions of $\varphi$ and $\psi$ to $\Cart \subset \Mfd$ are natural weak equivalences of functors $\Cart \to \sSet$ and $\Cart \to \Top$, respectively.
\end{lemma}

\begin{proof}
This follows readily from the observation that, for any cartesian space $c \in \Cart$, both $|S_e \scY_c|$ and $\rmD (\scY_c)$ are weakly equivalent to $* \in \Top$.
Hence, by the two-out-of-three property of weak equivalences any morphism $|S_e \scY_c| \longrightarrow \rmD (\scY_c)$ is a weak equivalence.
\end{proof}

\begin{proposition}
\label{st:weq S --> R_e o Sing}
There exists a natural weak equivalence
\begin{equation}
\begin{tikzcd}[column sep=2cm]
	\Top
	\ar[r, bend left=20, "R_e \circ \Sing", ""{name=U, inner sep=0pt, below, pos=.5}]
	\ar[r, bend right=20, "S"', ""{name=V, inner sep=0pt, above, pos=.5}]
	& \scH^{p/i\, I}\,.
	\arrow[Rightarrow, from=V, to=U, "\eta", "\sim"', shorten <= 2, shorten >= 2]
\end{tikzcd}
\end{equation}
\end{proposition}

\begin{proof}
It suffices to consider the projective case since $\scH^{p I}$ and $\scH^{i I}$ have the same weak equivalences (Proposition~\ref{st:same weqs}).
Given a topological space $T \in \Top$, by Equation~\eqref{eq:R_e explicit} and Lemma~\ref{st:Sing and exponential} we have
\begin{equation}
	R_e \circ \Sing (T)(c) \cong \Sing \big( T^{|S_e \scY_c|} \big)
	\qqandqq
	S(T)(c) \cong \Sing \big( T^{\rmD \ul{c}} \big)\,.
\end{equation}
The natural morphisms $\psi$ from Lemma~\ref{st:mps S_e -> Sing D, |-| S_e -> D} induce a morphism
\begin{equation}
	\eta_{|T} \coloneqq \Sing \big( T^{\psi} \big) \colon S(T) \longrightarrow R_e \circ \Sing (T)
\end{equation}
in $\scH$, which is natural in $T \in \Top$.

In order to see that the morphism $\Sing (T^{\psi})$ is a weak equivalence in $\scH^{p I}$, we first observe that both $S(T)$ and $R_e \circ \Sing(T)$ are fibrant in $\scH^{p I}$.
We have that
\begin{equation}
	\ev_* \big( S(T) \big) = \Sing(T)
	\qqandqq
	\ev_* \big( R_e \circ \Sing(T) \big)
	= \Sing \big( T^{|S_e (*)|} \big)
	= \Sing (T)\,.
\end{equation}
Further, the morphism $\psi_{|*} \colon |S_e(*)| = * \to * = \rmD(*)$ is the identity, so that also the morphism
\begin{equation}
	\Big( \Sing \big( T^{\psi} \big) \Big)_{|*} \colon S(T)(*) = \Sing(T) \longrightarrow \Sing(T) = R_e \circ \Sing (T)(*)
\end{equation}
is the identity.
The fact that $\eta$ is a natural weak equivalence now follows from the fact that the right Quillen equivalence $\ev_* \colon \scH^{p I} \to \sSet$ reflects weak equivalences between fibrant objects (which was also the content of Proposition~\ref{st:rigidity}).
\end{proof}

\begin{corollary}
There is a natural weak equivalence
\begin{equation}
\begin{tikzcd}[column sep=2cm]
	\Top
	\ar[r, bend left=25, "\Sing", ""{name=U, inner sep=0pt, below, pos=.5}]
	\ar[r, bend right=25, "S_e \circ S"', ""{name=V, inner sep=0pt, above, pos=.5}]
	& \sSet\,.
	\arrow[Rightarrow, from=V, to=U, "\eta'", "\sim"', shorten <= 2, shorten >= 2]
\end{tikzcd}
\end{equation}
\end{corollary}

\begin{proof}
The functor $S_e \colon \scH^{p/i\, I} \to \sSet$ is homotopical.
Applying it to the natural weak equivalence from Proposition~\ref{st:weq S --> R_e o Sing}, we obtain a natural weak equivalence
\begin{equation}
	S_e \eta \colon S_e \circ S \weq S_e \circ R_e \circ \Sing\,.
\end{equation}
Let $e \colon S_e \circ R_e \longrightarrow 1_\sSet$ denote the counit of the adjunction $S_e \dashv R_e$.
The fact that every object in $\sSet$ is cofibrant, together with the fact that $S_e \dashv R_e$ is a Quillen equivalence imply that the morphism $e_{|K} \colon S_e \circ R_e(K) \longrightarrow K$ is a weak equivalence in $\sSet$ for every fibrant simplicial set $K$.
Since $\Sing \colon \Top \to \sSet$ takes values in fibrant simplicial sets, it follows that the composition
\begin{equation}
	\eta' \colon S_e \circ S \xrightarrow[\sim]{S_e \eta} S_e \circ R_e \circ \Sing \xrightarrow[\sim]{e_{\Sing}} \Sing
\end{equation}
is a natural weak equivalence.
\end{proof}

\begin{corollary}
Let $Q^p \colon \scH^{p I} \to \scH^{p I}$ be a cofibrant replacement functor, with associated natural weak equivalence $q^p \colon Q^p \to 1_{\scH}$.
There is a zig-zag of natural weak equivalences
\begin{equation}
\label{eq:S sim Sing Re Q zig-zag}
	S_e \ \xleftarrow[\sim]{S_e q^p} \ S_e \circ Q^p \ \xrightarrow[\sim]{\eta''} \ \Sing \circ Re \circ Q^p
\end{equation}
of functors $\scH^{p I} \longrightarrow \sSet$.
\end{corollary}

\begin{proof}
The left-facing natural transformation is a weak equivalence since $S_e$ is homotopical.
The right-facing morphism is the composition
\begin{equation}
	\eta'' \colon S_e \circ Q^p \xrightarrow{S_e u_{Q^p}} S_e \circ S \circ Re \circ Q^p
	\xrightarrow[\sim]{(\eta')_{Re Q^p}} \Sing \circ Re \circ Q^p\,,
\end{equation}
where $u \colon 1_\Top \longrightarrow S \circ Re$ is the unit morphism of the adjunction $Re \dashv S$.
Since this adjunction is a Quillen equivalence and since every object in $\Top$ is fibrant, it follows that $u_F \colon F \to S \circ Re(F)$ is a weak equivalence for every cofibrant object $F \in \scH^{p I}$.
\end{proof}

\begin{proposition}
\label{st:|-| S sim Re Q zig-zag}
There exists a zig-zag of natural weak equivalences
\begin{equation}
\label{eq:|-| S_e sim Re Q zig-zag}
	|{-}| \circ S_e \ \xleftarrow[\sim]{|S_e (q^p)|} \ |{-}| \circ S_e \circ Q^p \ \xrightarrow[\sim]{\eta'''} \ Re \circ Q^p
\end{equation}
of functors $\scH^{p I} \longrightarrow \Top$.
In particular, there exists a natural isomorphism of total left derived functors between homotopy categories
\begin{equation}
\begin{tikzcd}[column sep=2cm]
	\rmh \scH^{p I}
	\ar[r, bend left=25, "\bbL Re", ""{name=U, inner sep=0pt, below, pos=.5}]
	\ar[r, bend right=25, "\rmh |{-}| \circ \rmh S_e"', ""{name=V, inner sep=0pt, above, pos=.5}]
	& \rmh \Top\,.
	\arrow[Rightarrow, from=V, to=U, "\eta''' ", "\cong"', shorten <= 2, shorten >= 2]
\end{tikzcd}
\end{equation}
\end{proposition}

Observe that $S_e$ and $|{-}|$ are already homotopical, so that we do not need to precompose them by a cofibrant replacement in order to obtain their total left derived functors.

\begin{proof}
We readily obtain a zig-zag as in~\eqref{eq:|-| S_e sim Re Q zig-zag} by applying the functor $|{-}|$ to the zig-zag~\eqref{eq:S sim Sing Re Q zig-zag} and then postcomposing by the counit $e \colon |{-}| \circ \Sing \weq 1_\Top$.
\end{proof}

\begin{remark}
There is an alternative way of obtaining a zig-zag as in~\eqref{eq:|-| S_e sim Re Q zig-zag} directly and explicitly, when $Q^p$ is Dugger's cofibrant replacement functor for $\scH^p$~\cite{Dugger:Universal_HoThys}.
It sends a simplicial presheaf $F$ to the two-sided bar construction
\begin{equation}
	Q^p F = B^{\scH}(F,\Cart,\scY)\,,
\end{equation}
in the notation of~\cite{Riehl:Cat_HoThy}.
(The superscript indicates in which simplicial category we are forming the bar construction.)
Using that $Re$ is simplicial and commutes with colimits, and that there is a natural isomorphism $Re \circ \scY_c \cong \rmD \ul{c}$ for any cartesian space $c \in \Cart$ (cf.~Lemma~\ref{st:Re = D on Dfg}), we obtain a canonical isomorphism
\begin{equation}
	Re \circ Q^p (F) \cong B^\Top (F, \Cart, \rmD)\,.
\end{equation}
Now we use that the morphism $\psi$ from Lemma~\ref{st:mps S_e -> Sing D, |-| S_e -> D} induces a natural weak equivalence $\psi \colon |{-}| \circ S_e \weq \rmD$ of functors $\Cart \to \Top$ (cf.~Lemma~\ref{st:nat weqs |-| S_e sim D on Cart}).
Since each of the functors $F \colon \Cart^\opp \to \sSet$ and $\rmD, |{-}| \circ S_e \colon \Cart \to \Top$ takes values in cofibrant objects, \cite[Cor.~5.2.5]{Riehl:Cat_HoThy} implies that $\psi$ induces a weak equivalence
\begin{equation}
	B^\Top(-, \Cart, \psi) \colon B^\Top (-, \Cart, |{-}| \circ S_e) \weq B^\Top(-, \Cart, \rmD) = Re \circ Q^p
\end{equation}
of functors $\scH \to \Top$.

On the other hand, since both $|{-}|$ and $S_e$ are left adjoints, we have a natural isomorphism
\begin{align}
	B^\Top (F, \Cart, |{-}| \circ S_e)
	&\cong |{-}| \circ S_e \big( B^{\scH}(F,\Cart,\scY) \big)
	\\
	&= |{-}| \circ S_e \circ Q^p (F)\,.
\end{align}
Now, the morphism $q^p \colon Q^p \weq 1_{\scH}$, together with the fact that both $|{-}|$ and $S_e$ preserve weak equivalences, yield the claim.
\qen
\end{remark}

\begin{remark}
\label{rmk:S_e neq Re for Dfg spaces}
Recall the embedding $\iota \colon \Dfg \hookrightarrow \scH$ of diffeological spaces into simplicial presheaves.
By Lemma~\ref{st:Re = D on Dfg}, the composition $Re \circ \iota$ agrees with the functor $\rmD \colon \Dfg \to \Top$ that sends a diffeological space to its underlying topological space, whose topology is induced by its plots.
It is interesting to ask whether the homotopy type of $\rmD X$ agrees with that of the smooth singular complex $S_e \iota(X)$ of $X$, for any diffeological space $X \in \Dfg$.
So far, however, we only see that the homotopy type of $S_e \iota X$ agrees with the cobar construction
\begin{align}
	S_e \iota(X) &\simeq B^\Top(X,\Cart, \rmD)
	\\*
	&= \int^n |\Delta^n| \times \Big( \coprod_{c_0, \ldots, c_n \in \Cart} \rmD \ul{c}_0 \times \Cart(c_0, c_1) \times \Cart(c_{n-1}, c_n) \times X(c_n) \Big)
\end{align}
rather than with the underlying topological space $\rmD X$ of $X$.
This is in accordance with---and maybe provides some further insight to---results from~\cite{CSW:D_Topology, OT:Smooth_approx_of_Diff} that the smooth singular complex of a diffeological space $X$ is not in general equivalent to the smooth singular complex of $\rmD X$.
\qen
\end{remark}

\subsection{Homotopy colimits of smooth spaces}
\label{sec:S_e as hocolim}

We can interpret the functor $Re$---and because of Proposition~\ref{st:|-| S sim Re Q zig-zag} also the functor $S_e$---in the context of the cohesive $\infty$-topos $\bH$ of presheaves of spaces on $\Cart$ as follows (see~\cite{Schreiber:DCCT} for more background and~\cite{Bunk:String} for a brief introduction).
From the proof of Proposition~\ref{st:|-| S sim Re Q zig-zag} we see that there are canonical weak equivalences
\begin{align}
	Re \circ Q^p (F) &\cong B^\Top (F, \Cart, \rmD)
	\simeq B^\Top (F, \Cart, *)
	\cong \big| B^\sSet (F, \Cart, *) \big|\,.
\end{align}
Using the fact that the topological realisation of a bisimplicial set is independent of which simplicial direction one realises first (up to canonical isomorphism), we obtain canonical weak equivalences
\begin{align}
	\big| B^\sSet (F, \Cart, *) \big|
	&\cong \big| B^\sSet (*, \Cart^\opp, F) \big|
	\\
	&\simeq \hocolim^\Top \big( |{-}| \circ F \colon \Cart^\opp \longrightarrow \Top \big)\,.
\end{align}
Combining this with Proposition~\ref{st:|-| S sim Re Q zig-zag}, we obtain

\begin{theorem}
Each of the functors
\begin{equation}
	Re \circ Q^p \simeq |{-}| \circ S_e \circ Q^p \simeq |{-}| \circ S_e
\end{equation}
models the homotopy colimit of diagrams $\Cart^\opp \to \Top$.
That is, there is a natural isomorphism between each of these functors and $\hocolim_{\Cart^\opp}$ as functors
\begin{equation}
	\Fun(\Cart^\opp, \Top) \longrightarrow \rmh \Top\,.
\end{equation}
It follows that $S_e$ models the homotopy colimit of diagrams $\Cart^\opp \longrightarrow \sSet$.
\end{theorem}

Therefore, each of these functors presents the $\infty$-categorical colimit functor for diagrams of spaces indexed by $\Cart^\opp$.
Consequently, on the level of the underlying $\infty$-categories they each present left-adjoints to the functor that sends a space to the constant presheaf whose value is that space.
This means that both $Re$ and $S_e$ provide explicit presentations for the left-adjoint $\Pi$ in the three-fold adjunction which implements the cohesive structure on $\bH$ (see~\cite{Schreiber:DCCT}).
A different argument for the case of $S_e$ has been given in~\cite{BEBdBP:Class_sp_of_oo-sheaves} (see the bottom of p.~2); here we establish a presentation of the $\infty$-categorical adjunction $\Pi \dashv \tilde{\sfc}$ as a Quillen adjunction $\scH^{p/i} \rightleftarrows \sSet$ and a Quillen equivalence $\scH^{p/i\,I} \rightleftarrows \sSet$.

\subsection{Recovering homotopy types of manifolds}
\label{sec:Approx Thm for mfds}

Recall the fully faithful embedding $\Mfd \hookrightarrow \scH$, $M \mapsto \ul{M}$.
In this section, we show that, for any manifold $M$, the smooth singular complex $S_e \ul{M}$ and the ordinary singular complex $\Sing (\rmD \ul{M})$ of $M$ are canonically weakly equivalent.
In fact, this is a well-known and classical result in geometric topology.
We nevertheless present this statement and a short proof in a language suited to this paper as a consistency check for the properties of the functor $S_e$.
Classical proofs of this statement can be found, for instance, in~\cite[Thm.~18.7]{Lee:Smooth_Mfds} and \cite[Secs.~5.31, 5.32]{Warner:Mfds_and_LieGrps} (a version with cubes instead of simplices is in~\cite{Fulton:Algebraic_Topolpogy}).

\begin{theorem}
\label{st:|S_e M| sim DM}
The natural transformations
\begin{equation}
\begin{tikzcd}[column sep=2cm]
	\Mfd
	\ar[r, bend left=25, "\Sing \circ \rmD \circ \ul{(-)}", ""{name=U, inner sep=0pt, below, pos=.5}]
	\ar[r, bend right=25, "S_e \circ \ul{(-)}"', ""{name=V, inner sep=0pt, above, pos=.5}]
	& \sSet\,.
	\arrow[Rightarrow, from=V, to=U, "\varphi", shorten <= 2, shorten >= 2]
\end{tikzcd}
\qquad
\begin{tikzcd}[column sep=2cm]
	\Mfd
	\ar[r, bend left=25, "\rmD \circ \ul{(-)}", ""{name=U, inner sep=0pt, below, pos=.5}]
	\ar[r, bend right=25, "{|{-}| \circ S_e \circ \ul{(-)}}"', ""{name=V, inner sep=0pt, above, pos=.5}]
	& \Top\,.
	\arrow[Rightarrow, from=V, to=U, "\psi", shorten <= 2, shorten >= 2]
\end{tikzcd}
\end{equation}
introduced in Lemma~\ref{st:mps S_e -> Sing D, |-| S_e -> D} are natural weak equivalences.
In particular, the smooth singular complex of $\ul{M}$ has the same homotopy type as $\Sing M$.
\end{theorem}

\begin{proof}
It suffices to prove the claim for $\varphi$, since $\psi$ is the composition of $|\varphi|$ with a weak equivalence in $\Top$, and $|{-}|$ preserves and detects weak equivalences.
Let $\scU = \{U_a\}_{a \in A}$ be a differentiably good open covering of $M$.
Then, the canonical morphism $\pi \colon \cC\scU \to \ul{M}$ from the \v{C}ech nerve of $\scU$ to $\ul{M}$ is a weak equivalence in $\scH^{p/i\, \ell}$:
the morphism $\cC \scU_0 = \coprod_{a \in A} \ul{U}_a \longrightarrow \ul{M}$ is a local epimorphism, or generalised cover (in the sense of~\cite[p.~7]{DHI:Hypercovers_and_sPShs}) with respect to the Grothendieck topology of differentiably good open coverings on $\Cart$~\cite{FSS:Cech_diff_char_classes_via_L_infty}.
Hence, the claim follows by~\cite[Cor.~A.3]{DHI:Hypercovers_and_sPShs}.
Therefore, $\pi$ becomes a weak equivalence under $S_e$ by Corollaries~\ref{st:coincidence of model structures} and~\ref{st:S_e pres and reflects weqs}.

We can equivalently write
\begin{equation}
	S_e \cC\scU = \underset{\bbDelta^\opp}{\hocolim}^\sSet (S_e \cC_\bullet \scU)\,,
\end{equation}
where $\cC_n \scU = \coprod_{a_0, \ldots, a_n \in A} U_{a_0 \cdots a_n}$ is the $n$-th level of the \v{C}ech nerve of $\scU$, and where we view the simplicial presheaf $\cC \scU$ as a diagram
\begin{equation}
	\bbDelta^\opp \longrightarrow \scH\,,
	\quad
	[n] \longmapsto \sfc_\bbDelta (\cC_n \scU)\,,
\end{equation}
of simplicially constant presheaves on $\Cart$.
We thus arrive at a commutative diagram
\begin{equation}
\begin{tikzcd}[column sep=2cm, row sep=1cm]
	S_e \cC \scU = \underset{\bbDelta^\opp}{\hocolim}^\sSet (S_e \cC_\bullet \scU)
	\ar[r, "\hocolim (\varphi)"] \ar[d, "S_e \pi"']
	& \underset{\bbDelta^\opp}{\hocolim}^\sSet \big( \Sing \circ \rmD (\cC_\bullet \scU) \big)
	\ar[d]
	\\
	S_e \ul{M}
	\ar[r, "\varphi_M"']
	& \Sing \circ \rmD (\ul{M})
\end{tikzcd}
\end{equation}
in $\sSet$.
We have already argued that the left-hand vertical morphism is a weak equivalence.
The top horizontal morphism is a weak equivalence as well:
for each $[n] \in \bbDelta$, the morphism $\varphi$ induces a morphism between simplicial sets which are a disjoint union of contractible components.
On the level of connected components, $\varphi$ induces a bijection, and hence is a weak equivalence.
Finally, the classical Nerve Theorem (see e.g.~\cite{Borsuk:Systems_of_compacta, Leray:Homol_of_loc_const_sp, Segal:Cats_and_Coho_Thys, DI:Top_hypercovers_and_A1}) implies that the right-hand vertical morphism is an equivalence as well.
\end{proof}

\section{Local fibrant replacement, concordance and mapping spaces}
\label{sec:Concordance}

In this section, we present a fibrant replacement functor in the model structures $\scH^{p/i\, I}$.
Its construction is motivated by the \emph{concordance sheaf} construction from~\cite{BEBdBP:Class_sp_of_oo-sheaves}.
This functor allows us to compute mapping spaces in $\scH^{p/i\, I}$ in Theorem~\ref{st:map spaces via concordance}.
We start by presenting the fibrant replacement functor:

\begin{lemma}
\label{st:exp pres smooth hos}
Suppose that $F_0, F_1 \in \scH$ and that $h \colon F_0 \times \ul{\RN} \to F_1$ is a smooth homotopy between morphisms $f,g \colon F_0 \to F_1$.
Then, for any $G \in \scH$, there is a smooth homotopy $\tilde{h} \colon G^{F_1} \times \ul{\RN} \to G^{F_0}$ from $G^f$ to $G^g \colon G^{F_1} \to G^{F_0}$.
\end{lemma}

\begin{proof}
Applying the exponential functor $G^{(-)}$ to the morphism $h$, we obtain a morphism
\begin{equation}
	G^h \in \scH(G^{F_1}, G^{F_0 \times \ul{\RN}}) \cong \scH(G^{F_1}, (G^{F_0}){}^{\ul{\RN}})\,.
\end{equation}
Using the internal-hom adjunction of $\scH$ then yields a morphism $\tilde{h} = (G^h)^\dashv \in \scH(G^{F_1} \times \ul{\RN}, G^{F_0})$ with the desired properties.
\end{proof}

We consider the following construction:
let $\ul{\Delta}_e^k \in \scH$ denote the simplicial presheaf represented by the extended affine simplex $\Delta_e^k \in \Cart$.
This provides a functor
\begin{equation}
	\ul{\Delta}_e \colon \bbDelta \to \scH\,,
	\qquad
	[k] \longmapsto \ul{\Delta}_e^k = \scY_{\Delta_e^k}\,.
\end{equation}
Given an object $F \in \scH$, we can compose this functor by the functor $F^{(-)} \colon \scH \to \scH$, obtaining an object $F^{\ul{\Delta}_e} \in \Cat(\bbDelta^\opp, \scH)$.
Equivalently, we can view this as a bisimplicial presheaf
\begin{equation}
	F^{\ul{\Delta}_e} \colon \Cart^\opp \longrightarrow s\sSet\,,
	\qquad
	c \longmapsto F(\Delta_e \times c)\,,
\end{equation}
which we can now compose by the diagonal functor $\delta^* \colon s\sSet \to \sSet$ to obtain a new simplicial presheaf on $\Cart$.
Putting everything together, this defines a functor
\begin{equation}
	\delta^* \circ (-)^{\ul{\Delta}_e} \colon \scH \to \scH\,,
	\qquad
	F \longmapsto \delta^* \circ F^{\ul{\Delta}_e}\,.
\end{equation}
The collapse morphisms $\Delta_e^k \to *$ induce a natural transformation $\ul{\Delta}_e \to *$ of functors $\bbDelta \to \scH$ (this even consists of $I$-local equivalences by Proposition~\ref{st:tensor and ev weqs in H^(p/i I)}).
From this we obtain a natural transformation
\begin{equation}
	\gamma \colon 1_{\scH} \longrightarrow \delta^* \circ (-)^{\ul{\Delta}_e}\,.
\end{equation}

Now, let $R^{p/i} \colon \scH \to \scH$ be a fibrant replacement functor for the projective (resp.~injective) model structure, with natural objectwise weak equivalence $r^{p/i} \colon 1_{\scH} \weq R^{p/i}$.
(Observe, in particular, that a fibrant replacement functor $R^p$ for the projective model structure can be obtained by postcomposition with a fibrant replacement functor in $\sSet$, i.e.~we can use $R^p(F) = R^\sSet \circ F$ for $F \in \scH$.
An explicit model for an injective fibrant replacement functor is given in Appendix~\ref{app:Injective fib rep}.)
We define functors
\begin{equation}
	\Cc^{p/i} \colon \scH \to \scH\,,
	\qquad
	F \longmapsto R^{p/i} \big( \delta^* \circ F^{\ul{\Delta}_e} \big)\,,
\end{equation}
for the projective and the injective model structure, respectively.
We define the natural transformation
\begin{equation}
	\cc^{p/i}_{|F} \colon F \xrightarrow{\gamma_{|F}} \delta^* \circ F^{\ul{\Delta}_e} \xrightarrow[\sim]{r^{p/i}} \Cc^{p/i} F\,.
\end{equation}

\begin{proposition}
\label{st:fib rep via concordance}
The functors $\Cc^{p/i}$, together with the natural morphisms $\cc^{p/i}$ provide a functorial fibrant replacement in $\scH^{p/i\, I}$.
\end{proposition}

\begin{proof}
First, we show that $\Cc^{p/i} F$ is indeed fibrant in the $I$-local model structure $\scH^{p/i\, I}$, for every $F \in \scH$.
By construction, $\Cc^{p/i} F$ is a fibrant object in $\scH^{p/i}$.
It thus remains to show that it is $\RN$-local.
For any $c \in \Cart$, we have that
\begin{equation}
	\big( \delta^* \circ F^{\ul{\Delta}_e} \big)(c)
	= \delta^* \big( F(c \times \Delta_e) \big)
	= \delta^* \big( F^{\scY_c}(\Delta_e) \big)
	= S_e(F^{\scY_c})\,.
\end{equation}
By Lemma~\ref{st:exp pres smooth hos}, the smooth homotopy equivalence $\scY_c \to *$ induces a smooth homotopy equivalence $F \to F^{\scY_c}$, which is a weak equivalence in $\scH^{p/i\, I}$ by Corollary~\ref{st:smooth hoeqs are weqs}.
Since $r^{p/i}$ is a natural weak equivalence of functors valued in $\scH^{p/i}$, its component
\begin{equation}
	\delta^* \circ F^{\ul{\Delta}_e} \xrightarrow[\sim]{r^{p/i}} R^{p/i} \big( \delta^* \circ F^{\ul{\Delta}_e} \big)
\end{equation}
is an objectwise weak equivalence.
Consequently, for every $c \in \Cart$, we have a commutative square
\begin{equation}
\begin{tikzcd}[column sep=1.5cm, row sep=1cm]
	\big( \delta^* \circ F^{\ul{\Delta}_e} \big)(c) \ar[r, "r^{p/i}", "\sim"']
	& R^{p/i} \big( \delta^* \circ F^{\ul{\Delta}_e} \big)(c)
	\\
	\big( \delta^* \circ F^{\ul{\Delta}_e} \big)(*) \ar[r, "r^{p/i}"', "\sim"] \ar[u, "\sim"]
	& R^{p/i} \big( \delta^* \circ F^{\ul{\Delta}_e} \big)(*) \ar[u]
\end{tikzcd}
\end{equation}
whose vertical morphisms are induced from the collapse morphism $c \to *$.
It follows that $\Cc^{p/i}F$ is $\RN$-local and hence a fibrant object in $\scH^{p/i\, I}$.

Finally, we need to show that the morphism \smash{$\gamma_{|F} \colon F \longrightarrow \delta^* \circ F^{\ul{\Delta}_e}$}, induced by the collapse $\Delta_e \to *$, is a weak equivalence in $\scH^{p/i\, I}$.
Since the functor $S_e \colon \scH^{p/i\, I} \to \sSet$ preserves as well as reflects weak equivalences, that is equivalent to showing that the induced morphism
\begin{equation}
	S_e (\gamma_{|F}) \colon S_e F \longrightarrow S_e \big( \delta^* \circ F^{\ul{\Delta}_e} \big)
\end{equation}
is a weak equivalence of simplicial sets.
More explicitly, $S_e(\gamma_{|F})$ is the morphism
\begin{equation}
	\tilde{\delta}^* \big( F(\tilde{\Delta}_e) \big) \longrightarrow \tilde{\delta}^* \big( \delta^* F(\tilde{\Delta}_e \times \Delta_e) \big)
\end{equation}
induced by collapsing the extended simplices $\Delta_e$ without the tilde.
Note that we have only added the tilde in order to keep track of which diagonal functor refers to which copy of $\Delta_e$.
We have canonical isomorphisms
\begin{equation}
	S_e \big( \delta^* \circ F^{\ul{\Delta}_e} \big)
	\cong \tilde{\delta}^* \big( \delta^* F(\tilde{\Delta}_e \times \Delta_e) \big)
	\cong \delta^* \big( (\tilde{\delta}^* \circ F^{\tilde{\Delta}_e}) (\Delta_e) \big)
\end{equation}
and we know from the first part of this proof that the morphism
\begin{equation}
	\tilde{\delta}^* \big( F(\tilde{\Delta}_e) \big)
	= (\tilde{\delta}^* \circ F^{\tilde{\Delta}_e}) (*)
	\longrightarrow (\tilde{\delta}^* \circ F^{\tilde{\Delta}_e}) (\Delta_e^k)
\end{equation}
is a weak equivalence, for every $k \in \NN_0$.
This induces a (levelwise) weak equivalence of bisimplicial sets
\begin{equation}
	\Delta^0 \boxtimes \tilde{\delta}^* \big( F(\tilde{\Delta}_e) \big)
	\weq (\tilde{\delta}^* \circ F^{\tilde{\Delta}_e}) (\Delta_e)\,,
\end{equation}
which under $\delta^*$ maps to the morphism $S_e(\gamma_{|F})$.
Since the diagonal functor $\delta^* \colon s\sSet \to \sSet$ is homotopical, we obtain that $S_e(\gamma_{|F})$, and thus also $\gamma_{|F}$, is indeed a weak equivalence.
\end{proof}

\begin{corollary}
\label{st:mapping spaces via spl hom}
Mapping spaces in $\scH^{p/i\, I}$ can be computed (up to isomorphism in $\rmh \sSet$) as the simplicially enriched hom spaces
\begin{align}
\label{eq:derived sections of concordance sheaves}
	\Map_{\scH^{i I}}(F,G) &\simeq \ul{\scH} \big( F, R^i(\delta^* \circ G^{\ul{\Delta}_e}) \big)\,,
	\\
	\Map_{\scH^{p I}}(F,G) &\simeq \ul{\scH} \big( Q^p F, R^\sSet \circ \delta^* \circ G^{\ul{\Delta}_e} \big)\,,
\end{align}
where $Q^p$ is a cofibrant replacement functor for the projective model structure $\scH^p$.
\end{corollary}

\begin{definition}
\label{def:derived concordance spaces}
Given an object $F \in \scH$, we call $\Cc^{p/i} F$ its \emph{(projective/injective) derived concordance sheaf}.
For $G \in \scH$ we refer to the spaces in~\eqref{eq:derived sections of concordance sheaves} as the spaces of \emph{derived concordances of morphisms} from $F$ to $G$.
\end{definition}

We can apply these insights to describe the mapping spaces in the model categories $\scH^{p/i\, I}$.
In particular, given any $G \in \scH$, part~(3) of the following theorem shows that the derived sections of the derived concordance sheaf of $G$ (in the sense of Definition~\ref{def:derived concordance spaces}) on manifolds are represented by maps from the space underlying $M$ to the smooth singular complex of $G$.

\begin{theorem}
\label{st:map spaces via concordance}
There are natural isomorphisms in $\rmh \sSet$ as follows:
\begin{myenumerate}
\item For $F,G \in \scH$, we have
\begin{equation}
	\Map_\sSet(S_e F, S_e G) \cong \Map_{\scH^{p/i\, I}}(F,G)\,.
\end{equation}

\item For $F,G \in \scH$, we have
\begin{equation}
	\Map_\Top(|S_e F|, |S_e G|) \cong \Map_{\scH^{p/i\, I}}(F,G)\,.
\end{equation}

\item For any manifold $M$ and $G \in \scH$, we have
\begin{equation}
	\Map_\Top(M, |S_e G|) \cong \Map_{\scH^{p/i\, I}}(\ul{M},G)\,.
\end{equation}
\end{myenumerate}
\end{theorem}

\begin{proof}
It suffices to consider one of the model structures by Proposition~\ref{st:same weqs} and the fact that the mapping spaces in a model category depend only on the underlying relative category (see before Corollary~\ref{st:mapping spaces in p and i}).
For claim~(1), we consider the injective model structure.
Using that every object in $\scH^{i I}$ is cofibrant, we have the following isomorphisms in $\rmh \sSet$:
\begin{align}
	\Map_\sSet(S_e F, S_e G)
	&\cong \ul{\sSet}(S_e F, R^\sSet S_e G)
	\\
	&\cong \ul{\scH} \big( F, R_e R^\sSet S_e(G) \big)
	\\
	&\cong \ul{\scH} \big( F, R_e R^\sSet S_e \Cc^i (G) \big)
	\\
	&\cong \ul{\scH} \big(F, \Cc^i (G) \big)
	\\
	&\cong \Map_{\scH^{i I}} (F,G)\,.
\end{align}
The first isomorphism is merely the fact that $\sSet$ is a simplicial model category in which every object is cofibrant.
In the second isomorphism, we use the fact that $S_e \dashv R_e$ is a simplicial adjunction.
To see the third isomorphism, we use the weak equivalence \smash{$\cc^i_{|G} \colon G \weq \Cc^i(G)$} from Proposition~\ref{st:fib rep via concordance}.
Since both $S_e$ and $R^\sSet$ are homotopical functors, and since $R_e$ preserves fibrant objects, the morphism \smash{$R_e R^\sSet S_e (\cc^i_{|G})$} is a weak equivalence between fibrant objects in $\scH^{i I}$, which is thus preserved by $\ul{\scH}(F, -)$.
The fourth isomorphism stems from the fact that $S_e \dashv R_e$ is a Quillen equivalence, so that the canonical natural transformation $1_{\scH} \longrightarrow R_e R^\sSet S_e$ is a weak equivalence in $\scH^{i I}$ on every cofibrant object---that is, it is a weak equivalence on every object, since all objects in $\scH^{i I}$ are cofibrant.
Further, its component at the object $\Cc^i(G)$ is a weak equivalence between fibrant objects in $\scH^{i I}$, which is again preserved by $\ul{\scH}(F, -)$.
The final isomorphism directly follows from the insight that $\Cc^i$ is a fibrant replacement functor for $\scH^{i I}$ (Proposition~\ref{st:fib rep via concordance}).

Claim~(2) then follows from the fact that every object in $\Top$ is fibrant and that every object of $\sSet$ is cofibrant:
if $K,L \in \sSet$, then there are canonical isomorphisms in $\rmh \sSet$
\begin{align}
	\Map_\Top \big( |K|, |L| \big)
	&\cong \ul{\Top} \big( |K|, |L| \big)
	\\*
	&\cong \ul{\sSet} \big( K, \Sing |L| \big)
	\\
	&\cong \ul{\sSet} \big( K, \Sing |R^\sSet(L)| \big)
	\\
	&\cong \ul{\sSet} \big( K, R^\sSet(L) \big)
	\\*
	&\cong \Map_\sSet(K,L)\,.
\end{align}

Claim~(3) then follows from combining part~(2) with Theorem~\ref{st:|S_e M| sim DM}.
\end{proof}

We can now give a direct proof of the relation between model categories announced in Remark~\ref{rmk:HoThy induced by S_e}:
we characterise the homotopy theory induced on $\scH$ by the smooth singular complex functor $S_e$:

\begin{theorem}
\label{st:HoThy induced by S_e}
Let $W_\sSet$ denote the class of weak equivalences in $\sSet$, and let $S_e^{-1}(W_\sSet)$ denote the class of morphisms in $\scH$ whose image under $S_e$ is in $W_\sSet$.
There is an identity of model categories
\begin{equation}
	\scH^{p/i\, I} = L_{S_e^{-1}(W_\sSet)} \scH^{p/i}\,.
\end{equation}
\end{theorem}

\begin{proof}
We set
\begin{equation}
	\scM^{p/i} \coloneqq L_{S_e^{-1}(W_\sSet)} \scH^{p/i}\,.
\end{equation}
The model categories $\scH^{p/i\, I}$ and $\scM^{p/i}$ have the same cofibrations, since they are left Bousfield localisations of the same model category.
(Here we use either the projective or the injective model structure on \emph{both} sides.)
By Theorem~\ref{st:criterion for coincidence of MoStrs} it now suffices to show that they have the same fibrant objects.
Corollary~\ref{st:S_e pres and reflects weqs} implies that $I \subset S_e^{-1}(W_\sSet)$; thus, any fibrant object in $\scM^{p/i}$ is also fibrant in $\scH^{p/i\, I}$.

To see that any fibrant object of $\scH^{p/i\, I}$ is also fibrant in $\scM^{p/i}$, consider a fibrant object $G \in \scH^{p/i\, I}$ and a morphism $f \colon F_0 \to F_1$ in $S_e^{-1}(W_\sSet)$.
By Theorem~\ref{st:map spaces via concordance}, we have a commutative diagram
\begin{equation}
\begin{tikzcd}[column sep=1.5cm, row sep=1cm]
	\Map_{\scH^{p/i\, I}}(F_1,G) \ar[r, "f^*"] \ar[d, "\cong"']
	& \Map_{\scH^{p/i\, I}}(F_0,G) \ar[d, "\cong"]
	\\
	\Map_\sSet (S_e F_1, S_e G) \ar[r, "{(S_e f)^*}"']
	& \Map_\sSet (S_e F_0, S_e G)
\end{tikzcd}
\end{equation}
in $\rmh \sSet$.
By assumption on $f$, the morphism $S_e f$ is a weak equivalence in $\sSet$.
Hence, it induces a weak equivalence on mapping spaces; that is, the bottom morphism in the diagram is an isomorphism in $\rmh \sSet$.
From that, it follows that also the top morphism is an isomorphism in $\rmh \sSet$, which implies that $G$ is $S_e^{-1}(W_\sSet)$-local, and therefore fibrant in $\scM^{p/i}$.
\end{proof}

\begin{remark}
\label{rmk:Rel to BEBdBP}
Part~(3) of Theorem~\ref{st:map spaces via concordance} is related to recent results by Berwick-Evans, Boavida de Brito and Pavlov from~\cite{BEBdBP:Class_sp_of_oo-sheaves}:
in that paper, the authors work with the category $\widetilde{\scH}$ of simplicial presheaves on $\Mfd$.
Given $\widetilde{G} \in \widetilde{\scH}$, they consider the concordance sheaf
\begin{equation}
\label{eq:BEBdBP conc sh}
	\rmB \widetilde{G} = \delta^* \widetilde{G}^{\ul{\Delta}_e}\,,
\end{equation}
where the exponential is taken in $\widetilde{\scH}$.
The inclusion functor $\iota \colon \Cart \hookrightarrow \Mfd$ induces a homotopy right Kan extension $\hoRan_\iota' \colon \scH \to \widetilde{\scH}$ (see Appendix~\ref{app:H^(pl) simeq wtH^(pl)} for details).
Explicitly, working with projective model structures, we can write
\begin{equation}
	\hoRan'_\iota(F)(M)
	\cong \ul{\scH}(Q' M, F)
\end{equation}
for any $M \in \Mfd$, where $Q' \colon \scH^p \to \scH^p$ is the cofibrant replacement functor introduced in Appendix~\ref{app:Injective fib rep}.
The functor $\hoRan_\iota' \colon \scH^{p \ell} \to \widetilde{\scH}^{p \ell}$ is a right Quillen functor, where on the target side we consider the \v{C}ech localisation with respect to open coverings---this is a consequence of~\cite[Prop.~3.16, Thm.~3.18]{Bunk:Higher_Sheaves}.
It provides a concrete way of comparing sheaves on $\Cart$ to sheaves on $\Mfd$:
the functor
\begin{equation}
	\hoRan_\iota' \colon \scH^{p \ell} \to \widetilde{\scH}^{p \ell}
\end{equation}
is a right Quillen equivalence by Theorem~\ref{st:H^(pl) simeq wtH^(pl)}.

We can now compare the derived concordance sheaves from Definition~\ref{def:derived concordance spaces} to the concordance sheaf construction~\eqref{eq:BEBdBP conc sh} from~\cite{BEBdBP:Class_sp_of_oo-sheaves}:
let $F \in \scH$ be any simplicial presheaf on $\Cart$.
On the one hand, we have
\begin{equation}
	\hoRan_\iota'(\Cc^p F)(M)
	= \ul{\scH}(Q'M, \Cc^p F)\,,
\end{equation}
and on the other hand, we have
\begin{equation}
	\big( \rmB \hoRan_\iota'(F) \big) (M)
	= \delta^* \big( \ul{\scH}(Q'(M \times \Delta_e), F) \big)\,.
\end{equation}
Given that $\hoRan_\iota'(F)$ is a sheaf on $\Mfd$, it now follows from~\cite[Thm.~1.1, Thm.~1.2]{BEBdBP:Class_sp_of_oo-sheaves} that $\rmB \hoRan_\iota'(F)$ is a sheaf on $\Mfd$ as well.
Then, Proposition~\ref{st:iota^* reflects weqs between fibrants} implies that $\hoRan_\iota'(\Cc^p F)$ and $\rmB \hoRan_\iota'(F)$ are isomorphic in $\rmh \widetilde{\scH}^{p \ell}$ if and only if their images under the restriction functor $\iota^* \colon \widetilde{\scH}^{p \ell} \to \scH^{p \ell}$ are isomorphic in $\rmh \scH^{p \ell}$.

If $M = c$ is a cartesian space, we readily obtain a natural weak equivalence
\begin{equation}
	\hoRan_\iota'(\Cc^p F)(c)
	= \ul{\scH}(Q'c, \Cc^p F)
	\overset{\sim}{\longleftarrow} \Cc^p(F)(c)\,,
\end{equation}
and we further obtain natural weak equivalences
\begin{align}
	\big( \rmB \hoRan_\iota'(F) \big) (M)
	&= \delta^* \big( \ul{\scH}(Q'(c \times \Delta_e), F) \big)
	\\
	&\overset{\sim}{\longleftarrow} \delta^* \big( \ul{\scH}(c \times \Delta_e, F) \big)
	\\
	&\cong \delta^* F(c \times \Delta_e)
	\\
	&\weq \Cc^p(F)(c)\,.
\end{align}
The first weak equivalence uses that $\Delta_e^k$ is a cartesian space, for each $k \in \NN_0$, so that $c \times \Delta_e^k$ is representable in $\scH$, and the last morphism arises from postcomposing with a fibrant replacement functor in $\sSet$.
This establishes a natural zig-zag of weak equivalences between $\hoRan_\iota'(\Cc^p F)$ and $\rmB \hoRan_\iota'(F)$.
Note that the results \cite[Thm.~1.1, Thm.~1.2]{BEBdBP:Class_sp_of_oo-sheaves} enter crucially in the construction of this zig-zag because we use the fact that $\rmB \hoRan_\iota'(F)$ is a sheaf on $\Mfd$.
\qen
\end{remark}

\begin{remark}
We outline an alternative proof of Theorem~\ref{st:map spaces via concordance}, based on Proposition~\ref{st:fib rep via concordance} and Theorem~\ref{st:tc -| ev_* is QEq} (which we recall goes back to~\cite{Dugger:HoSheaves}).
Let $F, G \in \scH$, and let $Q \colon \scH^p \to \scH^p$ denote a cofibrant replacement functor.
In the homotopy category $\rmh \sSet$ of spaces, we have natural isomorphisms
\begin{align}
	\Map_{\scH^{p I}}(F,G)
	&\cong \ul{\scH}(Q F, \Cc^p G)
	\\*
	&\cong \ul{\scH}(Q F, \tilde{\sfc} R^\sSet S_e G)
	\\
	&\cong \ul{\sSet}(\underset{\Cart^\opp}{\colim}\, Q F, R^\sSet S_e G)
	\\
	&\cong \ul{\sSet} (\underset{\Cart^\opp}{\hocolim}\, F, R^\sSet S_e G)
	\\
	&\cong \ul{\sSet} (S_e F, R^\sSet S_e G)
	\\*
	&\cong \Map_\sSet (S_e F, S_e G)\,.
\end{align}
In the second isomorphism, we have used that $\Cc^p G$ is essentially constant and that $\Cc^p G(*) = R^\sSet \circ S_e G$.
The third isomorphism arises from the Quillen adjunction $\colim : \scH^p \rightleftarrows \sSet : \tilde{\sfc}$, and the fourth isomorphism stems from fact that $\colim \circ Q$ models the homotopy colimit.
Finally, the fifth isomorphism arises from our observation at the end of Section~\ref{sec:comparison of spaces from presheaves} that $S_e F$ is a model for the homotopy colimit of the diagram $F \colon \Cart^\opp \to \sSet$.
Parts~(2) and~(3) of Theorem~\ref{st:map spaces via concordance} then follow as in our proof above.
\qen
\end{remark}

\begin{appendix}

\section{An injective fibrant replacement of simplicial presheaves}
\label{app:Injective fib rep}

From the definition of the projective and the injective model structure on $\scH$, it follows directly that there is a Quillen equivalence
\begin{equation}
\begin{tikzcd}
	\scH^p \ar[r, shift left=0.14cm, "\perp"' yshift=0.05cm]
	& \scH^i\,. \ar[l, shift left=0.14cm]
\end{tikzcd}
\end{equation}
Both of the functors in this adjunction are the identity on $\scH$.
Here, we will construct a Quillen equivalence in the opposite direction, i.e.
\begin{equation}
\begin{tikzcd}
	Q' : \scH^i \ar[r, shift left=0.14cm, "\perp"' yshift=0.05cm]
	& \scH^p : R'\,. \ar[l, shift left=0.14cm]
\end{tikzcd}
\end{equation}
We start by defining the functor $Q'$.
Its construction is not specific to simplicial presheaves on $\Cart$, but works for simplicial presheaves over any small category.
Thus, let $\scC$ be a small category, let $\scK$ denote the category of simplicial presheaves on $\scC$, and let $\scY \colon \scC \to \scK$ denote the Yoneda embedding.
We denote the projective and the injective model structures on $\scK$ by $\scK^p$ and by $\scK^i$, respectively.
The conventional two-sided simplicial bar construction provides a functor~\cite[Sec.~4.2]{Riehl:Cat_HoThy}
\begin{equation}
	B_\bullet \big( (-), \scC, \scY \big) \colon \scK \longrightarrow (\scK)^{\bbDelta^\opp}\,,
	\qquad
	F \longmapsto  B_\bullet(F,\scC,\scY)\,.
\end{equation}

\begin{lemma}
\label{st:B_bullet and cofibs}
The functor $B_\bullet ((-), \scC, \scY)$ sends injective cofibrations in $\scK$ to injective cofibrations in $(\scK^p)^{\bbDelta^\opp}$.
That is, if $f \colon F \to G$ is an objectwise cofibration of simplicial presheaves on $\scC$, then $B_n (f, \scC, \scY)$ is a projective cofibration of simplicial presheaves, for each $n \in \NN_0$.
\end{lemma}

\begin{proof}
We have that
\begin{equation}
	B_n(F,\scC, \scY)
	= \coprod_{c_0, \ldots, c_n \in \scC} F(c_n) \otimes \scC(c_{n-1}, c_n) \otimes \cdots \otimes \scC(c_0,c_1) \otimes \scY_{c_0}\,.
\end{equation}
Observing that $\scY_c \in \scK^p$ is cofibrant for every $c \in \scC$ and recalling that $\scK^p$ is a simplicial model category, we see that, in each part of the coproduct, the functor
\begin{equation}
	(-) \otimes \scC(c_{n-1}, c_n) \otimes \cdots \otimes \scC(c_0,c_1) \otimes \scY_{c_0}
	\colon \sSet \to \scK^p
\end{equation}
is a left Quillen functor.
\end{proof}

In order to obtain Dugger's cofibrant replacement functor $Q$, one needs to take the diagonal (objectwise) of the bisimplicial presheaf $B_\bullet((-),\scC, \scY)$:
\begin{equation}
\label{eq:Duggers Q}
	(QF)_n = \big( B_n(F,\scC, \scY) \big)_n
	= \coprod_{c_0, \ldots, c_n \in \scC} F_n(c_n) \otimes \scC(c_{n-1}, c_n) \otimes \cdots \otimes \scC(c_0,c_1) \otimes \scY_{c_0}\,.
\end{equation}
It is not evident to us whether taking the diagonal of a bisimplicial presheaf maps injective cofibrations in $(\scK^p){}^{\bbDelta^\opp}$ to projective cofibrations in $\scK$.
That is the motivation for constructing the functor $Q'$ below.

Let $\scI$ be a small category, $\scV$ a symmetric monoidal category, and $\scM$ a model category category enriched, tensored and cotensored over $\scV$ (i.e.~a model $\scV$-category in the terminology of~\cite{Barwick:Enriched_B-Loc}).
Following the notation in~\cite[Sec.~4.1]{Riehl:Cat_HoThy}, we let
\begin{equation}
	(-) \underset{\scI}{\otimes} (-) \colon \scV^{\scI^\opp} \times \scM^\scI \longrightarrow \scM
	\qqandqq
	\{-,-\}^\scI \colon \scV^\scI \times \scM^\scI \longrightarrow \scM
\end{equation}
denote the functor tensor product and the functor hom, respectively.
Let $\Cat_s$ denote the (1-)category of small categories, and let $\bbDelta_{/(-)} \colon \bbDelta \to \Cat_s$ denote the functor that sends $[k] \in \bbDelta$ to the slice category $\bbDelta_{/[k]}$, and let $N$ denote the nerve functor.
We now define the functor
\begin{equation}
\label{eq:def Q'}
	Q' \colon \scK \longrightarrow \scK\,,
	\qquad
	Q' \coloneqq N(\bbDelta_{/(-)})^\opp \underset{\bbDelta^\opp}{\otimes} B_\bullet \big( (-), \scC, \scY \big)\,.
\end{equation}

\begin{lemma}
\label{st:Q' and cofibs}
The functor $Q' \colon \scK^i \to \scK^p$ preserves cofibrations.
\end{lemma}

\begin{proof}
We view the functor tensor product in the definition of $Q'$ as
\begin{equation}
	(-) \underset{\bbDelta^\opp}{\otimes} (-)
	\colon (\sSet)^\bbDelta \times (\scK^p)^{\bbDelta^\opp} \longrightarrow \scK^p\,.
\end{equation}
The functor tensor product is a left Quillen bifunctor when endowing the two source categories with any of the pairs of model structures (projective, injective), (Reedy, Reedy), or (injective, projective)~\cite[Thms.~11.5.9, 14.3.1]{Riehl:Cat_HoThy}.
Further, the functor $N(\bbDelta_{/(-)})^\opp \colon \bbDelta \longrightarrow \sSet$ is projectively cofibrant~\cite[Prop.~14.8.8]{Hirschhorn:MoCats}.
(Note that it is then also Reedy cofibrant.)
Consequently, the functor
\begin{equation}
	N(\bbDelta_{/(-)})^\opp \underset{\bbDelta^\opp}{\otimes} (-)
	\colon (\scK^p)^{\bbDelta^\opp} \longrightarrow \scK^p
\end{equation}
is left Quillen with respect to the injective model structure on $(\scK^p)^{\bbDelta^\opp}$ (and hence also with respect to the Reedy model structure).
We can write $Q'$ as the composition
\begin{equation}
\begin{tikzcd}[column sep=3cm]
	\scK^i \ar[r, "{B_\bullet((-), \scC, \scY)}"]
	& \big( (\scK^p)^{\bbDelta^\opp} \big)_{inj} \ar[r, "{N(\bbDelta_{/(-)})^\opp \otimes_{\bbDelta^\opp} (-)}"]
	& \scK^p\,.
\end{tikzcd}
\end{equation}
The model category of simplicial diagrams in the middle carries the injective model structure.
We have shown in Lemma~\ref{st:B_bullet and cofibs} that the first functor preserves cofibrations, and it follows from our arguments above that the second functor preserves cofibrations as well.
\end{proof}

Next, we are going to employ the Bousfield-Kan map to show that $Q'$ provides a cofibrant replacement functor on $\scK^p$.
The Bousfield-Kan map is a morphism
\begin{equation}
	bk \colon N(\bbDelta_{/(-)})^\opp \longrightarrow \Delta^\bullet
\end{equation}
of cosimplicial simplicial sets.
We will use the following two statements:

\begin{proposition}
\label{st:BK map is R-weq}
\emph{\cite[Prop.~18.7.2]{Hirschhorn:MoCats}}
The Bousfield-Kan map is a Reedy weak equivalence between Reedy cofibrant cosimplicial simplicial sets.
\end{proposition}

\begin{proposition}
\label{st:B_bullet is R-cof}
For any $F \in \scK$, the simplicial object in $\scK$ given as $B_\bullet(F,\scC, \scY)$ is Reedy cofibrant.
\end{proposition}

\begin{proof}
This follows directly from~\cite[Rmk.~5.2.2]{Riehl:Cat_HoThy}, applied to the functor $\scY \colon \scC \to \scK^p$.
\end{proof}

We now consider the induced natural transformation
\begin{equation}
\label{eq:BK otimes B_bullet}
	bk \underset{\bbDelta^\opp}{\otimes} B_\bullet \big( (-), \scC, \scY \big)
	\colon Q' \longrightarrow \Delta^\bullet \underset{\bbDelta^\opp}{\otimes} B_\bullet \big( (-), \scC, \scY \big)
	= B \big( (-), \scC, \scY \big)\,.
\end{equation}
By Proposition~\ref{st:B_bullet is R-cof}, the functor
\begin{equation}
	(-) \underset{\bbDelta^\opp}{\otimes} B_\bullet \big( F, \scC, \scY \big)
	\colon \big( (\sSet)^\bbDelta \big)^{Reedy} \longrightarrow \scK^p
\end{equation}
is a left Quillen functor for every $F \in \scK$.
It then follows from Proposition~\ref{st:BK map is R-weq} and the arguments in the proof of Proposition~\ref{st:Q' and cofibs} that the natural transformation~\eqref{eq:BK otimes B_bullet} is a natural weak equivalence of functors $\scK \to \scK^p$.
Finally, we use that the functor $B((-),\scC,\scY)$ agrees with Dugger's cofibrant replacement functor $Q^p$ for $\scK^p$ from~\cite{Dugger:Universal_HoThys}.
In particular, it comes with a natural weak equivalence $q^p \colon Q^p \to 1_{\scK}$.
Composing $q^p$ with the morphism~\eqref{eq:BK otimes B_bullet}, we obtain a natural weak equivalence $q' \colon Q' \to 1_{\scK}$.
Putting everything together, we have proven

\begin{proposition}
\label{st:Q' is cof rep fctr}
The functor $Q'$ from~\eqref{eq:def Q'}, together with the natural weak equivalence $q'$, provides a cofibrant replacement functor for $\scK^p$.
In particular, $Q'$ preserves objectwise weak equivalences.
\end{proposition}

Finally, we observe that $Q'$ has a right adjoint, which is explicitly given by
\begin{equation}
\label{eq:def R'}
	R' \colon \scK \longrightarrow \scK\,,
	\qquad
	R'(G)
	= \int^n \prod_{c_0, \ldots, c_n} G(c_0)^{N(\bbDelta_{/[n]})^\opp \otimes \scC(c_0, c_1) \otimes \cdots \otimes \scC(c_{n-1}, c_n) \otimes \scY^{c_n}}
\end{equation}
where $\scY^{(-)} \colon \scC^\opp \to \Cat(\scC, \sSet)$ denotes the co-Yoneda embedding of $\scC$.

\begin{theorem}
\label{st:Q' -| R'}
The functors $Q'$ and $R'$ satisfy the following properties:
\begin{myenumerate}
\item The adjunction $Q' \dashv R'$ is a Quillen equivalence
\begin{equation}
\begin{tikzcd}
	Q' : \scH^i \ar[r, shift left=0.14cm, "\perp"' yshift=0.05cm]
	& \scH^p : R'\,. \ar[l, shift left=0.14cm]
\end{tikzcd}
\end{equation}

\item There is a natural transformation $r' \colon 1_{\scK} \to R'$ such that $r'_{|G} \colon G \to R'G$ is a weak equivalence in $\scK^i$ for every projectively fibrant $G \in \scK$.

\item Let $R^\sSet$ be a fibrant replacement functor for simplicial sets.
Then, $G \mapsto R' (R^\sSet \circ G)$ is a fibrant replacement functor on $\scK^i$.
\end{myenumerate}
\end{theorem}

\begin{proof}
Ad~(1):
Proposition~\ref{st:Q' is cof rep fctr}, together with the observation that $Q'$ is a left adjoint, readily implies that $Q'$ is a left Quillen functor.
The fact that this is a Quillen equivalence follows from the existence of the natural weak equivalence $q' \colon Q' \weq 1_{\scK}$.
Formally, this implies that the composition of $Q'$ by the left Quillen equivalence $1_{\scK} \colon \scK^p \to \scK^i$ is weakly equivalent to the identity functor on $\scK$.
Thus, the statement follows from the two-out-of-three property of Quillen equivalences~\cite{Hovey:MoCats}.

Ad~(2):
Let
\begin{equation}
	\tau_{F,G} \colon \scK(Q'F,G) \longrightarrow \scK(F, R'G)
\end{equation}
be the natural isomorphism that establishes the adjunction $Q' \dashv R'$.
We define $r' \colon 1_{\scK} \to R'$ to be the image under $\tau$ of the natural transformation $q'$.
Now consider the weak equivalence \smash{$q'_{|G} \colon Q'G \to G$}, for $G \in \scK^p$ fibrant.
By part~(1) and since each object in $\scK^i$ is cofibrant, the morphism \smash{$r'_{|G} \colon G \to R'G$} is a weak equivalence. 

Ad~(3):
Let $R^\sSet$ be a fibrant replacement functor in $\sSet$, with associated natural weak equivalence $r^\sSet \colon 1_\sSet \weq R^\sSet$.
Let $G \in \scK$ be arbitrary and consider the composition
\begin{equation}
\begin{tikzcd}[column sep=2cm]
	G \ar[r, "r^\sSet \circ 1_G"]
	& R^\sSet \circ G \ar[r, "r'{|R^\sSet G}"]
	& R'(R^\sSet \circ G)\,.
\end{tikzcd}
\end{equation}
The first morphism is an objectwise weak equivalence by definition.
Since $R^\sSet \circ G$ is projectively fibrant, the second morphism is a weak equivalence as well by part~(2).
\end{proof}

Finally, we recall that the mapping spaces in a model category depend only on the weak equivalences, i.e.~their homotopy types are determined completely by the underlying \textit{relative} category.
This was famously proven in~\cite[Prop.~4.4]{DK:Function_complexes}; the simpler case of simplicial model categories---which applies directly to the present situation---is in~\cite[Cor.~4.7]{DK:Function_complexes}.
In particular, we record the following direct consequence for reference:

\begin{proposition}
\label{st:mapping spaces in p and i}
Let $F,G$ be any two objects in $\scK$.
There is a canonical isomorphism in the homotopy category $\rmh \sSet$ of spaces between the mapping spaces of the projective and the injective model structures
\begin{equation}
	\Map_{\scK^p}(F,G) \cong \Map_{\scK^i}(F,G)\,.
\end{equation}
\end{proposition}

\section{Sheaves on manifolds and sheaves on cartesian spaces}
\label{app:H^(pl) simeq wtH^(pl)}

This appendix is devoted to the comparison of two localisations:
on the one hand, we consider the localisation $\scH^{p \ell}$ of the projective model structure $\scH^p$ of simplicial presheaves on cartesian spaces at the \v{C}ech nerves of differentiably good open coverings (see Section~\ref{sec:R-local MoStrs}).
On the other hand, we have the localisation $\widetilde{\scH}^{p \ell}$ of the projective model structure $\widetilde{\scH}^p$ of simplicial presheaves on manifolds at the \v{C}ech nerves of open coverings.
(The same arguments also apply to the localisation of $\widetilde{\scH}^p$ at the \v{C}ech nerves of surjective submersions.)

We will compare these localisations by means of the functors
\begin{align}
	&\hoRan_\iota \colon \scH \longrightarrow \widetilde{\scH}\,,
	\qquad
	F \mapsto \ul{\scH} \big( Q(-), F \big)\,,
	\\
	&\hoRan'_\iota \colon \scH \longrightarrow \widetilde{\scH}\,,
	\qquad
	F \mapsto \ul{\scH} \big( Q'(-), F \big)\,,
\end{align}
where $\iota \colon \Cart \hookrightarrow \Mfd$ is the canonical inclusion, and where $Q \colon \scH^p \to \scH^p$ is Dugger's cofibrant replacement functor for the projective model structure on simplicial presheaves~\eqref{eq:Duggers Q}.
Further, $Q' \dashv R'$ is the Quillen equivalence from Theorem~\ref{st:Q' -| R'}.
The natural transformation~\eqref{eq:BK otimes B_bullet} induces a natural transformation $\eta \colon \hoRan_\iota \longrightarrow \hoRan_\iota'$ whose component $\eta_{|F}$ on every fibrant object $F \in \scH^p$ is a projective weak equivalence.

The functor $\hoRan_\iota \colon \scH^p \to \widetilde{\scH}^p$ is a right Quillen functor with left adjoint $Q \circ \iota^*$, and this Quillen adjunction descends to a Quillen adjunction on \v{C}ech localisations,
\begin{equation}
\begin{tikzcd}
	Q \circ \iota^* : \widetilde{\scH}^{p \ell} \ar[r, shift left=0.14cm, "\perp"' yshift=0.05cm]
	& \scH^{p \ell} : \hoRan_\iota\,, \ar[l, shift left=0.14cm]
\end{tikzcd}
\end{equation}
by the universal property of left Bousfield localisations.
Analogously, $\hoRan_\iota' \colon \scH^p \to \widetilde{\scH}^p$ is a right Quillen functor (since $Q'$ is homotopical and valued in projectively cofibrant simplicial presheaves).
Further, the natural transformation $\eta \colon \hoRan_\iota \longrightarrow \hoRan_\iota'$ establishes that $\hoRan_\iota'$ maps local objects in $\scH^{p \ell}$ to local objects in $\widetilde{\scH}^{p \ell}$, and hence that we also have a Quillen adjunction
\begin{equation}
\begin{tikzcd}
	Q' \circ \iota^* : \widetilde{\scH}^{p \ell} \ar[r, shift left=0.14cm, "\perp"' yshift=0.05cm]
	& \scH^{p \ell} : \hoRan_\iota'\,. \ar[l, shift left=0.14cm]
\end{tikzcd}
\end{equation}
Note that we can also write the right adjoint as
\begin{equation}
	\hoRan_\iota' \cong \iota_* \circ R'\,,
\end{equation}
with $R' \colon \scH \to \scH$ defined as in~\eqref{eq:def R'}, and where $\iota_*$ is the right Kan extension along $\iota$.

\begin{lemma}
\label{st:counit of RhoRan'}
The derived counit of the Quillen adjunction $Q' \circ \iota^* \dashv \hoRan_\iota'$ is a projective weak equivalence on every fibrant $F \in \scH^p$.
\end{lemma}

\begin{proof}
Let $\widetilde{Q} \colon \widetilde{\scH}^p \to \widetilde{\scH}^p$ be a cofibrant replacement functor with associated natural weak equivalence $\widetilde{q} \colon \widetilde{Q} \weq 1_{\widetilde{\scH}}$.
Consider a fibrant object $F \in \scH^p$ and the composition
\begin{equation}
\begin{tikzcd}[column sep=1cm, row sep=1.25cm]
	(Q' \circ \iota^*) \circ \widetilde{Q} \circ (\iota_* \circ R')(F)
	\ar[r]
	&
	(Q' \circ \iota^*) \circ (\iota_* \circ R')(F)
	\ar[r, "\cong"]
	& (Q' \circ R')(F) \ar[r]
	& F\,,
\end{tikzcd}
\end{equation}
which is the derived counit.
The first morphism is given by \smash{$(Q' \circ \iota^*)\widetilde{q}_{|\iota_* \circ R'(F)}$}.
It is a weak equivalence in $\scH^p$ because \smash{$\widetilde{q}_{|\iota_* \circ R'(F)}$} is a projective weak equivalence and $\iota^* \colon \widetilde{\scH}^p \to \scH^p$ is homotopical, as is $Q'$.
The second morphism is the counit of the adjunction $\iota^* \dashv \iota_*$, which is an isomorphism by the Yoneda Lemma.
The third morphism is a weak equivalence by Theorem~\ref{st:Q' -| R'} and the fact that every object in $\scH^i$ is cofibrant.
\end{proof}

It follows that the total derived functor $\RN \hoRan_\iota' \colon \rmh \scH^p \to \rmh \widetilde{\scH}^p$ is fully faithful, and hence that also $\RN \hoRan_\iota' \colon \rmh \scH^{p \ell} \to \rmh \widetilde{\scH}^{p \ell}$ is fully faithful.

Next, we show that $\RN \hoRan_\iota'$ is essentially surjective.
We start by recalling that, for any open covering $\scU = \{U_a\}_{a \in A}$ of a manifold $M$, the induced \v{C}ech nerve $\cC \scU \to \ul{M}$ is a weak equivalence in both $\scH^{p \ell}$ and $\widetilde{\scH}^{p \ell}$ (this follows directly from~\cite[Cor.~A.3]{DHI:Hypercovers_and_sPShs}).
If $\scU$ is even a differentiably good open covering of $M$ (i.e.~each finite intersection of patches is either empty or a cartesian space), then $\cC \scU$ is levelwise a coproduct of representables in $\scH$.
Hence, it is cofibrant in $\scH^{p \ell}$.

Suppose that $G \in \widetilde{\scH}^{p \ell}$ is fibrant.
Given a manifold $M$ with a differentiably good open covering $\scU$, we have weak equivalences
\begin{equation}
\label{eq:Sh(Mfd) via Cart I}
	G(M) \cong \ul{\widetilde{\scH}}(\ul{M}, G)
	\weq \ul{\widetilde{\scH}}(\cC \scU, G)
	\cong \ul{\scH}(\cC \scU, \iota^*G)\,.
\end{equation}
In this sense, the value of a sheaf $G$ on manifolds is determined by its values on cartesian spaces.
However, the above weak equivalence is not functorial in $M$, so we need to improve on~\eqref{eq:Sh(Mfd) via Cart I}.

For $M \in \Mfd$, let $\Cov(M)$ denote the category whose objects are open covering $\scU = \{U_a\}_{a \in A}$ of $M$.
Given another open covering $\scV$ of $M$ there is a unique morphism $\scV \to \scU$ precisely if $\scV$ refines $\scU$.
Let $\GCov(M)$ denote the full subcategory of $\Cov(M)$ on the differentiably good open coverings.
Note that $\Cov(M)$ is cofiltered since every pair $(\scU,\scV)$ of open coverings has a common refinement $\scU \times_M \scV$.
Moreover, any open covering of $M$ has a differentiably good refinement~\cite[App.~A]{FSS:Cech_diff_char_classes_via_L_infty}, which implies that $\GCov(M)$ is cofiltered and that the canonical inclusion $j_M \colon \GCov(M) \hookrightarrow \Cov(M)$ is homotopy final, since for any $\scU \in \Cov(M)$ the slice category $j/\scU$ is cofiltered as well.

Given $G \in \widetilde{\scH}$, we consider the diagrams in simplicial sets
\begin{equation}
	\Cov(M)^\opp \longrightarrow \sSet\,,
	\qquad \scU \longmapsto \ul{\widetilde{\scH}} (\cC \scU, G)
	\qandq
	\scU \longmapsto G(M)\,.
\end{equation}
Since $\Cov(M)^\opp$ is filtered and $\sSet$ is combinatorial, the ordinary colimits of these diagrams model their homotopy colimits~\cite[Prop.~7.3]{Dugger:Comb_MoCats_Pres}, so that we obtain a morphism
\begin{equation}
\label{eq:plus construction for Mfd}
	G(M) \longrightarrow
	\underset{\scU \in \Cov(M)^\opp}{\colim^\sSet} \Big( \ul{\widetilde{\scH}} (\cC \scU, G) \Big)
	\simeq \underset{\scU \in \Cov(M)^\opp}{\hocolim^\sSet} \Big( \ul{\widetilde{\scH}} (\cC \scU, G) \Big)\,.
\end{equation}
(The right-hand side of~\eqref{eq:plus construction for Mfd} is Grothendieck's Plus Construction, applied to $M$ and $G$.)
If $G \in \widetilde{\scH}^{p \ell}$ is fibrant, then the morphism in~\eqref{eq:plus construction for Mfd} is a weak equivalence, for each manifold $M \in \Mfd$.
Since $j_M^\opp \colon \GCov(M)^\opp \hookrightarrow \Cov(M)^\opp$ is cofinal, we further obtain a canonical isomorphism
\begin{equation}
	\underset{\scU \in \Cov(M)^\opp}{\colim^\sSet} \Big( \ul{\widetilde{\scH}} (\cC \scU, G) \Big)
	\cong \underset{\scU \in \GCov(M)^\opp}{\colim^\sSet} \Big( \ul{\widetilde{\scH}} (\cC \scU, G) \Big)\,.
\end{equation}
If $\scU \in \GCov(M)$, then $\cC\scU$ is levelwise a coproduct of presheaves represented by cartesian spaces.
Hence, there is a further canonical isomorphism
\begin{equation}
\label{eq:(G)Cov and Cart}
	\underset{\scU \in \GCov(M)^\opp}{\colim^\sSet} \Big( \ul{\widetilde{\scH}} (\cC \scU, G) \Big)
	\cong \underset{\scU \in \GCov(M)^\opp}{\colim^\sSet} \Big( \ul{\scH} (\cC \scU, \iota^*G) \Big)\,,
\end{equation}
where $\iota \colon \Cart \hookrightarrow \Mfd$ is the inclusion.
Thus, for any $G \in \widetilde{\scH}$ and any $M \in \Mfd$ we obtain a morphism
\begin{equation}
	\gamma_{G,M} \colon G(M) \longrightarrow \underset{\scU \in \GCov(M)^\opp}{\colim^\sSet} \Big( \ul{\scH} (\cC \scU, \iota^*G) \Big)
\end{equation}
which is natural in $G$ and which is a weak equivalence whenever $G \in \widetilde{\scH}^{p \ell}$ is fibrant.

Next we show that $\gamma_{G,M}$ is also natural in $M$.
Let $f \colon M \to N$ be a smooth map.
This induces a functor $f^* \colon \Cov(N) \to \Cov(M)$, acting via $f^*(\{U_a\}_{a \in A}) = \{f^{-1}(U_a)\}_{a \in A}$.
This establishes the assignment $M \mapsto \Cov(M)$ as a (strict) functor $\Cov \colon \Mfd^\opp \to \Cat_s$ to the (large) category of small categories.
Given an open covering $\scU$ of $N$, the map $f$ induces a canonical morphism $\widehat{f} \colon \cC(f^*\scU) \to \cC \scU$ in $\widetilde{\scH}$.
This induces a natural transformation
\begin{equation}
\begin{tikzcd}[column sep=0.5cm]
	{\Cov(N)^\opp} \ar[rr, "f^*"] \ar[rd, "{\widetilde{\scH}(\cC(-), G)}"', ""{name=V, inner sep=0pt, above, pos=.5}]
	& & |[alias=U]| {\Cov(M)^\opp} \ar[ld, "{\widetilde{\scH}(\cC(-), G)}"]
	\\
	& \sSet &
	\arrow[Rightarrow, from=V, to=U, shorten <= 3, shorten >= 3, "\varphi" description]
\end{tikzcd}
\end{equation}
Further, recall that for any composition of functors $\scI \overset{F}{\longrightarrow} \scJ \overset{D}{\longrightarrow} \scC$ there is a canonical morphism $\colim(D \circ F) \to \colim(D)$.
Thus, we have a canonical morphism
\begin{equation}
	\underset{\scU \in \Cov(N)^\opp}{\colim^\sSet} \Big( \ul{\widetilde{\scH}} \big( \cC (f^*\scU), G \big) \Big)
	\longrightarrow \underset{\scU \in \Cov(M)^\opp}{\colim^\sSet} \Big( \ul{\widetilde{\scH}} (\cC \scU, G) \Big)\,.
\end{equation}
Combining this morphism with the natural transformation $\varphi$ yields a morphism
\begin{equation}
	\underset{\scU \in \Cov(N)^\opp}{\colim^\sSet} \Big( \ul{\widetilde{\scH}} \big( \cC \scU, G \big) \Big)
	\longrightarrow
	\underset{\scV \in \Cov(M)^\opp}{\colim^\sSet} \Big( \ul{\widetilde{\scH}} (\cC \scV, G) \Big)
\end{equation}
in $\sSet$ which is compatible with composition of smooth maps of manifolds.
Finally, we use the isomorphisms~\eqref{eq:(G)Cov and Cart} to define the top morphism in the diagram
\begin{equation}
\begin{tikzcd}[column sep=1.5cm, row sep=1cm]
	\underset{\scU \in \GCov(N)^\opp}{\colim^\sSet} \Big( \ul{\scH} \big( \cC \scU, \iota^*G \big) \Big)
	\ar[r, dashed] \ar[d, "\cong"]
	&\underset{\scV \in \GCov(M)^\opp}{\colim^\sSet} \Big( \ul{\scH} (\cC \scV, \iota^*G) \Big)
	\ar[d, "\cong"]
	\\
	\underset{\scU \in \Cov(N)^\opp}{\colim^\sSet} \Big( \ul{\widetilde{\scH}} \big( \cC \scU, G \big) \Big)
	\ar[r]
	& \underset{\scV \in \Cov(M)^\opp}{\colim^\sSet} \Big( \ul{\widetilde{\scH}} (\cC \scV, G) \Big)
\end{tikzcd}
\end{equation}
It follows that the assignment
\begin{equation}
	M \longmapsto \underset{\scV \in \GCov(M)^\opp}{\colim^\sSet} \Big( \ul{\scH} (\cC \scV, \iota^*G) \Big)
\end{equation}
defines a functor $\Mfd^\opp \to \sSet$.
That is, we obtain a functor
\begin{equation}
\label{eq:def Pl}
	\Pl \colon \widetilde{\scH} \to \widetilde{\scH}\,,
	\qquad
	\Pl G(M) = \underset{\scV \in \GCov(M)^\opp}{\colim^\sSet} \Big( \ul{\scH} (\cC \scV, \iota^*G) \Big)\,,
\end{equation}
coming with a natural transformation \smash{$\gamma \colon 1_{\widetilde{\scH}} \to \Pl$} which is an objectwise weak equivalence on every fibrant $G \in \widetilde{\scH}^{p \ell}$.
In particular, we have shown:

\begin{proposition}
\label{st:iota^* reflects weqs between fibrants}
Let $G,G' \in \widetilde{\scH}^{p \ell}$ be fibrant.
A morphism $g \colon G \to G'$ is a weak equivalence in $\widetilde{\scH}^{p \ell}$ if and only if $\iota^*g \colon \iota^*G \to \iota^*G'$ is a weak equivalence in $\scH^{p \ell}$.
\end{proposition}

Note that for $G,G' \in \widetilde{\scH}^{p \ell}$ fibrant, a morphism $g \colon G \to G'$ is a weak equivalence in $\widetilde{\scH}^{p \ell}$ precisely if it is an objectwise equivalence (since both $G$ and $G'$ are local objects), and analogously $\iota^*g \colon \iota^*G \to \iota^*G'$ is a weak equivalence if and only if it is an objectwise equivalence (since $\iota^* \colon \widetilde{\scH}^{p \ell} \to \scH^{p \ell}$ preserves fibrant objects).

\begin{lemma}
\label{st:R hoRan ff}
The functor $\RN \hoRan_\iota' \colon \rmh \scH^{p \ell} \longrightarrow \rmh \widetilde{\scH}^{p \ell}$ is essentially surjective.
\end{lemma}

\begin{proof}
Let $G \in \widetilde{\scH}^{p \ell}$ be fibrant.
We show that there is a zig-zag of weak equivalences in $\widetilde{\scH}^{p \ell}$ linking $G$ to $\hoRan_\iota' \circ \iota^*(G)$.
First, since $G$ and $\hoRan_\iota' \circ \iota^*(G)$ are both sheaves on $\Mfd$, there are local weak equivalences
\begin{equation}
	G \weq \Pl G
	\qqandqq
	\hoRan_\iota' \circ \iota^*(G) \weq \Pl \big( \hoRan_\iota' \circ \iota^*(G) \big)\,,
\end{equation}
where $\Pl \colon \widetilde{\scH}^{p \ell} \to \widetilde{\scH}^{p \ell}$ is the functor from~\eqref{eq:def Pl}.
Further, for $c \in \Cart$ we have that
\begin{equation}
	\iota^* \big( \hoRan_\iota' \circ \iota^*(G) \big)(c)
	= \ul{\scH} \big( Q'(c), \iota^*(G) \big)\,,
\end{equation}
so the natural weak equivalence $Q' \weq 1_{\scH}$ induces a projective weak equivalence
\begin{equation}
	\iota^*(G) \weq \iota^* \big( \hoRan_\iota' \circ \iota^*(G) \big)\,.
\end{equation}
For any manifold $M$ and any differentiably good open covering $\scU$ of $M$, this induces a weak equivalence
\begin{equation}
	\ul{\scH}(\cC \scU, \iota^*G) \weq \ul{\scH} \big( \cC \scU, \iota^*(\hoRan_\iota' \circ \iota^*(G)) \big)\,.
\end{equation}
Since $\GCov(M)^\opp$ is filtered, this yields a weak equivalence $\Pl G \weq \Pl (\hoRan_\iota' \circ \iota^*(G))$, thus establishing the desired zig-zag of weak equivalences.
\end{proof}

By Lemmas~\ref{st:counit of RhoRan'} and~\ref{st:R hoRan ff}, the total derived functor $\RN \hoRan_\iota' \colon \rmh \scH^{p \ell} \to \rmh \widetilde{\scH}^{p \ell}$ is an equivalence.

\begin{theorem}
\label{st:H^(pl) simeq wtH^(pl)}
The following adjunction is a Quillen equivalence:
\begin{equation}
\begin{tikzcd}
	Q' \circ \iota^* : \widetilde{\scH}^{p \ell} \ar[r, shift left=0.14cm, "\perp"' yshift=0.05cm]
	& \scH^{p \ell} : \hoRan_\iota'\,. \ar[l, shift left=0.14cm]
\end{tikzcd}
\end{equation}
\end{theorem}

\end{appendix}

\vspace{0cm}
\begin{small}

\makeatletter

\interlinepenalty=10000

\makeatother

\bibliographystyle{alphaurl}
\addcontentsline{toc}{section}{References}
\setlength{\bibsep}{0.0pt}
\bibliography{R-Loc_HoThy_Bib}

\vspace{0.2cm}

\noindent
Mathematical Institute, The University of Oxford, Andrew Wiles Building.
\\
severin.bunk@maths.ox.ac.uk

\end{small}

\end{document}